\documentclass[11pt]{amsart}
\usepackage[a4paper,margin=2.5cm]{geometry}

\usepackage{amssymb}
\usepackage{enumerate, xspace}
\usepackage{enumitem}
\usepackage{mathrsfs}
\usepackage[normalem]{ulem}
\usepackage{bbm}       %allows \mathds{E 1}, without sans is less heavy
\usepackage[backgroundcolor=blue!20!white, linecolor=blue!20!white,
textsize=footnotesize]{todonotes}
\usepackage{verbatim}
\usepackage[colorlinks=true,linkcolor=blue,citecolor=magenta]{hyperref}
\numberwithin{equation}{section}
\usepackage{stmaryrd}

\usepackage{enumitem}
\usepackage{hyperref}

\usepackage{crossreftools}

\makeatletter
\newcommand{\optionaldesc}[2]{%
  \phantomsection
  #1\protected@edef\@currentlabel{#1}\label{#2}%
}
\makeatother
\usepackage{xurl} % optional but recommended for line breaks in URLs

\usepackage{hyperref} % required to make hyperlinks clickable

\makeindex
\newtheorem{theorem}{Theorem}[section]
\newtheorem{lemma}[theorem]{Lemma}

\newtheorem{proposition}[theorem]{Proposition}
\newtheorem{corollary}[theorem]{Corollary}

\newtheorem{definition}[theorem]{Definition}
\newtheorem{remark}[theorem]{Remark}

\newcommand{\mylabel}[2]{#2\def\@currentlabel{#2}\label{#1}}

\newcommand\ve{\varepsilon}

\newcommand\vf{\varphi}

\newcommand\fR{\mathfrak R}

\newcommand\bS {\mathbb S}

\renewcommand{\ge}{\geqslant}
\renewcommand{\le}{\leqslant}

\renewcommand{\hat}{\widehat}
\renewcommand{\tilde}{\widetilde}
\renewcommand{\bar}{\overline}

\newcommand\cA{{\mathcal A}}
\newcommand\cB{{\mathcal B}}
\newcommand\cC{{\mathcal C}}
\newcommand\cD{{\mathcal D}}
\newcommand\cE{{\mathcal E}}
\newcommand\cF{{\mathcal F}}
\newcommand\cG{{\mathcal G}}
\newcommand\cH{{\mathcal H}}

\newcommand\cI{{\mathcal I}}
\newcommand\cJ{{\mathcal J}}
\newcommand\cL{{\mathcal L}}
\newcommand\cO{{\mathcal O}}
\newcommand\cM{{\mathcal M}}

\newcommand\cQ{{\mathcal Q}}
\newcommand\cR{{\mathcal R}}
\newcommand\cS{{\mathcal S}}
\newcommand\cT{{\mathcal T}}

\newcommand{\cZ}{\mathcal Z}

\newcommand\bH{{\mathbb H}}

\newcommand\bN{{\mathbb N}}

\newcommand\bR{{\mathbb R}}
\newcommand\bT{{\mathbb T}}

\newcommand\bZ{{\mathbb Z}}
\newcommand\Id{{\mathbbm{1}}}

\DeclareMathAlphabet{\mymathbb}{U}{BOONDOX-ds}{m}{n}

\newcommand{\cW}{\mathcal W}

\newcommand{\Center}{\text{Center}}
\newcommand{\rw}{r_{1,W}}
\newcommand{\funny}{2e^{\varpi_2}}

\begin{document}
\title{Linear response for Sinai billiards with small holes}

\author{Giovanni Canestrari}
\address{Department of Mathematics, University of Toronto, 
40 St.~George Street, Toronto, ON M5S 2E4, Canada. Email:
\texttt{giovanni.canestrari@utoronto.ca}.}

\date{\today}

\begin{abstract}
We show that the conditional survival probability measure for a Sinai billiard with a small hole on the boundary of the table is differentiable with respect to the size \(t\) of the hole at \(t = 0\) and we compute the derivative.
\end{abstract}

\maketitle

\section{Introduction}\label{sec:intro}

The importance of linear response is hard to overestimate. It is a central tool in statistical mechanics and condensed matter physics,  climate science, and engineering (see, e.g., the introduction of \cite{GaLu26} for an overview of the applications). However, from the mathematical point of view, the results are rather limited. In particular, in dynamical systems, there exist strong results only for cases in which the dynamics and the observables are smooth, and the perturbed system is conjugate to the original one. These are very strong limitations and a lot of energy has been devoted to trying to overcome them (see, e.g.\! \cite{MR3729040, MR3622285, MR3876559, CANESTRARI2026111008}). The present paper deals with billiards with holes: an example in which the perturbations are not conjugate to the original system, establishing, for the first time, a linear response formula for the transition from a close to an open system. As the approach is rather general, it is likely to open the door to further applications.

The study of physically relevant invariant measures and their behavior under perturbations is a well-developed field in smooth ergodic theory and dynamical systems. A class of perturbation with its own lore is obtained by opening a hole in the phase space of a dynamical system (see \cite{MR3213491} for recent developments on the topic). These are paradigmatic examples of perturbations that turn a conservative system into a dissipative one. A hole \(\cH\) is a portion of the phase space \(\cM\) with the property that trajectories entering \(\cH\) are removed from the ensemble. If the system is chaotic enough, it is expected that Lebesgue a.e.\! trajectory explores the whole phase space and eventually enters the hole. In this scenario, the analog of SRB measures are the so-called \textit{conditionally invariant measures}. These are measures whose evolution, conditioned to the event of survival, is constant in time. The set of conditionally invariant measures is typically huge, as it contains, for example, periodic orbits that do not intersect the hole, but there are criteria to select measures considered more physically relevant. In particular, if the associated closed system preserves an SRB measure \(\mu_0\), it is natural to consider the evolution of \(\mu_0\) conditioned to survival. Starting with \cite{MR534126}, it has been shown that in many instances such evolution converges to a conditionally invariant measure \(\mu_{\cH}\). (See, e.g., \cite{MR1294552} for expanding maps on the interval with Markovian and non-Markovian \cite{MR1908547, MR2151600} holes, \cite{MR1479506, MR1479505, MR1779391, MR1653291} for Anosov maps, \cite{MR2643708} for non-uniformly hyperbolic settings). Typically, when this convergence takes place, one could start also from another initial ensemble different from \(\mu_0\) and obtain the same limit, and the convergence is exponentially fast for uniformly hyperbolic systems \cite{MR1406590, MR1978986, MR2199394}.

Once existence of physically relevant conditionally invariant measures is established, the next immediate question is to understand their dependence on the hole. A way to formalize this is to consider a family \(\cH_t\) of holes indexed by a real parameter \(t\), select a smooth observable \(\phi: \cM \to \bR\) and study the regularity of the function \(t \to \mu_{\cH_t}(\phi)\). In this context, typical results comprise Hölder continuity of \(\mu_{\cH_t}\). See, for example, \cite{MR2403704} for holes in piecewise \(\cC^2\) hyperbolic maps and \cite{MR3213499, MR2947939} for holes in the billiard. There are of course other ways to grasp information about the dependence of the surviving dynamics on the hole, which include the regularity of the topological entropy \cite{MR857203, MR3681986, MR3833340}, escape rates (see, e.g., \cite{MR2579459, MR2535206, MR4509333} and references therein) and Poissonian approximations (see, e.g.,\! \cite{MR4751722} and reference therein).

In the present work, we consider Sinai billiards, which are paradigmatic examples of conservative chaotic mechanical systems. The dynamics is given by a point-like particle which moves frictionless in a straight line until it collides elastically with some convex scatterer, or wall, having infinite mass. At collision, the particle conserves its speed and obeys the law that the angle of reflection equals the angle of incidence, as prescribed by momentum and energy conservation. Hyperbolicity is due to the strict convexity of the walls. This guarantees that parallel incoming particles are reflected as a divergent beam after a collision, thus spreading trajectories.

We consider an open system as we place a hole of size \(t\) in the boundary of the billiard table. The main contribution of this paper is a linear response result: we prove that the function \(t \to \mu_{\cH_t}(\phi)\) is differentiable at \(t = 0\) for any \(\cC^1\) observable \(\phi\). To the best of the author's knowledge, this is the first result establishing linear response when passing from a closed to an open dynamical system. In this setting, the perturbed evolution fails to be conjugated to the original one so that the traditional techniques for linear response based on structural stability (see \cite{MR1463827}) cannot be applied. In order to bound the modulus of continuity of \(\mu_{\cH_t}(\phi)\), one needs to consider convergence under the conditional evolution of an initial measure, which is essentially an indicator function supported on the image of the hole. As the hole shrinks to a point in the boundary of the table, the support of this initial measure becomes closer and closer to an unstable curve. Despite being singular, measures supported on unstable curves have in fact nice mixing properties for systems with some hyperbolicity. Indeed, this class of measures, called \textit{standard pairs}, has a long history, tracing back at least to Pesin and Sinai \cite{MR721733}, and the modern form has been introduced by Dolgopyat \cite{MR2241812}. Standard pairs are often used together with coupling, a technique of probabilistic origin employed to prove exponential decay of correlations, introduced into dynamical systems by Lai-Sang Young \cite{MR1637655}.

Unfortunately, the cases considered in the literature are not sufficient to obtain the needed mixing result for unstable curves, because none of them consider the leaky evolution. To remedy this, we prove that standard families (which are measures supported on a collection of unstable curves sufficiently long on average) converge exponentially fast to the relevant invariant measure under the conditional evolution of the billiard with the hole. This is a result of independent interest, as it expands the class of admissible initial conditions which relax to \(\mu_{\cH_t}\) in the infinite time limit. Together with the established invariance of standard families for the open dynamics, the above tool allows us to prove differentiability of \(\mu_{\cH_t}\).

Linear response lies at the core of non-equilibrium statistical mechanics because of its relation with fluctuation dissipation relations \cite{MR1412073, MR1224092}. Since the pioneering works of Ruelle in the smooth hyperbolic case \cite{MR1463827, MR1963142}, linear response has been a very active field in smooth ergodic theory and it has been established in a variety of settings (see, e.g., \cite{MR2031432, MR3551256} and recently \cite{CANESTRARI2026111008, AB26, Fu26} and references therein). A class of perturbations that somehow deviates from the classical theory is given by discontinuous perturbations and perturbations of discontinuous systems. These emerge in natural physical models \cite{MR1832968, MR1224092, MR2499824, MR2403704}, and are receiving a lot of attention nowadays \cite{MR4034566, MR4827592, MR4752563, MR4935941, OLTIBOEV2025116944}. However, on the mathematical side, the situation is not neat. We know that whenever the perturbation changes the topological class of the unperturbed map, at least for the corresponding theory of 1D expanding systems, there are many `typical' situations in which differentiability fails (see, e.g., \cite{MR3789184, MR4200481} and the survey \cite{MR3729040} and references therein).

In \cite{CANESTRARI2026111008}, a mechanism responsible for the differentiability of the SRB measure of certain discontinuous perturbations of 2D conservative chaotic systems has been investigated. We prove here that a similar mechanism also applies when the perturbation is obtained by opening a hole. At an intuitive level, the emergence of a hole represents a departure from conservativity and leads to study inhomogeneities in the mass distribution, which are similar to those produced by contributions from retracting or overlapping branches in deterministic perturbations. Note however that this is just an analogy and the case of holes is different from a classical perturbation. For example, the perturbed operator, which specifies the conditional evolution, is nonlinear and the usual `resolvent-type' formulas for the derivative of the SRB measures do not hold.

We remark that our arguments are not billiard-specific - and in fact we believe that they can be extended to different systems with some hyperbolicity and to higher dimensions -  but rely significantly on the geometry of the hole. The relevant property is whether one can foliate the image of the hole with long standard pairs. In particular, a circular hole in the phase space of an Anosov map would be out of reach, and we actually expect failure of linear response in that case. Another related but different question is the regularity of \(\mu_{\cH_t}\) for \(t >0\), i.e., not in the small hole regime (or, more in general, starting from a dissipative unperturbed system). Both these cases seem to have in common the necessity to introduce very small standard pairs in the picture, thus posing additional obstructions to linear response.

Lastly, we mention that, in the existing literature, the problem of stability for holes or other discontinuous perturbations has also been addressed via the `Banach space technology' \cite{MR1294552, MR2403704, MR3170623, MR3213499}. This is a set of powerful techniques, complementary to coupling, based on the abstract perturbation theory developed by Keller and Liverani \cite{MR1679080}. Once this machinery is applied, it yields all at once results about stability of conditionally invariant measures, escape rates and resonances. At the moment, we believe that proving differentiability with Banach space techniques is unfeasible for the systems investigated here. In fact, to obtain linear response, it seems necessary to study the evolution of a single unstable curve. However, as of today and to the author's knowledge, there are no Banach spaces for general piecewise hyperbolic systems (without a skew-product structure) containing standard pairs. 

The paper is structured as follows: in Section \ref{sec:setting}, we introduce a uniform class of Sinai billiards with holes, and we state our main result. In Section \ref{sec:basic-billiard} we recap some basic properties of the billiard map. In Section \ref{sec:standard-families} we state two fundamental results for the conditional dynamics of standard families, namely invariance and `conditional' decay of correlations. These results are then proved in Section \ref{sec:coupling}. Finally, in Section \ref{sec:linear-response} we prove the main theorem about linear response. Sections \ref{sec:linear-response} and \ref{sec:coupling} are independent of each other.

Throughout the paper \(\bN = \{1,2,...\}\), \(\bN_0 = \bN \cup \{0\}\), \(\bT^n = \bR^n /\bZ^n \) and \(c, C >0\) are constants which depend only on the billiard table and on the regularity data \(\varpi_1, B_1, D_1\) of the initial standard families. We also denote by \(C_{\mathrm{cone}} >0\) a constant that depends only on the unstable and stable cones. Actually \(C_{\mathrm{cone}}\) could also be denoted by \(C\), but we highlight in this way some dependencies in our arguments. For a vector \(v = (v_1, v_2) \in \bR^2\), \(|v| = \sqrt{v_1^2 + v_2^2}\) is the Euclidean norm and \(B(v,r)\) is the open ball of radius \(r\) centered in \(v\). For a set \(S \subset \bR^2\), \([S]_r = \cup_{x\in S}B(x,r)\). For a disjoint union of curves \(\Gamma\subset \cM\), we denote by \(|\Gamma|\) the sum of their lengths measured in the Euclidean metric. The scatterers are \(\cD = \cup_i \cD_i\) and the phase space for the billiard is \(\cM = \cup_{i} \partial \cD_{i} \times [-\pi/2, \pi/2]\) so that \(|\partial \cD|\) is the total perimeter of the scatterers. We denote by \(\cC^k(\cM)\) the set of continuous \(k\) times differentiable functions on \(\cM\). We write that \(P \sim Q\) and \(P \lesssim Q\) as shortcuts for \(C^{-1}P \le Q \le C P\) and \(P \le C Q\) respectively.\\

\textbf{Acknowledgments:} I am grateful to Nicholas Fleming Vázquez, Jacopo De Simoi and Carlangelo Liverani for many useful conversations. I acknowledge membership in the GNFM/INDAM and the UMI group DinAmicI.\\
\section{Setting and results}
\label{sec:setting}
We consider a Sinai billiard in \(\bT^2\) of type A according to the classification in \cite{MR2229799}, namely with bounded horizon and no corners. More precisely, we let \(\cD = \cup_{i}\cD_i \subset \bT^2\) be a finite and disjoint union of strictly convex and simply connected regions of \(\bT^2\). We call these regions scatterers and we assume that each \(\partial \cD_i\) is a \(\cC^3\) closed curve. We denote by \(\cM\) the \textit{phase space} for the billiard dynamics
\[
    \cM = \cup_{i} \partial \cD_{i} \times \biggl[-\frac \pi 2, \frac \pi 2\biggr],
\]
and by \(\cF: \cM \to \cM\) the billiard map. To define \(\cF\), we consider a point-like particle of unit mass and speed that is moving frictionless on \(\bT^2 \setminus \cD\) and when it hits the boundary, it reflects according to the law that the angle of reflection equals the angle of incidence. The map \(\cF\) is the Poincaré map corresponding to collisions with the scatterers. Each initial condition is identified by the index \(i\) of the scatterer, the arc-length coordinate on \(\partial \cD_i\) and the angle of the post-collisional velocity with the normal to the scatterers. For notational convenience, we identify \(\cup_i \{i\} \times [0, |\partial \cD_i|]\) with \([0, |\partial \cD|]\) so that we have only one arc-length coordinate and one angle.

For \(x = (r,\vf) \in \cM\), we set \(x_n = (r_n,\vf_n) = \cF^n (x)\). We also call \(\pi_r\) and \(\pi_{\vf}\) the projections on the first and second coordinate respectively. We assume that the billiard has bounded horizon, i.e., there are no trajectories for the billiard flow that never hit a scatterer. In particular, this implies that the free path between collisions \(\tau\) is bounded above by some constant (\cite[Exercise 2.17]{MR2229799}). We call \(\kappa\) the curvature of the boundary, which is the modulus of the second derivative of the arc-length parametrization of \(\partial \cD\). Both \(\kappa\) and \(\tau\) are functions \(\cM \to \bR^{+}\). Moreover, the strict convexity assumption implies that \(\kappa\) is bounded from below. To sum up, as a consequence of our assumptions, there exist \(\tau_{\mathrm{min}}, \tau_{\mathrm{max}}, \kappa_{\mathrm{min}}, \kappa_{\mathrm{max}} >0 \) such that
\begin{equation}\label{eq:bound-table}
\tau_{\mathrm{min}} \le \tau \le \tau_{\mathrm{max}} \quad \text{and}\quad \kappa_{\mathrm{min}} \le \kappa \le \kappa_{\mathrm{max}}.
\end{equation}
The other bounds are a consequence of the fact that the scatterers are disjoint and their boundaries are \(\cC^3\) close curves. It is well known \cite[Lemma 2.35]{MR2229799} that the billiard map preserves a unique absolutely continuous invariant probability measure \(\mu_0\), whose density w.r.t. Lebesgue is given by 
\[
d\mu_0 (r, \vf) = \frac{\cos \vf}{2|\partial \cD|} d\vf dr.
\]
We now introduce the hole. We fix for the rest of the paper a location \(r_{*} \in [0,|\partial \cD|]\) on the boundary of the table and we let \(\cH_t\) be the hole centered at \(r_{*}\) of size \(t\)
\begin{equation}\label{eq:def-hole}
\cH_t = \biggl[r_{*} -\frac{t|\partial \cD|}{2}, r_{*} + \frac{t|\partial \cD|}{2}\biggr] \times \biggl[-\frac{\pi}{2}, \frac{\pi}{2}\biggr],
\end{equation}
in the boundary of some scatterer. The normalization is chosen so that \(\mu_0 (\cH_t) = t\). We consider the corresponding \textit{open dynamical system} in which particles entering the hole \textit{escape} and we stop considering them. For a small \(t_0 > 0\) and \(t \in [0, t_0]\), \(n \in \bN\), we define the open domains \
\begin{equation}\label{eq:open-domain}
\cM_t = \cM \setminus \cH_t \quad \text{and} \quad \cM_{t}^n = \cap_{k =0}^{n}\cF^{-k}(\cM_t).
\end{equation}
The billiard map with the hole \(\cF_t\) is defined as \(\cF\) but restricted to \(\cM_t^1\)
\[
\begin{split}
    &\cF_t: \cM_t \cap \cF^{-1}( \cM_t) \to  \cM_t, \quad  \cF_t = \cF_{|\cM_t \cap \cF^{-1}( \cM_t)}, \quad \cF_t^n = \cF^n_{| \cM_t^{n}}.
\end{split}
\]
This lifts to a map on the set of probability measures given by the conditional probability of survival. For a measure \(\nu\) with \(\nu(\cM_t \cap \cF^{-1}(\cM_t)) > 0\), we define its conditional evolution \(\cL_{t}\nu\) as, for any \(\phi \in \cC^0(\cM)\),
\begin{equation}\label{eq:def-push}
(\cL_{t}\nu) (\phi) = \frac{\nu (\phi \circ \cF  \Id_{\cM_t\cap \cF^{-1}(\cM_t)})}{\nu(\cM_t \cap \cF^{-1}(\cM_t))} = \frac{\nu (\phi \circ \cF  \Id_{\cM_t^1})}{\nu(\cM_t^1)} .
\end{equation}
It is easy to check that iterating \eqref{eq:def-push}, for any \(n \in \bN\) and whenever it is well defined,
\begin{equation}\label{eq:iterate-transf-op-hole}
(\cL^n_{t}\nu) (\phi) =  \frac{\nu (\phi \circ \cF^n  \Id_{\cM_t^n})}{\nu(\cM_t^n)} .
\end{equation}
Note that \(\cL_{t}\) preserves the space of probability measures and, in the case \(t =0\), it reduces to the usual transfer operator \((\cL_0 \nu)(\phi) = \nu (\phi \circ \cF)\). It is well known (\cite[Theorem 2]{MR2579459}, \cite[Theorem 8.1]{MR3213499}) that for all \(t\) small enough there exists a measure \(\mu_t\) such that, for any \(\phi \in \cC^0 (\cM)\),
\begin{equation}\label{eq:conditional-invariant-measure}
\lim_{n \to \infty} (\cL_{t}^n \mu_0) (\phi) = \mu_t (\phi).
\end{equation}
These measures are conditionally invariant in the sense that \(\cL_t \mu_t = \mu_t\). Recall that \(r_{*}\) is the center of the hole \(\cH_t\). The main result of this paper is the following.
\begin{theorem}\label{thm:linear-response}
    For any \(\phi \in \cC^1(\cM)\) the function \(t \to \mu_t( \phi)\) is differentiable at \(t = 0\) and
    \[
    \lim_{t \to 0} \frac{\mu_t(\phi) - \mu_0(\phi)}{t} =\sum_{k=0}^{\infty}\biggl[ \mu_0 (\phi) - \frac{1}{2}\int_{- \pi/ 2}^{\pi/2} \phi \circ \cF^{k}(r_{*}, \vf) \cos \vf d \vf \biggr].
    \]
\end{theorem}
The term \(\frac{1}{2}\int_{- \frac \pi 2}^{\frac \pi 2} \phi \circ \cF^{k}(r_{*}, \vf) \cos \vf d \vf\) can be written as \(\mu_{r_{*}} (\phi \circ \cF^{k})\) where \(\mu_{r_*}\) is the measure obtained by integrating \(\cos \vf\) over the vertical line \(\{(r_*,\vf), \vf \in [-\pi/2, \pi/2]\}\). As a consequence of Lemma \ref{lem:push-vertical-line} and Proposition \ref{prop:invariance-standard-families} in the case \(t=0\), the measures \(\cL_0^{k} \mu_{r_*}\) correspond to standard families, a class of measures defined in Section \ref{sec:standard-families} which are weighted collections of singular measures supported on unstable curves which are not too short on average. It is well known (\cite[Theorem 7.31]{MR2229799} or Theorem \ref{thm:exp-mixing-sf} in the case \(t =0\)) that standard families converge exponentially fast to \(\mu_0\) in the dual of \(\cC^1\) under the dynamics. This in particular implies that the sum on the right-hand side of the equation in Theorem \ref{thm:linear-response} is convergent.

\section{Basic properties of the billiard map}\label{sec:basic-billiard}

Here we recall some well-known properties of \(\cF\) and introduce additional notation related to the billiard dynamics. Let \(\cS_0 := \partial \cM = \{\vf = \pm \pi/2\}\) and \(\cS_1 = \partial \cM \cup \cF^{-1}(\partial \cM)\), \(\cS_{-1} = \partial \cM \cup \cF (\partial \cM)\) be the singularities of the billiard map and its inverse. We define inductively the singularities of the iterates of \(\cF\) and \(\cF^{-1}\) respectively as
\[
     \cS_{n+1} = \cS_n \cup \cF^{-1}(\cS_n) \quad \text{and}\quad \cS_{-(n+1)} = \cS_{-n} \cup \cF(\cS_{-n}).
\]
Homogeneity strips are a well-known tool used to control distortion. Let \(k_0 \in \bN\) be a large constant (how large will be specified later on) and for \(k \ge k_0\), we set
\begin{equation}\label{eq:homogeneity}
\begin{split}
        \bH_{k} = \bH_k^{+} \cup \bH_k^{-} &= \{(r, \vf): \pi/2 -k^{-2} < \vf < \pi/2 -(k+1)^{-2}\}\\
        & \hspace{2cm}\cup \{(r, \vf): -\pi/2 + (k+1)^{-2} < \vf < -\pi/2 +k^{-2}\},\\
        \bH_0 &= \{(r,\vf): -\pi/2 + k_0^2 < \vf < \pi/2 - k_0^2\}.
    \end{split}
\end{equation}
The boundaries of these strips
\[
        \bS_k   = \{(r, \vf):   |\vf| = \pi/2 -k^{-2}\} \cup \{(r, \vf):   |\vf| = -\pi/2 +k^{-2}\}, \quad \bS = \bigcup_{k \ge k_0}\bS_k
\]
define the extended singularity set
\[
 \cS_{n}^{\bH} = \cS_n \cup \bigcup_{k=0}^{n}\cF^{-k}(\bS) \quad \text{and}\quad \cS_{-n}^{\bH} = \cS_{-n} \cup \bigcup_{k=0}^{n}\cF^{k}(\bS).
\]
We denote by \(\partial \cH_t\) the discontinuities of the hole, namely the two vertical lines \(\{ (r_{*} \pm t|\partial \cD|/2, \vf): \vf \in [-\pi/2, \pi/2]\}\). (This set slightly differs from the boundary of \(\cH_t\) since it does not contains the horizontal boundaries, which are already included in \(\cS_0\)). We introduce the set
\[
 \Xi_{n,t} = \bigcup_{k=0}^{n} \cF^{-k}(\partial \cH_t),
\]
which separates the phase space between trajectories that survive up to time \(n\) and those that don't. Finally, we introduce the most general discontinuities
\begin{equation}\label{eq:added-discontinuity-hole}
\cS_{n,t} = \cS_n \cup \Xi_{n,t} \quad \text{and}\quad \cS^{\bH}_{n,t} = \cS_n^{\bH} \cup \Xi_{n,t}.
\end{equation}
Recalling \eqref{eq:open-domain} for the definition of \(\cM_t^n\), 
\begin{equation}\label{eq:discontinuity-survival}
    \partial \cM_t^n \subseteq \bigcup_{k=0}^{n}\partial \cF^{-k} ( \cM_t) \subseteq \bigcup_{k=0}^{n} \cF^{-k}(\partial \cM_t) \cup \cS_{k} =  \bigcup_{k=0}^{n}\cF^{-k} (\partial \cH_t \cup \cS_0) \cup \cS_{k} = \cS_{n,t}.
\end{equation}
As for the hyperbolicity properties, we start by discussing the differential of the map. Denoting by \( D_{(r,\vf)}\cF\) the differential of \(\cF\) at the point \((r,\vf)\), we have (see \cite[Equation (2.26)]{MR2229799})
\begin{equation}\label{eq:differential}
\begin{split}
    D_{(r,\vf)}\cF &=   \frac{-1}{\cos \vf_1}
    \begin{pmatrix}
      \tau \kappa + \cos \vf & \tau \\
      \tau \kappa \kappa_1 + \kappa \cos \vf_1 + \kappa_1 \cos \vf & \tau\kappa_1 + \cos \vf_1
    \end{pmatrix} \\
    &:= \frac{-1}{\cos \vf_1}\begin{pmatrix}
      A & B \\
      C & D
    \end{pmatrix}.
    \end{split}
\end{equation}
Here, we recall that \((r_1, \vf_1) = \cF(r,\vf)\) and we introduced functions \(A,B,C\) and \(D : \cM \to \bR\), which by \eqref{eq:bound-table} have bounded \(L^{\infty}\) norm. In particular,
\begin{equation}\label{eq:bound-differential}
    \|D_{(r,\vf)}\cF\| \le \frac{C}{\cos \vf_1}.
\end{equation}
We also introduce the following unstable cone fields defined on the tangent bundle,
\[
\cC^u_x =\left\{(dr, d\vf): \kappa \le \frac{d\vf}{dr} \le \infty\right\}\quad \text{and}\quad \hat\cC^u_x = \left\{(dr, d\vf): \kappa\le \frac{d\vf}{dr} \le \kappa + \frac{\cos \vf}{\tau_{-1}}\right\},
\]
and stable counterparts obtained by exchanging \(\vf\) with \(-\vf\),
\[
\cC^s_x = \left\{(dr, d\vf): -\infty \le \frac{d\vf}{dr} \le -\kappa\right\} \quad \text{and}\quad \hat\cC^u_x = \left\{(dr, d\vf): -\kappa - \frac{\cos \vf}{\tau} \le \frac{d\vf}{dr} \le  -\kappa\right\}.
\]
In the expression above, the subscript \(-1\) means that the function is evaluated at the preimage of the point. We will sometimes refer to the cone field \(\cC^u\) as the large unstable cone field and to \(\hat\cC^u\) as the small unstable cone field, and similarly for the stable cone fields. According to \cite[Exercise 4.19]{MR2229799}, the following strong form of cone invariance holds
\begin{equation}\label{eq:cone-invariance}
D_x\cF (\cC_x^u) \subseteq  \hat\cC^u_{\cF(x)}\quad \text{and} \quad D_x\cF^{-1} (\cC_x^s) \subseteq  \hat\cC^s_{\cF^{-1}(x)}.
\end{equation}
It is also useful to notice that for any vector \(v = (v_1, v_2) \in \hat \cC_x^u\) or \(\hat  \cC_x^s\), we have that \(C_{\mathrm{cone}}^{-1} \le |v_1/v_2| \le C_{\mathrm{cone}}\). We also know (see e.g. \cite[Equation (4.19)]{MR2229799}) that there exist \(c,\Lambda_0 >1\), such that, for all \(n \in \bN\),
\begin{equation}\label{eq:cone-espansion}
\inf_{v \in \cC^u_x} \frac{\|D\cF^n v\|}{\|v\|} \ge c\Lambda_0^n \quad \text{and}\quad \inf_{v \in \cC^s_x} \frac{\|D\cF^{-n} v\|}{\|v\|} \ge c\Lambda_0^n.
\end{equation}
The relations \eqref{eq:cone-invariance} and \eqref{eq:cone-espansion} make the billiard map \(\cF\) uniformly hyperbolic.\footnote{In the sense that they imply the existence of invariant stable and unstable subspaces defined almost everywhere with uniform contraction and expansion.} We call a \(\cC^1\) curve \(W\subset \cM\) unstable if the tangent line at any point \(x \in W\) belongs to the small unstable cone \(\hat\cC_x^u\) and stable if it belongs to the small stable cone \(\hat\cC_x^s\). By \eqref{eq:cone-espansion} if \(W\) is an unstable, resp. stable, curve (or actually just aligned with the large unstable resp.\! stable cone),
 \begin{equation}\label{eq:curve-expanding}
 |\cF^n(W)| \ge C \Lambda_0^n |W|\quad \text{and}\quad |\cF^{-n}(W)| \ge C \Lambda_0^n |W|.
 \end{equation}
Recall that \(|\cdot|\) measures length in the Euclidean norm. The definition of a stable curve is similar. \cite[ Exercise 4.50]{MR2229799} gives the following bound on the growth of stable and unstable curves.
\begin{lemma}\label{lem:upp-bound-stretch-curves}
There exists \(C>0\) such that for any unstable, resp. stable, curve \(W\), we have  \( |\cF(W)| \le C \sqrt{|W|}\), resp. \( |\cF^{-1}(W)| \le C \sqrt{|W|}\).
\end{lemma}

Recall that \(P \sim Q\) if \(C^{-1}P \le Q \le C P\). The following fact will be used many times.
\begin{lemma}\label{lem:unstable-curves-cosine-k}
For any unstable curve \(W \subset \bH_k\) which intersects both sides of \(\bH_k\), \(\vf \in \pi_{\vf}(\bH_k)\),
  \[
  |W| \sim k^{-3}\sim \cos^{\frac{3}{2}} \vf .
  \]
\end{lemma}
\begin{proof}
The first relation follows from the fact that vectors in the small unstable cone have slope bounded from above and below so that the length of any unstable curve in \(\bH_k\) is comparable with the distance between \(\bS_k\) and \(\bS_{k+1}\). The second relation is a consequence of \(\cos \vf \sim k^{-2}\) on \(\bH_k\), see \eqref{eq:homogeneity}.
\end{proof}
It is useful to notice that \(\cS_{n,t}\setminus \cS_{0,t}\) is a union of stable curve and \(\cS_{-n,t}\setminus \cS_{0,t}\) is a union of unstable curves. Indeed, for \(t = 0\) the statement follows by \cite[Proposition 4.41]{MR2229799} and for \(t>0\) we observe that \(\partial \cH_t\) are vertical lines and so they are aligned to both the large stable and the large unstable cone fields. Then the statement follows by cone invariance \eqref{eq:cone-invariance}. As such, \(\cS_{n,t}\) and \(\cS_{-n,t}\) are transversal to unstable and stable curves respectively, uniformly in the location. We call a curve monotonic if both coordinates \((r,\vf)\) are monotonic functions w.r.t. to the parametrization. The next property is often referred to as \textit{continuation of singularity lines}.
\begin{lemma}\label{lem:continuation-of-singularities}
    For any \(n \in \bN\) and \(t \in [0,t_0]\), every curve \(S \subseteq \cS_{n,t} \setminus \cS_{0}\) is part of some monotonic continuous and piece-wise \(\cC^1\) curve \(\tilde S\subseteq \cS_{n,t} \setminus \cS_{0}\) which terminates on \(\partial \cM = \cS_0\).
\end{lemma}
In the case \(t =0\) this is \cite[Proposition 4.47]{MR2229799}. For \(t > 0\), we note that \(\cS_{n,t}\) coincides with the singularities of a billiard with a couple of `fake corners' placed at the boundary of the hole \(\partial \pi_r(\cH_t)\). These billiards can be regarded as billiards with corners and \cite[Proposition 4.47]{MR2229799} holds as well, proving the statement. We will also use the following simple fact.
\begin{lemma}\label{lem:bound-number-smooth-components}
For each \(n \in \bN\), \(\cS_{n,t}\) is a finite union of \(\cC^1\) components whose number is bounded uniformly in \(t\).
\end{lemma}
\begin{proof}
For \(n = 0\) the statement holds. Suppose that the statement is true for some \(n\) so that \(\cS_{n,t}\) is a finite union of at most \(M_n >0\) curves, uniformly in \(t\). Notice that \(\cS_{n+1,t} = \cS_{n,t} \cup \cF^{-1}(\cS_{n,t})\). The set \(\cS_{n,t}\) consists of curves which are transversal to the small unstable cone field so that each of them intersects \(\cS_{-1}\setminus \cS_0\) in at most \(Q\) points, where \(Q\) is the number of smooth components of \(\cS_{-1}\). Thus, the preimage under \(\cF\) of any of these curves is composed of at most \(Q+1\) smooth components and \(\cS_{n+1,t}\) is composed of at most \((Q+1) M_n + M_n\) curves, independently of \(t\).
\end{proof}

\section{Standard Families}\label{sec:standard-families}
To prove linear response, we need a preliminary result about convergence to equilibrium for the open billiard dynamics starting from a certain class of initial measures. Recall that a smooth curve \(W \subset \cM\) is called unstable if the tangent space \(\cT_x W\) at \(x \in W\) is aligned with \(\hat\cC^u_x\) for any \(x \in W\). Since the vectors in the small unstable cone have bounded slope, \(W\) is the graph of a function \(\vf_W\). For \(D>0\) and \(\delta_{*} \le 1\), we denote by \(\cW(D)\) the set of admissible unstable curves,
\begin{equation}\label{eq:graph-sp-def}
\begin{split}
\cW(D) = \biggl\{ W = \vf_{W}(I): \text{ } |W|\le \delta_{*}, \quad \vf_{W} &\in \cC^2\biggl( I , \biggl[-\frac \pi 2 , \frac \pi 2\biggr]\biggr),  \\
&\cT_x W \in \hat\cC_{x}^u \text{ } \forall x \in W, \quad  |\vf_W''| \le D \biggr\}.\\
\end{split}
\end{equation}
In the definition above, \(I\subset \bR\) is an interval which coincides with the \(r\)-projection of \(W\). Whenever we want to indicate the interval corresponding to some particular \(W \in \cW(D)\), we write
\[
I_W := \{r \in [0, |\Gamma|]: (r,\vf_W(r)) \in W\} = \pi_r (W).
\]
The curves in \eqref{eq:graph-sp-def} are everywhere aligned with the unstable cone field and are graphs of \(\cC^2\) functions.
Moreover, their size is at most \(\delta_{*}\). The parameter \(\delta_{*}\) is to be considered fixed (even if its specific value will be specified later on), and so we do not track this dependence in the notation. We now equip the curves in \(\cW\) with a regular density. For \(\varpi>0\), set
\begin{equation}\label{eq:cones-def}
    \begin{split}
    &\cC(\varpi) = \biggl\{\rho \in \cC^1(I, \bR):\quad \frac{\rho(x)}{\rho(y)} \le e^{\varpi |x-y|^{\frac{1}{3}}} \text{ } \forall x,y \in I,\quad \rho >0, \text{ } \int_I \rho = 1 \biggr\},
\end{split}
\end{equation}
and \(\cC = \cup_{\varpi \in \bR^+}\cC(\varpi)\). Quite often we will use the following simple estimate.
\begin{lemma}\label{lem:ubiquous-cone}
Let \(W \in \cW(D)\) and \(\rho \in \cC(\varpi)\) defined on \(I_W\). We have, for some \(C_{\mathrm{cone}} >0\),
\[
\frac{ e^{-\varpi }}{|W|} \le \rho \le \frac{C_{\mathrm{cone}} e^{\varpi}}{|W|}.
\]
\end{lemma}
\begin{proof}
    By the mean value theorem there exists \(\bar x \in I_W\) such that \(\rho(\bar x) = 1/|I_W|\). Since \(\cT_x W \subseteq \hat\cC_x^u\), we have that \( |I_W| \le |W| \le C_{\mathrm{cone}} |I_W|\). The bound then follows by the definition of the cone and the fact that \(|I_W| \le |W| \le \delta_{*} \le 1\).
\end{proof}
We call \textit{standard pair} the pair \(\ell = (W, \rho)\), where \(W \in \cW(D)\) and \(\rho \in \cC\) is defined on \(I_W\). For a countable set of indices \(\cJ \subseteq \bN\), we call the collection \(\cG = \{(p_j, \ell_j)\}_{j \in \cJ} \equiv \{(p_j, W_j, \rho_j)\}_{j \in \cJ}\) a standard family, where \(p_j \ge 0\), \(\sum_{j} p_j =1\) and \(\ell_j\) are standard pairs. We will often omit the index set \(\cJ\). Any standard pair \(\ell\) induces a measure on \(\cM\) by
\[
\mu_{\ell}(\phi) = \int_{I_W} \rho(r) \phi(r,\vf_W(r))dr, \quad \phi \in \cC^0(\cM).
\]
Similarly, a standard family induces a measure given by the convex combination of the measures of the associated standard pairs,
\[
    \mu_{\cG} (\phi) = \sum_j p_j \mu_{\ell_j}(\phi), \quad \phi \in \cC^0(\cM).
\]
We need to keep track of how small is the average curve is in a standard family. To this aim, we introduce the \textit{measure of the boundary} (or just \textit{boundary}\footnote{See \cite[Exercise 7.14]{MR2229799} for an explanation of the name.}) \(\cZ(\cG)\) of the standard family \(\cG = \{(p_j, W_j, \rho_j)\}\),
\begin{equation}\label{eq:boundary-def}
\cZ(\cG) = \sum_{j} p_j \frac{1}{|W_j|}.   
\end{equation}
It turns out that when evolving standard families it is sufficient to keep track of only a few parameters: the value of \(\cZ\) in \eqref{eq:boundary-def} and \(\varpi\) in \eqref{eq:cones-def}. The first value tells how regular the support of \(\mu_{\cG}\) is, while the second tells how regular the densities are. We will sometimes refer to these values as the regularity of the standard family. For technical reasons, we also need to record the bound on the second derivative of the graphs \(\vf_W\) in \(\cW\). 
\begin{definition}\label{def:regularity-book-keeping}
We say that a standard family \(\cG = \{\bigl(p_j, W_j,\rho_j)\}_{j\in \cJ}\) is a \((\varpi, B, D)\)-standard family for some \(\varpi, B, D >0 \) if \(\rho_j \in \cC(\varpi)\), \(W_j \in \cW(D)\) for all \(j \in \cJ\) and \(\cZ(\cG) \le B\).    
\end{definition}

Because of transversality between the directions in the unstable cone and \(\partial \cH_t\), standard families consisting of long unstable curves cannot concentrate too much mass inside the hole.

\begin{lemma}\label{lem:hole-intersection}
There exists \(C >0\) such that for any  \((\varpi, B,D)\)-standard family \(\cG\), \(t \in [0,t_0]\),
    \[
    \mu_{\cG} (\cH_{t}) \le C e^{\varpi} B t.
    \]
\end{lemma}
\begin{proof}
Let \(\cG = \{(p_j, W_j, \rho_j)\}\) and note that \(|\pi_r (W_j \cap \cH_t)| \le |\pi_r (\cH_t)| = t|\partial \cD|\). Moreover, since  \(\rho_j \in \cC(\varpi)\) and \(\cZ(\cG) \le B\), using Lemma \ref{lem:ubiquous-cone},
\[
\begin{split}
     \mu_{\cG} (\cH_{t} ) = \sum_{j}p_j \int_{I_{W_j}} &\Id_{\cH_t}(r, \vf_{W_{j}}(r)) \rho_j (r)dr =    \sum_{j}p_j \int_{\pi_r (W_j \cap \cH_t)} \!\!\!\rho_j \\
     &\le  |\partial \cD| C_{\mathrm{cone}} e^{\varpi}\sum_{j}p_j \frac{  t}{|W_j|} \le  C e^{\varpi} \cZ(\cG) t\le  C e^{\varpi} B t.
\end{split}
\]
\end{proof}

We now state two results on standard families, Proposition \ref{prop:invariance-standard-families} and Theorem \ref{thm:exp-mixing-sf}, that will be proved in Section \ref{sec:coupling}. These are essentially all we need to know for our application to linear response. The first result is about invariance of standard families.

\begin{proposition}\label{prop:invariance-standard-families}
For any \(\varpi_1, B_1, D_1 >0\) there exist \( \varpi_2, B_2, D_2>0\) and \(t_0>0\) such that the following happens. For any  \((\varpi_{1}, B_{1}, D_{1})\)-standard family \(\cG\), any \(n \in \bN_0\), and any \(t \in [0,t_0]\), there exists another \((\varpi_{2}, B_{2}, D_{2})\)-standard family \(\cG'\) such that, for any \(\phi \in \cC^0(\cM)\),
    \[
    (\cL_{t}^n\mu_{\cG} )(\phi) = \mu_{\cG'} (\phi).
    \]
\end{proposition}

We now fix throughout Sections \ref{sec:standard-families} and \ref{sec:linear-response} some big enough value of  \(\varpi_1, B_1\) and \(D_1\), and we let \(\varpi_{2}\), \(B_{2}\) and \(D_{2}\) be given by Proposition \ref{prop:invariance-standard-families}. In particular, we consider \(\varpi_1 \ge \varpi_{*}\), \(B_1 \ge B_*\) and \(D_1\ge D_*\) where \(\varpi_{*}, B_*, D_* >0\) are given by Lemma \ref{lem:push-vertical-line} and depend only on the billiard table. We call any \((\varpi_{1}, B_{1}, D_{1})\)-standard family a \textit{initial regular standard family} and any \((\varpi_{2}, B_{2}, D_{2})\)-standard family a \textit{regular standard family}. With this terminology, Proposition \ref{prop:invariance-standard-families} tells that the conditional evolution of \(\mu_{\cG}\) for an initial regular standard family \(\cG\) is a measure associated with a regular standard family.

\begin{remark}\label{rem:constant-invariance}
The constants \(B_{2}, \varpi_{2}, D_{2}\) in Proposition \ref{prop:invariance-standard-families} depend only on \(B_1, \varpi_1, D_1\) and the billiard table. Hence, they can be incorporated in the constant \(C\). See the discussion about the general constant \(C\) at the end of Section \ref{sec:intro}.
\end{remark}

We don't really need to keep track of the value of the second derivative \(\vf_{W}''\), but we need only to have a uniform bound throughout the argument. Hence, we will often omit to record the data about \(|\vf_{W}''|\). In particular, we say that \(\cG\) is a \((\varpi, B)\)-standard family if the conditions on the density and boundary are satisfied and \(W_j \in \Sigma(D_2)\) for all \(j\). Before presenting the second result, we add an observation for later, which is a consequence of Proposition \ref{prop:invariance-standard-families}. It cannot happen that, for some iterate, all the mass of a regular standard family ends up in the hole.

\begin{lemma}\label{lem:lower-bound-survival}
    There exist \(C,t_0>0\) such that, for any initial regular standard family \(\cG\), \(n \in \bN\) and \(t \in [0,t_0]\), we have \( \mu_{\cG} (\cM_t^n) \ge (1-Ct)^n\).
\end{lemma}
\begin{proof}
Let \(\tilde \cG\) be a regular standard family (a \((\varpi_2,B_2,D_2)\)-standard family, as described after Proposition \ref{prop:invariance-standard-families}). Observe that
\[
\mu_{\tilde \cG}(\cH_t \cup \cF^{-1}(\cH_t)) \le \mu_{\tilde \cG}(\cH_t) + \mu_{\tilde \cG}(\cF^{-1}(\cH_t)) = \mu_{\tilde\cG}(\cH_t) + (\cL_{0}\mu_{\tilde\cG}) (\cH_t),
\]
where \(\cL_0\) is given by \eqref{eq:def-push} with \(t=0\). We can then apply Proposition \ref{prop:invariance-standard-families} in the case \(t =0\) for initial standard families that are \((\varpi_2, B_2,D_2)\) and we have that \((\cL_{0}\mu_{\tilde \cG})\) is a \((\varpi_3, B_3,D_3)\)-standard family for some \(\varpi_3, B_3, D_3>0\) depending only on \(\varpi_2, B_2, D_2\) and the billiard table. Therefore, by Remark \ref{rem:constant-invariance}, the constants \(\varpi_2, \varpi_3, B_2, B_3, D_2\) and \(D_3\) can be incorporated in the generic constant \(C\) as well. Hence, by Lemma  \ref{lem:hole-intersection} and the equation above, there exists \(C>0\) such that for any regular standard family \(\tilde \cG\),
\begin{equation}\label{eq:hole-plus-preimage}
\mu_{\tilde \cG}(\cH_t \cup \cF^{-1}(\cH_t)) \le  C e^{\varpi_2} B_2 t +  C e^{\varpi_3} B_3 t\le Ct.
\end{equation}
In other words, there exists \(C>0\) such that for any regular standard family \(\tilde \cG\) and \(t \in [0,t_0]\),
\begin{equation}\label{eq:mass-hole-first-iteration}
    \mu_{\tilde \cG}(\cM_t^1) = 1 - \mu_{\tilde \cG}(\cH_t \cup \cF^{-1}(\cH_t)) \ge 1-Ct.
\end{equation}
We now prove the statement by induction. Applying \eqref{eq:mass-hole-first-iteration} to the initial regular standard family \(\cG\) of the statement proves the case \(n=1\). For any \(n \in \bN\),
\begin{equation}\label{eq:induction-hole-mass-inside}
\begin{split}
\mu_{\cG}(\cM_{t}^{n+1}) = \mu_{\cG}(\cM_{t}^{n} \cap \cF^{-n-1}(\cM_t)) &= \mu_{\cG}(\cM_{t}^{n})(\cL_t^n \mu_{\cG})(\cF^{-1}(\cM_t)).
\end{split}
\end{equation}
Let \(t_0>0\) be such that Proposition \ref{prop:invariance-standard-families} applies and assume that the statement holds for some \(n \in \bN\). By the mentioned proposition, for any \(n \in \bN\) and \(t \in [0,t_0]\), there exists a regular standard family \(\cG_{n}\) such that \(\cL_{t}^{n} \mu_{\cG} = \mu_{\cG_n}\). Therefore, by \eqref{eq:mass-hole-first-iteration} in the case \(\tilde \cG = \cG_n\), equation \eqref{eq:induction-hole-mass-inside} and the inductive hypothesis,
\[
\begin{split}
\mu_{\cG}(\cM_t^{n+1}) &= \mu_{\cG}(\cM_t^n) \mu_{\cG_n}(\cF^{-1}(\cM_t)) \ge \mu_{\cG}(\cM_t^n) \mu_{\cG_n}(\cM_{t}^1) \ge (1-Ct)^{n+1}.
\end{split}
\]
This concludes the proof of the statement.
\end{proof}

The second result tells that initial regular standard families converge exponentially fast to the physically relevant conditionally invariant measure \(\mu_t\) under the action of the renormalized push-forward.

\begin{theorem}\label{thm:exp-mixing-sf}
    There exist \(t_0 >0\), \(C>0\) and \(\gamma \in (0,1)\) such that for any two initial regular standard families \(\cG_1\), \(\cG_2\) for any \(n \in \bN\), \(\phi \in \cC^1 (\cM)\) and \(t \in [0,t_0]\),
    \[
    |(\cL_{t}^n\mu_{\cG_1}) (\phi) - (\cL_t^n\mu_{\cG_2}) (\phi)| \le C \|\phi\|_{\cC^1} \gamma^n.
    \]
\end{theorem}

The above result, together with the fact that \(\mu_0\) is well approximated by standard families (see Lemma \ref{lem:exp-conv-initial}), implies exponential convergence of standard families under the conditional evolution to the unique physically relevant measure \(\mu_t\) of Equation \ref{eq:conditional-invariant-measure}.

\begin{corollary}\label{cor:exp-ren-mixing}
     There exist \(t_0 >0\), \(C>0\) and \(\gamma \in (0,1)\) such that for any initial regular standard family \(\cG\), for any \(n \in \bN\), \(t \in [0,t_0]\), \(\phi \in \cC^1 (\cM)\),
    \[
    |(\cL_{t}^n\mu_{\cG}) (\phi) - \mu_t (\phi)| \le C \|\phi\|_{\cC^1} \gamma^n.
    \]
\end{corollary}
The proof of this Corollary is in the next section after Lemma \ref{lem:exp-conv-initial}.

\section{Linear Response}\label{sec:linear-response}
In this section we use the results of Section \ref{sec:standard-families} to prove Theorem \ref{thm:linear-response}. We start by proving that the push-forward of certain measures supported on vertical lines are standard families. 

\begin{lemma}\label{lem:push-vertical-line}
There exist \(\varpi_{*}, B_{*},D_{*}>0\), depending only on \(\cD\), such that, for any \(r \in [0, |\partial \cD|]\), there exists a \((\varpi_{*}, B_{*}, D_{*})\)-standard family \(\cG_{r}\) such that, for any \(\phi \in \cC^0 (\cM)\),
\[
\frac{1}{2 }\int_{-\pi/2}^{\pi/2} \phi \circ \cF (r,\vf) \cos \vf d\vf = \mu_{\cG_r}(\phi).
\]
\end{lemma}
\begin{proof}
Denote by \(l_r\) the vertical line
\[
l_r =\{r\} \times \biggl[-\frac{\pi}{2}, \frac{\pi}{2}\biggr].
\]
The support of a candidate standard family \(\mu_{\cG_r}\) have to coincide with \(\cF(l_r)\). Recalling that \((r_1 (r,\vf), \vf_1 (r,\vf)) = \cF(r,\vf)\) and that \(r\) is fixed, and using \eqref{eq:differential} for the differential and \eqref{eq:bound-table}, we obtain 
\begin{equation}\label{eq:implicit-function-theorem}
    \frac{d }{d \vf}r_1 (r,\vf) = -\frac{\tau}{\cos\vf_1} (r,\vf)\le -\tau_{\mathrm{min}} < 0.
\end{equation}
Hence, the function \(\vf \to r_1 (r,\vf)\) is strictly decreasing and the inverse \(\bar\vf_r : \pi_r (\cF(l_r)) \to [-\pi/2, \pi/2]\) is well defined and it is such that \(r_1 (r, \bar\vf_r (r_1)) = r_1\). Moreover, 
\begin{equation}\label{eq:derivative-inverse}
\frac{d\bar \vf_r}{dr_1} (r_1) = - \frac{\cos \vf_1}{\tau} (r, \bar \vf_r (r_1)).
\end{equation}
Recall \(\bH_k^{\pm}\) from \eqref{eq:homogeneity}. For \(j,k \in \{0\} \cup \{n \ge k_0\}\), \(s_1, s_2 \in \{+,-\}\), we denote by \(W_{j,k,s_1,s_2}\) the maximal connected components of
\begin{equation}\label{eq:1-homogeneous-push-vertical}
\cF(l_r) \cap \bH_k^{s_1} \cap \cF(\bH_j^{s_2}).
\end{equation}
These are the subsets of \(\cF(l_r)\) which belong to just one homogeneity strip and such that their preimage also belongs to just one homogeneity strip. Set \(\vf_{\cF(l_r)} : \pi_r (\cF(l_r)) \to [-\pi/2, \pi/2]\),
\begin{equation}\label{eq:push-graph}
\vf_{\cF(l_r)} (r_1) = \vf_1 (r, \bar \vf_r (r_1)).
\end{equation}
This is the graph of \(\cF(l_r)\). By \eqref{eq:derivative-inverse}, \eqref{eq:push-graph} and a change of variable, and partitioning \(\pi_r (\cF(l_r)) \) into homogeneous components,
\begin{equation}\label{eq:change-variable-vertical}
\begin{split}
\frac{1}{2} \int_{-\pi/2}^{\pi/2} &\phi \circ \cF (r,\vf) \cos\vf d\vf = \frac{1}{2} \int_{\pi_r (\cF(l_r))} \phi  (r_1, \vf_{\cF(l_r)} (r_1))  \cos\bar\vf_r (r_1) \frac{\cos \vf_{\cF(l_r)}(r_1)}{\tau}  dr_1 \\
&= \frac{1}{2} \sum_{j,k,s_1,s_2} \int_{I_{W_{j,k,s_1,s_2}}} \phi  (r_1, \vf_{\cF(l_r)} (r_1))  \cos\bar\vf_r (r_1) \frac{\cos \vf_{\cF(l_r)}(r_1)}{\tau}  dr_1.
\end{split}
\end{equation}
Here the minus sign in \eqref{eq:implicit-function-theorem} cancels with the reversed orientation of \(\pi_r (\cF(l_r))\). Setting 
\begin{equation}\label{eq:density-push-vertical}
\begin{split}
p_{j,k,s_1,s_2} &= \int_{I_{W_{j,k,s_1,s_2}}} \rho ,\quad \rho  =  \frac{\cos\bar\vf_r \cos \vf_{\cF(l_r)}}{2\tau} , \\
&\cG_{r} = \{(p_{j,k,s_1,s_2}, W_{j,k,s_1,s_2}, (\rho \Id_{I_{W_{j,k,s_1,s_2}}} )/ p_{j,k,s_1,s_2})\}_{j,k,s_1,s_2}, 
\end{split}
\end{equation}
Equation \eqref{eq:change-variable-vertical} can be written in the form of a measure associated with the standard family \(\cG_r\). However, we need to check the properties of the graph \(\vf_{\cF(l_r)}\), the density \(\rho\) and the boundary \(\cZ\). By \eqref{eq:cone-invariance} and because the vertical direction belongs to \(\cC^u\), we have that \(\cT_x W_{j,k,s_1,s_2} \in \hat \cC_x^u\). Moreover, recalling the expression for \(\cD\cF\) in \eqref{eq:differential} and \eqref{eq:derivative-inverse}, we have, on \(I_{W_{j,k,s_1,s_2}}\),
\begin{equation}\label{eq:derivativeG-push-vertical}
\begin{split}
\vf_{\cF(l_r)}' (r_1)= \frac{d  }{d \vf}\vf_1 (r, \cdot)_{|\bar \vf_r (r_1)} \frac{d\bar \vf_r}{dr_1} (r_1) &= -\frac{\tau \kappa_1 + \cos \vf_{\cF(l_r)} (r_1)}{\cos \vf_{\cF(l_r)}(r_1)} \biggl (- \frac{\cos \vf_{\cF(l_r)}(r_1)}{\tau} \biggr) \\
&= \frac{\tau \kappa_1 + \cos \vf_{\cF(l_r)}(r_1)}{\tau}.
\end{split}
\end{equation}
Here \(\tau = \tau (r, \bar \vf_r (r_1))\) and \(\kappa_1 = \kappa_1 (r_1)\) are functions of \(r_1\) only. In particular, calling \((\bar x, \bar y)\) the parametrization of \(\partial \cD\) w.r.t.\! the arc-length,
\begin{equation}\label{eq:tau-regularity}
\tau (r,\bar \vf_r (r_1)) = \sqrt{\bigl (\bar x (r) - \bar x(r_1)\bigr )^2 + \bigl (\bar y (r) -\bar y(r_1) \bigr )^2}.
\end{equation}
Hence, by the \(\cC^3\) assumption on \(\partial \cD\) and the fact that \(\tau \ge \tau_{\mathrm{min}}>0\), we have that the function \(r_1 \to \tau (r,\bar \vf_r (r_1))\) is uniformly \(\cC^1\) in the domain \(I_{W_{j,k,s_1,s_2}}\) since \(W_{j,k,s_1,s_2} \subset \cM \setminus  \cS_{-1}\). The same holds for \(\kappa\) because of the regularity of \(\partial \cD\). By differentiating \eqref{eq:derivativeG-push-vertical} once more with respect to \(r_1\) we obtain
\[
\vf_{\cF(l_r)}'' = \frac{\tau' \kappa_1 + \tau \kappa_1 ' - \vf_{\cF(l_r)}' \sin \vf_{\cF(l_r)} }{\tau} - \frac{\tau \kappa_1 + \cos \vf_{\cF(l_r)}}{\tau^2}\tau'.
\]
By \eqref{eq:derivativeG-push-vertical} (or simply noticing that \(\cF(l_r)\) belongs to the small unstable cone), we have that \(|\vf_{\cF(l_r)}'|\) is bounded. Therefore, by the discussion on the regularity of \(\tau\) and \(\kappa\) and the equation above, there exists \(D_{*} >0\) big enough such that \(|\vf_{\cF(l_r)}''| \le D_{*}\). In order to dilute the technical part of the argument, we postpone the proof that \(\rho \in \cC(\varpi_{*})\) for some \(\varpi_{*}>0\) to Lemma \ref{lem:technical-cone-rho} right below. It remains to estimate the boundary \(\cZ\). Recalling the definition \eqref{eq:1-homogeneous-push-vertical} of \(W_{j,k,s_1,s_2}\) and by Lemma \ref{lem:unstable-curves-cosine-k}, for all \(j,k,s_1,s_2\),
\begin{equation}\label{eq:estimates-cosines-0}
\begin{split}
    & \max_{ I_{W_{j,k,s_1,s_2}}}\cos \bar \vf_r \sim j^{-2} \quad \text{and}\quad \max_{ I_{W_{j,k,s_1,s_2}}}\cos \vf_{\cF(l_r)} \sim k^{-2}.
\end{split}
\end{equation}
In the first equation we used that \(\cF^{-1}(W_{j,k,s_1,s_2}) \subset \bH_j^{s_2}\) and in the second that \(W_{j,k,s_1,s_2} \subset \bH_k^{s_1}\).
Hence, by \eqref{eq:density-push-vertical} and \eqref{eq:estimates-cosines-0}, we have that 
\[
p_{j,k,s_1,s_2} = \int_{I_{W_{j,k,s_1,s_2}}}  \frac{\cos\bar\vf_r \cos \vf_{\cF(l_r)}}{2\tau} \le C \frac{|I_{W_{j,k,s_1,s_2}}|}{j^2 k^2}\le C \frac{|W_{j,k,s_1,s_2}|}{j^2 k^2}.
\]
Recalling that \(s_{1},s_2 \in \{+, -\}\), there exists \(B_*>0\) such that
\[
\begin{split}
\cZ(\cG_{r}) =  \sum_{j,k,s_1,s_2} p_{j,k,s_1,s_2} \frac{1}{|W_{j,k,s_1,s_2}|} \le C \sum_{j,k}\frac{1}{j^2 k^2} \le B_*.
\end{split}
\]
This concludes the proof of the Lemma.
\end{proof}

\begin{lemma}\label{lem:technical-cone-rho}
    Let \(\rho\) be defined by \eqref{eq:density-push-vertical} and \(W_{j,k,s_1,s_2}\) by \eqref{eq:1-homogeneous-push-vertical}. There exists \(\varpi_{*}>0\) such that for any \(r_a, r_b \in I_{W_{j,k,s_1,s_2}}\),
    \[
    \frac{\rho(r_a)}{\rho(r_b)} \le e^{\varpi_* |r_a - r_b|^{\frac 1 3}}.
    \]
\end{lemma}
\begin{proof}
Recall \eqref{eq:push-graph} and the discussion after \eqref{eq:implicit-function-theorem} for the definition of \(\vf_{\cF(l_r)}\) and \(\bar \vf_r\). By \eqref{eq:bound-differential}, for some \(c>0\),
\[
|\cF^{-1}(W_{j,k,s_1,s_2})| \ge c  |W_{j,k,s_1,s_2}| \min_{I_{W_{j,k,s_1,s_2}}} \cos \vf_{\cF(l_r)} 
\]
Moreover, by definition \(\cF^{-1}(W_{j,k,s_1,s_2})\) is a subset of \(\bH_j^{s_2}\) so that by Lemma \ref{lem:unstable-curves-cosine-k}, for every \(r_1 \in I_{W_{j,k,s_1,s_2}}\), we have that \( |\cF^{-1}(W_{j,k,s_1,s_2})|^{2/3} \lesssim \cos \bar \vf_r\) and, for some \(C>0\),
\[
\begin{split}
\frac{\cos\vf_{\cF(l_r)}(r_1)}{\cos \bar \vf_r (r_1)} &\le C \frac{\cos \vf_{\cF(l_r)}(r_1)}{|\cF^{-1}(W_{j,k,s_1,s_2})|^{\frac 2 3}} \le C \frac{\cos \vf_{\cF(l_r)}(r_1)}{|W_{j,k,s_1,s_2}|^{\frac 2 3} \min_{I_{W_{j,k,s_1,s_2}}} \cos \vf_{\cF(l_r)}(r_1)^{\frac 2 3}}\\
&\le C\frac{1}{|W_{j,k,s_1,s_2}|^{\frac 2 3}}\le C C_{\mathrm{cone}} \frac{1}{|I_{W_{j,k,s_1,s_2}}|^{\frac 2 3}}.
\end{split}
\]
Therefore, using the lower bound on \(\tau\) and \eqref{eq:derivative-inverse}, for some \(C>0\) and any \(r_a, r_b \in I_{W_{j,k,s_1,s_2}}\),
\[
\begin{split}
|\ln \cos \bar\vf_r (r_{a}) -\ln \cos \bar\vf_r (r_{b})| &\le \int_{r_a}^{r_b} \biggl|\frac{\sin \bar \vf_r (r_1)}{\cos \bar \vf_r (r_1)}\frac{\cos \vf_{\cF(l_r)}(r_1)}{\tau} \biggr | dr_1 \le  \frac{ C |r_b - r_a|}{|I_{W_{j,k,s_1,s_2}}|^{\frac 2 3}} \le C |r_b - r_a|^{\frac 1 3}.
\end{split}
\]
Since \( |W_{j,k,s_1,s_2}|^{2/3}\sim |I_{W_{j,k,s_1,s_2}}|^{2/3} \lesssim \cos \vf_{\cF(l_r)}\), a similar computation yields,  for some \(C>0\),
\[
\begin{split}
|\ln \cos \vf_{\cF(l_r)} (r_{a})& -\ln \cos \vf_{\cF(l_r)} (r_{b})| \le \int_{r_a}^{r_b} \biggl|\frac{\sin \vf_{\cF(l_r)}(r_1)}{\cos \vf_{\cF(l_r)}(r_1)} \biggr| |\vf_{\cF(l_r)}' (r_1)|dr_1 \le    C |r_{b} - r_{a}|^{\frac 1 3}.
\end{split}
\]
Finally, using the fact that \(\tau\) is uniformly \(\cC^1\) as a function of \(r_1\) on \(I_{W_{j,k,s_1,s_2}}\), as discussed after \eqref{eq:tau-regularity}, for some \(C>0\),
\[
|\ln \tau(r_a) - \ln \tau (r_b)| \le \int_{r_b}^{r_a} \biggl|\frac{\tau'(r_1)}{\tau(r_1)}\biggr|dr_1 \le C |r_a- r_b|.
\]
Since \(\rho  = (\cos\bar\vf_r \cos \vf_{\cF(l_r)})/(2\tau)\), the previous three equations prove the Lemma.
\end{proof}

Recalling the discussion after Proposition \ref{prop:invariance-standard-families}, we fix \(\varpi_1> \varpi_{*}\), \(B_1 >B_{*}\) and \(D_1 >D_{*}\) so that, according to Lemma \ref{lem:push-vertical-line}, the measures
\[
\phi \to \frac{1}{2 }\int_{-\pi/2}^{\pi/2} \phi \circ \cF (r,\vf) \cos \vf d\vf,
\]
correspond to initial regular standard families for any \(r\). Working with regular standard families is almost sufficient for our calculations, but not really. We need to `pass to the limit' with the following slightly larger class of measures \(\cM_{\omega, B,D}\).

\begin{definition}\label{def:initial-measures}
 We denote by \(\cM_{\omega, B, D}\), for \(\omega, B,D >0\), the weak-\(\star\) closure of the set of measures corresponding to \((\omega, B, D)\)-standard families. We say that \(\mu\) is a initial regular measure if \(\mu \in \cM_{\varpi_1, B_1, D_1}\) and that \(\mu\) is a regular measure if \(\mu \in \cM_{\varpi_2, B_2, D_2}\).
\end{definition}

In the next Lemma we show that \(\mu \in \cM_{\varpi, B, D}\) has a representation in terms of standard families also when it acts on discontinuous functions whose discontinuities are allowed only on \(\cS_{n,t}\). This result uses transversality between unstable curves and \(\cS_{n,t}\) and the fact that the family has bounded \(\cZ\).

\begin{lemma}\label{lem:convergence-discontinuous-function} Let \(\mu \in \cM_{\omega, B,D}\) and \(\{\cG_k\}_k\) be any sequence of \((\omega, B, D)\)-standard families such that \(\mu = \lim_k \mu_{\cG_k}\). For any bounded function \(h\), discontinuous at most on \(\cS_{n,t}\), \(n \in \bN\), \(t \in [0,t_0]\), we have that \(\lim_{k \to \infty}\mu_{\cG_k}(h) = \mu(h)\).
\end{lemma}
\begin{proof}
For any \(\ve >0\), let \(h_{\ve} \in \cC^0(\cM)\) be such that \(h_{\ve} = h\) on \(\cM \setminus [\cS_{n,t}]_{\ve}\) and \(\|h_{\ve}\|_{\cC^0} \le \|h\|_{L^{\infty}}\). Set also \(\cG_k = \{(p_j^k, W_{j}^k, \rho_{j}^k)\}\). Recall from the discussion after Lemma \ref{lem:unstable-curves-cosine-k} that \(\cS_{n,t} \setminus \cS_{0,t}\) is a union of stable curves. Therefore, using also Lemma \ref{lem:bound-number-smooth-components}, for any \(n\), the set \(\cS_{n,t}\) consists of a finite union of curves transversal to the unstable cone whose number and transversality is uniform in \(t\). Hence, there exists \(C_n >0\) such that, for all \(k\), \(j\), \(t \in [0,t_0]\) and \(\ve >0\),
\begin{equation}\label{eq:transversality-unstable-discontinuities-unif-t}
  |\pi_r (W_j^k \cap [\cS_{n,t}]_{\ve})| \le  | W_j^k \cap [\cS_{n,t}]_{\ve}| \le C_n \ve.
\end{equation}
By definition of \(\mu \in \cM_{\omega, B,D}\), we have  that \(\rho_{j}^k \in \cC(\varpi)\). Hence, by \eqref{eq:transversality-unstable-discontinuities-unif-t}, Lemma \ref{lem:ubiquous-cone} and the fact that \(\cZ(\cG_k) \le B\), for some \(C_{\mathrm{cone}} >0\),
\begin{equation}\label{eq:discontinuities-measure-standard-family}
\begin{split}
 \mu_{\cG_k} ([\cS_{n,t}]_{\ve}) &= \sum_{j}p_j^k \int_{\pi_r (W_j^k \cap [\cS_{n,t}]_{\ve})} \rho_j^k \le C_{\mathrm{cone}}  e^{\varpi}\sum_{j}p^k_j \frac{C_n \ve}{|W_j^k|} \le C_{\mathrm{cone}} e^{\varpi} C_n B \ve.
 \end{split}
\end{equation}
Therefore, by definition of \(h_{\ve}\), for each \(k \in \bN\) and \(\ve >0\),
\begin{equation}\label{eq:Gk-todiscontinuous}
|\mu_{\cG_k} (|h-h_{\ve}|) | \le 2\|h\|_{L^{\infty}}\mu_{\cG_k} ([\cS_{n,t}]_{\ve})  \le C_{\mathrm{cone}} e^{\varpi} C_n B \|h\|_{L^{\infty}}\ve.
\end{equation}
By Lusin's theorem, for any \(\ve >0\), there exist a set \(L_{\ve}\) and a continuous function \(\tilde h_{\ve}\) such that \(h_{\ve} - h = \tilde h_{\ve}\) on \(L_{\ve}\) and \(\mu (\cM \setminus L_{\ve}) \le \ve\). We can also require that \(\|\tilde h_{\ve}\|_{\cC^0} \le \|h - h_{\ve}\|_{L^{\infty}} \le 2\|h\|_{L^{\infty}}\). Therefore,
\[
\begin{split}
\mu(h_{\ve}) -\mu(h) = \mu(h_{\ve} - h) &= \mu (\tilde h_{\ve}) + \mu ((\tilde h_{\ve} - h_{\ve} - h) \Id_{\cM \setminus L_{\ve}}) \\
&= \lim_{k \to \infty}\mu_{\cG_k}(\tilde h_{\ve}) + \mu((\tilde h_{\ve} - h_{\ve} - h)\Id_{\cM \setminus L_{\ve}}).
\end{split}
\]
By the equation above and \eqref{eq:Gk-todiscontinuous}, for any \(\ve >0\)
\begin{equation}\label{eq:muheps}
|\mu(h_{\ve}) - \mu(h)| \le \sup_k \mu_{\cG_k} (|h - h_{\ve}|) + \mu (\Id_{\cM \setminus L_{\ve}}) \|\tilde h_{\ve} - h_{\ve} - h\|_{L^{\infty}} \le  ( C_{\mathrm{cone}} e^{\varpi} C_n B + 4)\|h\|_{L^{\infty}} \ve .
\end{equation}
Moreover, since \(h_{\ve}\) is a continuous function, \(\lim_{k \to \infty} \mu_{\cG_k}(h_\ve) = \mu (h_{\ve})\). Therefore, by \eqref{eq:Gk-todiscontinuous} and \eqref{eq:muheps}, for any \(\ve >0\) there exists \(k\) big enough such that
\[
\begin{split}
|\mu(h) - \mu_{\cG_k}(h)| &\le |\mu_{\cG_k} (h) - \mu_{\cG_k} (h_{\ve}) | + |\mu_{\cG_k}(h_{\ve}) - \mu(h_{\ve})| + |\mu(h_{\ve}) - \mu(h)|\\
& \le C_{\mathrm{cone}} e^{\varpi} C_n B \|h\|_{L^{\infty}}\ve + \ve + ( C_{\mathrm{cone}} e^{\varpi} C_n B + 4)\|h\|_{L^{\infty}} \ve.
\end{split}
\]
The arbitrariness of \(\ve\) concludes the proof of the statement.
\end{proof}
Recall that initial regular measures were introduced by Definition \ref{def:initial-measures}.
\begin{lemma}\label{lem:limit-iterate}
 Let \(\mu\) be a initial regular measure and \(\{\cG_k\}_k\) be a sequence of initial regular standard families such that \(\mu = \lim_{k}\mu_{\cG_k}\). There exists \(t_0>0\) such that, for any \(n\in \bN\) and \(t \in [0, t_0]\), 
 \[
 \lim_{k \to \infty} \cL_{t}^n \mu_{\cG_k} = \cL_{t}^n \mu.
 \]
\end{lemma}
\begin{proof}
First notice that by Lemma \ref{lem:lower-bound-survival} there exists \(t_0\) small enough such that \(\mu_{\cG_k}(\cM_{t}^n) \ge (1/2)^n\) for any \(k\) and \(t \in [0,t_0]\). By \eqref{eq:discontinuity-survival} the function \(\Id_{\cM_{t}^n}\) is discontinuous at most on \(\cS_{n,t}\) so that, by Lemma \ref{lem:convergence-discontinuous-function}, we have that \( \mu(\cM_{t}^n) = \lim_{k \to \infty}\mu_{\cG_k}(\cM_{t}^n)  \ge (1/2)^n\) as well. Therefore, for any \( t\in [0,t_0]\) and \(n \in \bN\), both measures \(\cL_{t}^n \mu\) and \(\cL^n_t \mu_{\cG_k}\) are well defined.
Recalling the expression \eqref{eq:iterate-transf-op-hole} for the iterates \(\cL_t^n\), for any \(\phi \in \cC^0(\cM)\),
\[
(\cL_{t}^n \mu)(\phi) - (\cL_{t}^n \mu_{\cG_k})(\phi) = \frac{\mu(\phi \circ \cF^n \Id_{\cM_t^n} )\mu_{\cG_k}(\cM_t^n) - \mu_{\cG_k}(\phi \circ \cF^n \Id_{\cM_t^n})\mu(\cM_t^n)}{\mu(\cM_t^n) \mu_{\cG_k}(\cM_t^n)}.
\]
Fix \(\ve >0\). By \eqref{eq:discontinuity-survival}, \(\phi \circ \cF^n\Id_{\cM_t^n}\) is discontinuous at most on \(\cS_{n,t}\). Hence, by Lemma \ref{lem:convergence-discontinuous-function}, there exists \(k\) big enough such that 
\[
|\mu(\phi \circ \cF^n \Id_{\cM_t^n} ) - \mu_{\cG_k}(\phi \circ \cF^n \Id_{\cM_t^n} )| \le \ve \quad \text{and} \quad |\mu(\cM_t^n ) - \mu_{\cG_k}(\cM_t^n )| \le \ve.
\]
Therefore, using the lower bounds on \(\mu_{\cG_k}(\cM_{t,n})\) and \(\mu(\cM_{t}^n)\) 
\[
\begin{split}
|(\cL_{t}^n \mu)(\phi) - (\cL_{t}^n \mu_{\cG_k})(\phi)| &\le \frac{1}{\mu(\cM_t^n) \mu_{\cG_k}(\cM_t^n)} \bigl( |\mu(\phi \circ \cF^n \Id_{\cM_t^n} ) - \mu_{\cG_k}(\phi \circ \cF^n \Id_{\cM_t^n})|\\
&\hspace{1cm}|\mu_{\cG_k}(\cM_t^n)| + |\mu_{\cG_k}(\phi \circ \cF^n \Id_{\cM_t^n})| |\mu(\cM_t^n)-\mu_{\cG_k}(\cM_t^n)| \bigr)\\
&\le 4^n (\ve + \|\phi\|_{\cC^0} \ve ).
\end{split}
\]
By the arbitrariness of \(\ve\), we conclude the proof.
\end{proof}

We now `lift' invariance for standard families to the bigger class \(\cM_{\varpi, B, D}\).

\begin{lemma}\label{lem:invariance}
There exists \(t_0 >0\) such that, for any initial regular measure \(\mu\), \(n \in \bN\) and \(t \in [0, t_0]\), we have that \(\cL_{t}^n \mu\) is a regular measure. 
\end{lemma}
\begin{proof}
    By Definition \ref{def:initial-measures} of initial regular measure \(\mu\) there exists a sequence \(\{\cG_k\}_k\) of initial regular standard families such that \(\mu = \lim_{k \to \infty} \mu_{\cG_k}\). By Lemma \ref{lem:limit-iterate}, for \(t_0 >0\) small enough and \(t \in [0,t_0]\), we have that \(\cL_{t}^n \mu = \lim_{k \to \infty} \cL_{t}^n \mu_{\cG_k}\) for any \(n\). On the other hand, by Proposition \ref{prop:invariance-standard-families} and because \(\cG_k\) are initial regular, there exist regular standard families \(\cG_{n,k}'\) such that \(\cL_{t}^n \mu_{\cG_k} = \mu_{\cG_{k,n}'} \). Hence, \(\cL_{t}^n \mu = \lim_{k \to \infty} \mu_{\cG_{n,k}'}\), showing that \(\cL_{t}^n \mu\) are regular measures.
\end{proof}

The decay of correlations for initial regular standard families carries over to initial regular measures.

\begin{lemma}\label{lem:exp-mixing}
There exists \(t_0>0\), \(C>0\) and \(\gamma \in (0,1)\) such that for any two initial regular measure \(\mu_1, \mu_2\), \(t \in [0,t_0]\) and \(n \in \bN\), for any \(\phi \in \cC^1 (\cM)\),
   \[
   \bigl|(\cL_{t}^n \mu_1)(\phi) - (\cL_{t}^n \mu_2)(\phi)  \bigr| \le C \|\phi\|_{\cC^1}\gamma^n .
   \]
\end{lemma}
\begin{proof}
Let \(C>0\) and \(\gamma \in (0,1)\) be given by Theorem \ref{thm:exp-mixing-sf} and \(t_0 >0\) small enough so that Theorem \ref{thm:exp-mixing-sf} and Lemma \ref{lem:limit-iterate} apply, and take \(t \in [0,t_0]\), \(n \in \bN\) and \(\phi \in \cC^1(\cM)\). By Lemma \ref{lem:limit-iterate} and because \(\mu_1\) and \(\mu_2\) are initial regular measures, for any \(\ve >0\) there exist two initial regular standard families \(\cG_{1,\ve}\) and \(\cG_{2,\ve}\) such that,
\[
|(\cL_{t}^n \mu_\iota)(\phi) - (\cL_{t}^n \mu_{\cG_{\iota,\ve}})(\phi) | \le \ve \quad \text{for } \iota \in \{1,2\}.
\]
Hence, by Theorem \ref{thm:exp-mixing-sf},
\[
\begin{split}
|(\cL_{t}^n \mu_1)(\phi) - (\cL_{t}^n \mu_2)(\phi)| &\le   |(\cL_{t}^n \mu_{\cG_{1,\ve}})(\phi) - (\cL_{t}^n\mu_{\cG_{2,\ve}})(\phi)| + |(\cL_{t}^n \mu_1)(\phi) - (\cL_{t}^n \mu_{\cG_{1,\ve}})(\phi)|\\
&\hspace{0.5cm}+  |(\cL_{t}^n \mu_2)(\phi) - (\cL_{t}^n \mu_{\cG_{2,\ve}})(\phi)| \\
&\le C \|\phi\|_{\cC^1}\gamma^n + 2\ve,
\end{split}
\]
and, by the arbitrariness of \(\ve\), we conclude the proof of the Lemma.
\end{proof}
For \(t \in [0,t_0]\), we define the measures \(\cD_t \) as, for \(\phi \in \cC^0(\cM)\),
\begin{equation}\label{eq:survival-and-hole-measures}
 \cD_t (\phi) = \frac{\mu_0(\phi  \Id _{\cF(\cH_t)})}{\mu_0(\cF(\cH_t))} = \frac{\mu_0(\phi  \Id _{\cF(\cH_t)})}{t} = \frac{\mu_0(\phi \circ \cF  \Id _{\cH_t})}{t}.
\end{equation}
Recall that \(\mu_0\) is the SRB probability measure for the billiard map \(\cF\).
\begin{lemma}\label{lem:exp-conv-initial}
For any \(t \in [0, t_0]\), the measures \(\cD_t\) and \(\mu_0\) are initial regular measures.
\end{lemma}
\begin{proof}
Let \(t \in [0,t_0]\). We first show that \(\cD_t\) is an initial regular measure. Recall that \(r_* \in [0, |\partial \cD|]\) is the arc-length coordinate of the center of the hole \(\cH_t\). For any \(N \in \bN\) and \(j \in \{0,1,...,N-1\}\), let \(r_{j} = r_{*} - \frac{t|\partial \cD|}{2} + j\frac{t|\partial \cD|}{N}\) be a evenly spread grid in \(\pi_r (\cH_t)\) and set
\begin{equation}\label{eq:discrete-approx}
\cD_{t,N} (\phi) =  \frac{1}{N}\sum_{j=0}^{N-1} \frac{1}{2} \int_{-\pi/2}^{\pi/2}\phi \circ \cF (r_{j},\vf)\cos\vf d\vf.
\end{equation}
The measures \(\cD_{t,N}\) are discrete - vertical lines - approximations of \(\cD_t\) as it can be seen by the last expression for \(\cD_t\) in \eqref{eq:survival-and-hole-measures}.
By Lemma \ref{lem:push-vertical-line}, there exist initial regular standard families \(\cG_{r_{j}}\) such that,
\[
\cD_{t,N} =  \sum_{j=0}^{N-1} \frac{1}{N} \mu_{\cG_{r_j}}.
\]
Since the regularity of standard families is preserved by convex combinations, we have that \(\cD_{t,N} = \mu_{\cG_{t,N}}\) for some initial regular standard family \(\cG_{t,N}\). Hence, in order to prove that \(\cD_t\) is an initial regular measure, it remains to prove that \(\lim_{N \to \infty}\cD_{t,N} = \cD_t\) on continuous functions. To this aim, consider the measures, for \(\phi \in \cC^0(\cM)\),
\[
\begin{split}
&\zeta_t (\phi) =  \frac{1}{t}\int_{\cH_t} \phi (r,\vf)\frac{\cos \vf d\vf dr}{2|\partial \cD|} \quad \text{and} \quad \zeta_{t,N}(\phi) =  \frac{1}{ N}\sum_{j=0}^{N-1} \frac{1}{2} \int_{-\pi/2}^{\pi/2}\phi  (r_{j},\vf)\cos\vf d\vf.
\end{split}
\]
Note that \(\cD_t(\phi) = \zeta_t (\phi \circ \cF)\),\(\cD_{t,N} (\phi) = \zeta_{t,N} (\phi \circ \cF)\) and \(\lim_{N \to \infty} \zeta_{t,N} = \zeta_t\) on continuous functions because \(\{r_j\}_j\) is a \(|\pi_r(\cH_t)|/N\)-dense grid in \(\pi_r(\cH_t)\).  However, we cannot conclude directly because \(\phi \circ \cF\) might be discontinuous on \(\cS_1\) and we must argue by approximation. Fix any \(\ve >0\). There exists a continuous function \(g_{\ve}\) such that \(g_{\ve} = \phi \circ \cF\) on \(\cM \setminus [\cS_{1}]_{\ve}\) and \(\|g_{\ve}\|_{L^{\infty}} \le \|\phi \circ \cF\|_{L^{\infty}} \le \|\phi\|_{\cC^0}\). Moreover,
\[
\zeta_t([\cS_{1}]_{\ve}) =  \frac{1}{t}\int_{\cH_t} \Id _{[\cS_{1}]_{\ve}}\frac{\cos \vf d\vf dr}{2|\partial \cD|} \le \frac{m ([\cS_{1}]_{\ve} \cap \cH_t)}{ 2|\partial \cD|t}.
\]
As already mentioned, \(\cS_1 \setminus \cS_0\) belong to the small stable cone and \(\cS_0\) are horizontal lines. Hence, \(\cS_1\) consists of a finite union of smooth curves transversal to the vertical direction. Therefore, \(m ([\cS_{1}]_{\ve} \cap \cH_t) \le C_{\mathrm{cone}}t \ve\) for some \(C_{\mathrm{cone}}>0\) and \(\zeta_t([\cS_{1}]_{\ve}) \le C_{\mathrm{cone}}\ve\). Because it is supported on vertical lines, a similar bound holds for \(\zeta_{t,N}\), yielding, uniformly in \(N\),
\[
\max \{\zeta_t([\cS_{1}]_{\ve}), \zeta_{t,N}([\cS_{1}]_{\ve})\} \le C_{\mathrm{cone}} \ve.
\]
Therefore, 
\[
|\zeta_t (\phi \circ \cF) - \zeta_t (g_{\ve})| \le \|\phi \circ \cF - g_{\ve}\|_{L^{\infty}} \zeta_t ([\cS_{1}]_{\ve}) \le 2C_{\mathrm{cone}} \|\phi\|_{\cC^0} \ve,
\]
and, uniformly in \(N\),
\[
\begin{split}
|\zeta_{t,N} (\phi \circ \cF) - \zeta_{t,N} (g_{\ve})| \le \|\phi \circ \cF - g_{\ve}\|_{L^{\infty}} \zeta_{t,N} ([\cS_{1}]_{\ve}) \le 2C_{\mathrm{cone}} \|\phi\|_{\cC^0} \ve.
\end{split}
\]
Let \(N\) be big enough such that \(|\zeta_{t,N} (g_{\ve}) - \zeta_t (g_{\ve})| \le \ve\). Then, using the last two equations,
\[
\begin{split}
    |\cD_{t,N}(\phi)& - \cD_t (\phi)| = |\zeta_{t,N} (\phi \circ \cF) - \zeta_t (\phi \circ \cF)| \le |\zeta_{t,N} (\phi \circ \cF) - \zeta_{t,N} (g_{\ve})| \\
    &+ |\zeta_{t,N} (g_{\ve}) - \zeta_t (g_{\ve})| + |\zeta_t(g_{\ve}) - \zeta_t(\phi \circ \cF)|\le 4C_{\mathrm{cone}}\|\phi\|_{\cC^0}\ve + \ve.
\end{split}
\]
By the arbitrariness of \(\ve\) this shows that \(\lim_{N \to \infty}\cD_{t,N} = \cD_t\) on continuous functions and \(\cD_t\) is a initial regular measure. As for \(\mu_0\), by invariance under \(\cF\),
\[
\mu_0 (\phi) = \mu_0 (\phi \circ \cF) = \int_{\cM} \phi \circ \cF (r,\vf) \frac{\cos \vf d\vf dr}{2|\partial \cD|}.
\]
We can then approximate the equation above with measures \(\mu_{0,N}\) supported on vertical lines as in \eqref{eq:discrete-approx}, which become denser and denser in \(\cM\), and by transversality one obtains that \(\max \{\mu_0([\cS_{1}]_{\ve}), \mu_{0,N}([\cS_{1}]_{\ve})\} \le C_{\mathrm{cone}} \ve\). Hence, we can repeat the same argument as above and conclude that \(\mu_0 = \lim_{N \to \infty} \mu_{\cG_k}\) for a sequence of initial regular standard families \(\{\cG_k\}_k\). This concludes the proof of the Lemma.
\end{proof}

\begin{proof}[Proof of Corollary \ref{cor:exp-ren-mixing}]
Let \(t_0, C>0\) and \(\gamma \in (0,1)\) be given by  Lemma \ref{lem:exp-mixing}. By \eqref{eq:conditional-invariant-measure}, we have that \(\lim_{k \to \infty} \cL_t^k \mu_0 = \mu_t\). Hence, for any initial regular standard family \(\cG\) and \(n \in \bN\)
\begin{equation}\label{eq:decay-to-mut}
    |\cL_t^n \mu_{\cG} - \mu_t| = \lim_{m \to \infty} |\cL_t^n \mu_{\cG} - \cL_t^{n} (\cL_t^m\mu_0)| .
\end{equation}
By Lemma \ref{lem:exp-conv-initial}, \(\mu_0\) is an initial regular measure and, by Lemma \ref{lem:invariance}, \(\cL_t^m \mu_0\) is a regular measure. In other words \(\cL_t^m \mu_0 \in \cM_{\omega_2, B_2, D_2}\) for any \(m \in \bN\). Since we are free to choose any arbitrary \(\varpi_1, B_1, D_1\) as initial conditions in Proposition \ref{prop:invariance-standard-families}, we may consider measures in \(\cM_{\omega_2, B_2, D_2}\) as initial regular measures. Therefore, possibly with a larger \(C>0\), Lemma \ref{lem:exp-mixing} holds for measures in \(\cM_{\omega_2, B_2, D_2}\), and the expression inside the limit in \eqref{eq:decay-to-mut} is less than \(C \gamma^n \|\phi\|_{\cC^1}\) for any \(m\). This concludes the proof of the Corollary.
\end{proof}

In the next Lemma we establish some sort of independence, uniform in time, between the event \(\cM_{t}^n\) of surviving up to time \(n\) in the future and the event \(\cF(\cH_t)\) of ending up in the hole one iteration in the past, when both events are considered w.r.t.\! the SRB measure \(\mu_0\).

\begin{lemma}\label{lem:nasty-term}
There exist \(C, t_0>0\) such that, for any \(n \in \bN\) and \(t \in [0,t_0]\),
    \[
     \frac{\mu_0( \cF(\cH_t) \cap \cM_t^n)}{\mu_0 (\cM_t^n)} \le Ct.
    \]
\end{lemma}
\begin{proof}
Using that \(\cF(\cM_t) \cup \cF(\cH_t)\) and \(\cF^{-n-1}(\cM_t) \cup \cF^{-n-1}(\cH_t)\) are partitions of \(\cM\) and that \(\mu_0\) is \(\cF\)-invariant,
\[
\begin{split}
\mu_{0}(\cF(\cH_t) \cap \cM_t^n) &= \mu_0(\cM_t^n) - \mu_0 (\cF(\cM_t)\cap \cM_t^n) = \mu_0(\cM_t^n) - \mu_0 (\cM_t^{n+1}) \\
&= \mu_{0}(\cM_t^n \cap \cF^{-n-1}(\cH_t)).
\end{split}
\]
Hence, recalling the expression \eqref{eq:iterate-transf-op-hole} for \(\cL_t^n\),
\[
\frac{\mu_0( \cF(\cH_t) \cap \cM_t^n)}{\mu_0 (\cM_t^n)} = \frac{ \mu _0 \bigl((\Id_{\cH_t}  \circ \cF)  \circ \cF^n  \Id_{\cM_t^n}\bigr)}{\mu_0 (\cM_t^n)} = (\cL_{t}^n \mu_0) (\Id_{\cH_t}  \circ \cF).
\]
By Lemma \ref{lem:exp-conv-initial}, \(\mu_0\) is a initial regular measure. Hence, for \(t_0>0\) small enough, we can apply Lemma \ref{lem:invariance}, obtaining that \(\cL_{t}^{n}\mu_0\) is a regular measure for each \(n \in \bN_0\) and \(t \in [0,t_0]\). Moreover, \(\Id_{\cH_t} \circ \cF\) is discontinuous at most on \(\cS_{1,t}\). Therefore, by Lemma \ref{lem:convergence-discontinuous-function}, there exists a sequence \(\{\cG_k\}_k\) of regular standard families such that 
\[
(\cL_{t}^{n}\mu_0) (\Id_{\cH_t} \circ \cF) = \lim_{k \to \infty}\mu_{\cG_k} (\Id_{\cH_t} \circ \cF).
\]
Finally, by Equation \eqref{eq:hole-plus-preimage}, there exists \(C>0\) such that
\[
|\mu_{\cG_k} (\Id_{\cH_t} \circ \cF)| \le |\mu_{\cG_k} (\cH_t \cup \cF^{-1}(\cH_t))|\le Ct
\]
uniformly in \(k\). This concludes the proof.
\end{proof}

The next one is the crucial computation involving decay of correlations.

\begin{lemma}\label{lem:response-formula}
There exist \(C, t_0 >0\) and \(\gamma \in (0,1)\) such that, for any \(n \in \bN_0\), \(t \in [0,t_0]\) and \(\phi \in \cC^1(\cM)\), we have
\[
|(\cL_{t}^{n+1} \mu_0) (\phi) - (\cL_{t}^n \mu_0) (\phi)| \le C t \gamma^n \|\phi\|_{\cC^1}.
\]
\end{lemma}
\begin{proof}
Partitioning \(\cM =  \cF(\cM_t) \cup \cF(\cH_t)\) and using that \(\mu_0\) is \(\cF\)-invariant,
\[
\begin{split}
\frac{\mu_0 (\phi \circ \cF^{n}\Id_{\cM_t^{n}}) }{\mu_0 (\cM_t^{n}) } &= \frac{\mu_0 (\phi \circ \cF^{n}(\Id_{\cF(\cM_t)} + \Id_{\cF(\cH_t)})\Id_{\cM_t^n} )}{\mu_0 (\cM_t^n)} \\
&= \frac{\mu_0 (\phi \circ \cF^{n}\Id_{\cF(\cM_t) \cap \cM_t^n} )}{\mu_0 (\cM_t^n)} + \frac{\mu_0 (\phi \circ \cF^{n} \Id_{\cF(\cH_t)}\Id_{\cM_t^n} )}{\mu_0 (\cM_t^n)}\\
& = \frac{\mu_0 (\phi \circ \cF^{n+1}\Id_{\cM_t^{n+1}} )}{\mu_0 (\cM_t^n)} + \frac{\mu_0 (\phi \circ \cF^{n} \Id_{\cF(\cH_t)}\Id_{\cM_t^n} )}{\mu_0 (\cM_t^n)}.
\end{split}
\]
Therefore, for any \(\phi \in \cC^0 (\cM)\) and \(n \ge 1\),
\[
\begin{split}
(\cL_{t}^{n+1} \mu_0) &(\phi) - (\cL_{t}^n \mu_0) (\phi) = \frac{\mu_0 (\phi \circ \cF^{n+1}\Id_{\cM_t^{n+1}}) }{\mu_0 (\cM_t^{n+1}) } - \frac{\mu_0 (\phi \circ \cF^{n}\Id_{\cM_t^{n}}) }{\mu_0 (\cM_t^{n}) }\\
&=\mu_0 (\phi \circ \cF^{n+1}\Id_{\cM_t^{n+1}})\biggl(\frac{1}{\mu_0 (\cM_t^{n+1})} - \frac{1}{\mu_0 (\cM_t^{n})} \biggr) - \frac{\mu_0 (\phi \circ \cF^{n} \Id_{\cF(\cH_t)}\Id_{\cM_t^n} )}{\mu_0 (\cM_t^n)}\\
& =\frac{\mu_0 (\phi \circ \cF^{n+1}\Id_{\cM_t^{n+1}} )}{\mu_0 (\cM_t^{n+1})} \biggl(\frac{\mu_0 (\cM_t^{n}) - \mu_0 (\cM_t^{n+1})}{\mu_0 (\cM_t^{n})} \biggr) \\[10pt]
    &\hspace{5cm}- \frac{\mu_0( \cF(\cH_t) \cap \cM_t^n)}{\mu_0 (\cM_t^n)}\frac{\mu_0(\phi \circ \cF^n \Id _{\cF(\cH_t)} \Id_{\cM_t^n})}{\mu_0(\cF(\cH_t)\cap \cM_t^n)}.
\end{split}
\]
Since \(\mu_0 (\cM_t^{n}) - \mu_0 (\cM_t^{n+1}) = \mu_0(\cM_t^{n}) - \mu_0 (\cF(\cM_t) \cap \cM_t^n) = \mu_0 (\cF(\cH_t)\cap \cM_t^n)\), and recalling the definition \eqref{eq:survival-and-hole-measures} of \(\cD_t\),
\begin{equation}\label{eq:direct-computation}
\begin{split}
&(\cL_{t}^{n+1} \mu_0) (\phi) - (\cL_{t}^n \mu_0) (\phi) \\[10pt]
& =\frac{\mu_0( \cF(\cH_t) \cap \cM_t^n)}{\mu_0 (\cM_t^n)}\biggl( \frac{\mu_0 (\phi \circ \cF^{n+1}\Id_{\cM_t^{n+1}} )}{\mu_0 (\cM_t^{n+1})}- \frac{\mu_0(\phi \circ \cF^n \Id _{\cF(\cH_t)} \Id_{\cM_t^n})}{\mu_0(\cF(\cH_t) \cap \cM_t^n)} \biggr)\\[10pt]
&=\frac{\mu_0( \cF(\cH_t) \cap \cM_t^n)}{\mu_0 (\cM_t^n)} \bigl((\cL_{t}^{n+1}\mu_0)(\phi) - (\cL_{t}^n \cD_t)(\phi) \bigr).
\end{split}
\end{equation}
Therefore, by Lemma \ref{lem:nasty-term}, there exists \(C>0\) such that, for all \(n\in \bN\) and \(t \in [0,t_0]\),
\begin{equation}\label{eq:first-part-decay-auxiliary}
\begin{split}
|(\cL_{t}^{n+1} \mu_0) (\phi)& - (\cL_{t}^n \mu_0) (\phi)| \le Ct\bigl|(\cL_{t}^{n+1}\mu_0)(\phi) - (\cL_{t}^n \cD_t)(\phi) \bigr|.
\end{split}
\end{equation}
Let \(t_0 >0\) be small enough. By Lemma \ref{lem:exp-conv-initial}, \(\mu_0\) and \(\cD_t\) are initial regular measures and, by Lemma \ref{lem:invariance}, \(\cL_t \mu_0\) is a regular measure for any \(t \in [0,t_0]\). Therefore, by considering regular measures as initial regular measures (note that Proposition \ref{prop:invariance-standard-families} holds for any \(\varpi_1\), \(B_1\), \(D_1\)), we can apply Lemma \ref{lem:exp-mixing}, possibly with a bigger \(C>0\), to the measures \(\cL_t \mu_0\) and \(\cD_t\). This implies that, for some \(C>0\) and \(\gamma \in (0,1)\), we have, for any \(t \in [0,t_0]\), \(n \in \bN\) and \(\phi \in \cC^1 (\cM)\),
\[
\bigl|(\cL_{t}^{n+1}\mu_0)(\phi) - (\cL_{t}^n \cD_t)(\phi) \bigr| \le C \gamma^n \|\phi\|_{\cC^1}.
\]
This, together with \eqref{eq:first-part-decay-auxiliary}, proves the statement of the Lemma for all \(n \ge 1\). For \(n = 0\) \eqref{eq:direct-computation} does not hold and using the invariance of the measure,
\[
\begin{split}
&(\cL_{t} \mu_0) (\phi) -  \mu_0 (\phi) = \frac{\mu_0 (\phi \circ \cF \Id_{\cM_t^1})}{\mu_0(\cM_t^1)} - \mu_0 (\phi \circ \cF (\Id_{\cM_t^1} + \Id_{\cH_t \cup \cF^{-1}(\cH_t)})) \\
& = \mu_0 (\cH_t \cup \cF^{-1}(\cH_t)) \biggl(\frac{\mu_0 (\phi \circ \cF \Id_{\cM_t^1})}{\mu_0(\cM_t^1)} - \frac{\mu_0 (\phi \circ \cF \Id_{\cH_t \cup \cF^{-1}(\cH_t)})}{\mu_0 (\cH_t \cup \cF^{-1}(\cH_t))} \biggr).
\end{split}
\]
For \(t\) small enough the hole and its preimage are disjoint so that \(\mu_0 (\cH_t \cup \cF^{-1}(\cH_t)) = 2t\) and
\begin{equation}\label{eq:nasty-zero-case}
(\cL_{t} \mu_0) (\phi) -  \mu_0 (\phi) = t \biggl(2\frac{\mu_0 (\phi \circ \cF \Id_{\cM_t^1})}{\mu_0(\cM_t^1)} - \frac{\mu_0 (\phi \circ \cF \Id_{\cH_t})}{t} - \frac{\mu_0 (\phi \Id_{\cH_t})}{t} \biggr).
\end{equation}
Therefore, we obtain the bound \(|(\cL_{t} \mu_0) (\phi) -  \mu_0 (\phi)| \le 4\|\phi\|_{\cC^0} t\). This concludes the proof of the Lemma.
\end{proof}

By the previous result, it follows that the modulus of continuity is of order \(t\). Indeed,
\begin{equation}\label{eq:formal-power-series}
|\mu_t(\phi) - \mu_0 (\phi)| = \bigl|\lim_{n \to \infty} (\cL^n_{t}\mu_0)(\phi)  - \mu_0(\phi) \bigr| \le  \sum_{k=0}^{\infty} \bigl|(\cL_{t}^{k+1} \mu_0) (\phi) - (\cL_{t}^k \mu_0) (\phi)\bigr|,
\end{equation}
and by Lemma \ref{lem:response-formula} the above is of order \(t\) for \(\phi \in \cC^1(\cM)\). We now turn to the original problem of Linear response.  We start with a very mild control on the dependence of the preimages of the hole on its size.

\begin{lemma}\label{lem:neigh-sing}
    For each \(n \in \bN\) there exists a function \(h_n :[0,t_0] \to \bR^+\), \(\lim_{t \to 0}h_n (t) = 0\), such that
    \[
    \cF^{-n}(\cH_t) \subseteq \bigl[ \cS_{n} \cup \cF^{-n}(\partial \cH_t) \bigr ]_{h_n(t)}.
    \]
\end{lemma}
\begin{proof}
Let \(x \in \cF^{-n}(\cH_t)\) and let \(\bar x := \cF^{n}(x) \in \cH_t\). Let \(W_0\) be a stable curve connecting \(\bar x\) to \(\partial \cH_t\). Since \(W_0 \subseteq \cH_t\) and the vectors in the stable cone have bounded slope, \(|W_0| \le C_{\mathrm{cone}}t\) for some \(C_{\mathrm{cone}}>0\). We define \(W_k\), \(k \in \{1,...,n\}\), inductively in the following way. Consider the unique connected component \(\tilde W_k\) of \(W_k \setminus \cS_{-1}\) containing \(\cF^{-k}(\bar x)\) and set \(W_{k+1} = \cF^{-1}(\tilde W_k)\). Hence, \(\tilde W_0\) has endpoints in \(\cS_{-1} \cup \partial \cH_t\) and \(W_1 = \cF^{-1}(\tilde W_0)\) has endpoints in \(\cS_1 \cup \cF^{-1}(\partial \cH_t)\). In general, arguing by induction, one has that \(W_k\) is a stable curve containing \(\cF^{-k}(\bar x) = \cF^{n-k}(x)\) whose endpoints belong to \(\cS_{k} \cup \cF^{-k}(\partial \cH_t) \). Moreover, by Lemma \ref{lem:upp-bound-stretch-curves}, there exists \(C>0\) such that \(|W_{k+1}| = |\cF^{-1}(\tilde W_k)|\le |\cF^{-1}(W_k)| \le C\sqrt{|W_k|}\). Hence, for each \(n \in \bN\), there exists \(C_n >0\) such that for all \(t\in [0,t_0]\),
\[
|W_{n}| \le C_n |W_0|^{\frac{1}{2^n}} \le C_n (C_{\mathrm{cone}}t)^{\frac{1}{2^n}}.
\]
Because \(W_n\) is a continuous curve connecting \(x\) to \(\cS_{n} \cup \cF^{-n}(\partial \cH_t)\),
\[
\inf_{y \in \cS_{n} \cup \cF^{-n}(\partial \cH_t)}|x - y| \le  C_n (C_{\mathrm{cone}}t)^{\frac{1}{2^n}} := h_n (t),
\]
and the statement follows by the fact that \(x \in \cF^{-n}(\cH_t)\) is arbitrary and \(\lim_{t \to 0}h_n (t) =0\).
\end{proof}

Note that we didn't prove anything about the uniformity of \(h_n\) in \(n\). Hence, Lemma \ref{lem:neigh-sing} is useful only for fixed \(n\) and very small \(t\). This is also the case for the following result.

\begin{lemma}\label{lem:int-hole}
    For each \(n \ge 1\), \(\mu_{0}(\cH_t \cap \cup_{k=1}^n \cF^{-k}(\cH_t)) = o(t)\) as \(t \to 0\).
\end{lemma}
\begin{proof}
For any \(k \in \bN\), let \(h_k\) be the function given by Lemma \ref{lem:neigh-sing}, and let \(\theta_n = \max_{k \in \{1,...,n\}} h_k\). One has that \(\lim_{t \to 0} \theta_n(t) = 0\). Moreover, by Lemma \ref{lem:neigh-sing} and using that \(\cos \vf \le 1\),
\begin{equation}\label{eq:estimate-measure-intersections}
\begin{split}
        \mu_{0}\bigl(\cH_t \cap \cup_{k=1}^n \cF^{-k}(\cH_t)\bigr) &\le \frac{1}{2|\partial \cD|} m \bigl(\cH_t \cap\cup_{k=1}^{n} \bigl[\cS_{k} \cup \cF^{-k}(\partial \cH_t) \bigr ]_{h_k(t)}\bigr)\\
        &=  \frac{1}{2|\partial \cD|}  m \bigl(\cH_t \cap \bigl[\cS_{n,t} \setminus \partial \cH_t\bigr]_{\theta_n (t)} \bigr) .
\end{split}
\end{equation}
In the last inequality we used that the sum starts at \(k=1\) and not \(k=0\).
As we mentioned after Lemma \ref{lem:unstable-curves-cosine-k}, the set \(\cS_{n,t} \setminus \partial \cH_t\) consists of a finite union of smooth stable curves \(\{\alpha_j\}_j\) (except for the two horizontal singularity lines \(\cS_0 = \{\vf = \pm \pi/2\}\) which are not aligned with \(\hat \cC^s\) but nonetheless we group with the \(\{\alpha_j\}_j\)). In particular, there is a bound, uniform in \(t\), on the slope that any \(\alpha_j\) can have. I.e., the tangent lines to any \(\alpha_j\) are far from being vertical. Hence, there exists \(C_{\mathrm{cone}}>0\) such that \(m \bigl([\alpha_j]_{\theta_n(t)} \cap \cH_t\bigr) \le C_{\mathrm{cone}} t \theta_n (t)\) for any \(j\) and \(t\). Moreover, by Lemma \ref{lem:bound-number-smooth-components}, the cardinality of \(\{\alpha_j\}_j\) is also bounded by some \(M_n >0\) uniform in \(t\). Therefore,
\[
 m \bigl(\cH_t \cap \bigl[\cS_{n,t} \setminus \partial \cH_t\bigr]_{\theta_n (t)} \bigr) \le \sum_{j} m (\cH_t \cap [\alpha_j]_t ) \le M_n C t \theta_n (t),
\]
which, together with \eqref{eq:estimate-measure-intersections}, proves the statement.
\end{proof}

To compute the derivative of \(\mu_t\) we fix \(n\) and consider separately each term of the series \eqref{eq:formal-power-series} divided by \(t\). 

\begin{lemma}\label{lem:limit-single-terms}
For each \(n \ge 1\) and \(\phi \in \cC^0(\cM)\),
\[
\lim_{t \to 0}\frac{ (\cL_{t}^{n+1} \mu_0) (\phi) - (\cL_{t}^n \mu_0) (\phi)}{t} = \mu_0 (\phi) - \frac{1}{2}\int_{- \frac \pi 2}^{\frac \pi 2} \phi \circ \cF^{n+1}(r_{*}, \vf) \cos \vf d \vf,
\]
and, for \(n =0\),
\[
\lim_{t \to 0}\frac{ (\cL_{t} \mu_0) (\phi) - \mu_0 (\phi)}{t} = \mu_0 (\phi) - \frac{1}{2}\int_{- \frac \pi 2}^{\frac \pi 2}\!\!\phi (r_{*}, \vf)  \cos \vf d \vf + \mu_0 (\phi) - \frac{1}{2}\int_{- \frac \pi 2}^{\frac \pi 2} \!\!\phi \circ \cF (r_{*}, \vf)  \cos \vf d \vf.
\]
\end{lemma}
\begin{proof}
Since \(\mu_0\) is \(\cF\)-invariant and using Lemma \ref{lem:int-hole},  
\begin{equation}\label{eq:o(t)}
     \mu_0(\cF(\cH_t) \cap (\cM_{t}^n)^c ) =  \mu_0 (\cF(\cH_t) \cap \cup_{k=0}^{n}\cF^{-k}(\cH_t))= \mu_0 (\cH_t \cap \cup_{k=1}^{n+1}\cF^{-k}(\cH_t)) = o(t).
\end{equation}
Hence,
\[
\begin{split}
    \mu_0( \cF(\cH_t) \cap \cM_t^n ) &=  \mu_0 (\cF(\cH_t)) - \mu_0(\cF(\cH_t) \cap (\cM_{t}^n)^c ) =  t + o(t).
\end{split}
\]
Therefore, using also that \(\lim_{t \to 0}\mu_0 (\cM_{t}^n) = \mu_0 (\cM) = 1\), we have,
\begin{equation}\label{eq:limit-zero}
    \lim_{t \to 0} \frac{\mu_0(\cF(\cH_t) \cap \cM_t^n )}{t\mu_0 (\cM_t^n)} =  1.
\end{equation}
Moreover, by \eqref{eq:o(t)},
\[
\begin{split}
(\cL_{t}^n \cD_t)(\phi) &= \frac{\mu_0 (\phi \circ \cF^n \Id_{\cF(\cH_t) \cap \cM_t^n})}{\mu_0 (\cF(\cH_t) \cap \cM_t^n)} = \frac{\mu_0 (\phi \circ \cF^n \Id_{\cF(\cH_t)}) - \mu_0 (\phi \circ \cF^n \Id_{\cF(\cH_t) \cap (\cM_t^n)^c})}{\mu_0(\cF(\cH_t)) - \mu_0 (\cF(\cH_t) \cap (\cM_t^n)^c)}\\[9pt]
&= \frac{\mu_0 (\phi \circ \cF^n \Id_{\cF(\cH_t)}) + \|\phi\|_{\cC^0}o(t)}{t + o(t)}.
\end{split}
\]
Therefore, using again that \(\mu_0\) is \(\cF\)-invariant,
\begin{equation}\label{eq:limit-1}
\begin{split}
\lim_{t \to 0} (\cL_{t}^n \cD_t)(\phi) &= \lim_{t \to 0}  \frac{\mu_0 (\phi \circ \cF^n \Id_{\cF(\cH_t)})}{t } = \lim_{t \to 0}  \frac{\mu_0 (\phi \circ \cF^{n+1} \Id_{\cH_t})}{t } \\
&= \frac{1}{2}\int_{- \frac \pi 2}^{\frac \pi 2} \phi \circ \cF^{n+1}(r_{*}, \vf) \cos \vf d \vf.
\end{split}
\end{equation}
Note that the expression on the r.h.s.\! of the first line of the equation above is an average of \(\phi \circ \cF^{n+1}\) on the hole \(\cH_t\) centered at \(r_{*}\). In order to obtain the last equality, we approximate \(\phi \circ \cF^{n+1}\) with a continuous function on \(\cM \setminus [\cS_{n+1}]_{\ve}\) and use that fact that \(\mu_0 (\Id_{\cH_t} \cap [\cS_{n+1}]_{\ve}) \le C_{\mathrm{cone}} t \ve\) by transversality.
Finally, by Lebesgue dominated convergence theorem and the invariance of \(\mu_0\),
\begin{equation}\label{eq:limit-2}
 \lim_{t \to 0} (\cL_{t}^{n+1} \mu_0)(\phi) = \lim_{t \to 0} \int \phi \circ \cF^{n+1} \Id_{\cM_t^n}  \frac{\cos \vf d \vf dr}{2|\partial \cD|} = \mu_0 (\phi \circ \cF^{n+1}) = \mu_0(\phi),
\end{equation}
where in the second equality we used again that  \(\lim_{t \to 0}\mu_0 (\cM_{t}^n) = 1\). The expression \eqref{eq:direct-computation}
\[
\begin{split}
&(\cL_{t}^{n+1} \mu_0) (\phi) - (\cL_{t}^n \mu_0) (\phi) =\frac{\mu_0( \cF(\cH_t) \cap \cM_t^n)}{\mu_0 (\cM_t^n)} \bigl((\cL_{t}^{n+1}\mu_0)(\phi) - (\cL_{t}^n \cD_t)(\phi) \bigr),
\end{split}
\]
together with \eqref{eq:limit-zero}, \eqref{eq:limit-1} and \eqref{eq:limit-2}, proves the statement for \(n \ge 1\). The case \(n =0\) follows by \eqref{eq:nasty-zero-case} and the same arguments of above.
\end{proof}

Finally, Lemma \ref{lem:response-formula} guarantees the needed uniformity to exchange the limit with the sum in the expression for the derivative.

\begin{proof}[Proof of Theorem \ref{thm:linear-response}]
By \eqref{eq:conditional-invariant-measure} and by adding and subtracting \((\cL_{t}^k \mu_0)(\phi)\), we obtain the telescopic series
\begin{equation}\label{eq:telescopic}
\begin{split}
    \frac{\mu_t(\phi) - \mu_0(\phi)}{t} &= \frac{\lim_{n \to \infty} (\cL^n_{t}\mu_0)(\phi)  - \mu_0(\phi)}{t} = \sum_{k=0}^{\infty} \frac{ (\cL_{t}^{k+1} \mu_0) (\phi) - (\cL_{t}^k \mu_0) (\phi)}{t}.  
\end{split}
\end{equation}
By Lemma \ref{lem:response-formula}, there exist \(C>0\) and \(\gamma \in (0,1)\) such that, for \(t_0 >0\) small enough, uniformly in \(t \in [0,t_0]\),
\begin{equation}\label{eq:exp-conv-derivative}
\biggl |  \frac{ (\cL_{t}^{k+1} \mu_0) (\phi) - (\cL_{t}^k \mu_0) (\phi)}{t}\biggr | \le C \|\phi\|_{\cC^1} \gamma^k .
\end{equation}
Moreover, by Lemma \ref{lem:limit-single-terms}, for each \(k \in \bN_0\), the limit for \(t \to 0\) of the l.h.s.\! of \eqref{eq:exp-conv-derivative} exists. Therefore, we can exchange the limit in \(t\) with the series and, applying Lemma \ref{lem:limit-single-terms} for each \(k\) we obtain,
\[
\begin{split}
\lim_{t \to 0} \frac{\mu_t(\phi) - \mu_0(\phi)}{t} &=  \sum_{k=0}^{\infty} \lim_{t \to 0}\frac{ (\cL_{t}^{k+1} \mu_0) (\phi) - (\cL_{t}^k \mu_0) (\phi)}{t}\\
& = \sum_{k = 0}^{\infty} \mu_0 (\phi) - \frac{1}{2}\int_{- \frac \pi 2}^{\frac \pi 2} \phi \circ \cF^{k}(r_{*}, \vf) \cos \vf d \vf.
\end{split}
\]
This concludes the proof of the Theorem.
\end{proof}

\section{Coupling}\label{sec:coupling}
In this section we prove Proposition \ref{prop:invariance-standard-families} and Theorem \ref{thm:exp-mixing-sf}. For the reader interested in those results and not in their application to linear response or vice-versa, we remark that this section is completely independent of Section \ref{sec:linear-response}. The coupling techniques are similar to those developed in \cite{MR3556527, CANESTRARI2026111008} which use `finite time' holonomies and avoid Cantor rectangles. This is somehow convenient because the dependence on the perturbation might be easier to track in finite time arguments. Note however that the presence of a hole is a novelty in coupling and it requires additional arguments in order to implement the coupling algorithm. Fix any \(\varpi_1, B_1, D_1 >0\).
\subsection{Evolution of unstable curves and densities}\label{subsec:graph-transform}
\textit{In this section we consider a single standard pair and we study how the regularity of the unstable curve \(W\) and density \(\rho\) are affected by the dynamics.}

We start with the invariance for the support of standard pairs, i.e., of the set \(\cW(D)\) of unstable curves introduced in \eqref{eq:graph-sp-def}. We do not present a full proof of the required statement since it is well known in the literature, but not immediate. In particular, the fact that the second derivative of unstable curves remains bounded under iterations of the billiard map is not obvious given the blow-up of the derivatives of \(\cF\) near singularities. Therefore, we rely on the so-called \textit{curvature bounds}. See below for references to the relevant statement. 
\begin{lemma}\label{lem:invariance-unstable-curves}
 There exists \(D_2 >0 \) such that, for any \(W \in \cW(D_1)\) and \(n \in \bN\), there exists a finite set \(\{W_{i}\}_i \subseteq \cW(D_2)\) such that \(\cF^n(W) = \cup_i W_i \). 
\end{lemma}
\begin{proof}
    We construct \(\{W_i\}_i\) in the following way. Take the connected components \(\{\tilde W_i\}_i\) of \(W \setminus \cS_n\), which are a finite number, and set \(W_i = \cF^n (\tilde W_i)\). By the requirement on the tangent space of \(W\) and cone invariance, each tangent line at \(W_i\) is aligned with the small unstable cone field and so \(W_i\) is the graph of a function \(\vf_{W_i}\). By differentiability of \(\cF^n\) on \(\cM \setminus \cS_n\) and the fact that \(W\) is the graph of a twice differentiable function \(\vf_W\) each \(\vf_{W_i}\) is twice differentiable. The fact that \(|\vf_{W_i}''| \le D_2\) for \(D_2 >0\) big enough follows from \cite[Lemma~4.1]{MR1832968} and uses the \(\cC^3\) regularity of the boundary. Finally, we can ensure that \(|W_i| \le \delta_*\) by partitioning each \(W_i\) into several pieces if needed.
\end{proof}

\begin{remark}
We fix \(D_2\) to be the value given by Lemma \ref{lem:invariance-unstable-curves}. In particular, we will always assume that initial standard pairs are supported on unstable curves in \(\cW(D_1)\) so that their push-forwards will be supported on elements of \(\cW(D_2)\). Therefore, when this does not create confusion, we will suppress the dependence on \(D\) in \(\cW\).     
\end{remark}

The next task is to study the evolution of the density. Let \(W \in \cW\) and assume that \(W\cap \cS_1 = \emptyset\) (if not, we consider connected subsets of \(W\setminus \cS_1\)). Recall that \(I_W = \pi_r (W)\) and \((r_1, \vf_1) = \cF(r,\vf)\). We introduce the function \(r_{1,W} : I_W \to I_{\cF(W)}\) defined by
\begin{equation}\label{eq:def-map-r-coordinate}
\rw (r) = r_1 (r, \vf_W (r)).
\end{equation}
Let \(W\) be a short unstable curve such that \(\cF^n\) is smooth on \(W\). We define \(\rw^n: I_{W} \to I_{\cF^n(W)}\) as the composition
\begin{equation}\label{eq:powers-expanding-base}
r_{1,W}^{n} (r) = r_{1,\cF^{n-1}(W)} \circ r_{1, \cF^{n-2}(W)} \circ ... \circ r_{1, W}(r).
\end{equation}
Moreover, whenever \(r_{1,\cF^{-n}(W)}^n:  I_{\cF^{-n}(W)} \to I_W\) is well defined, we denote its inverse by \(\rw^{-n} :  I_{W} \to I_{\cF^{-n}(W)}\). These compositions behave essentially as expanding maps, even if a single iteration can be even contracting locally. We now prove some basic properties of these compositions. 
To compare the next statements with other literature on hyperbolic billiards, note that \(\rw' \sim \cJ_{W}\cF\), the Jacobian of the map restricted to the unstable curve \(W\) w.r.t.\! the Euclidean metric. This is the case because the small unstable cone has bounded slope and the expansion on the \(r\)-direction is comparable with the expansion of the tangent vectors in the Euclidean norm. In particular the constant in \(\sim\) depends only on the cone. In view of these facts, the results in the next three Lemmata are well known, but we include the proofs because of the slightly different conventions. Considering the derivative of \(\rw\) in place of \(\cJ_{W}\cF\) is natural for our coupling arguments, as it will be clear in the sequel. 

\begin{lemma}\label{lem:basic-FGu}
There exists \(c>0\) such that, for every \(r \in I_W\), 
\[
|\rw ' (r)| \ge \frac{c}{\cos \vf_1 (r,\vf_W(r))}.
\]
Moreover, there exists \(\Lambda > 1\) such that, for every \(n \in \bN\), one has \( | {\rw^n} '| \ge c \Lambda^n\).
\end{lemma}
\begin{proof}
By \eqref{eq:def-map-r-coordinate} and \eqref{eq:differential}, we have
\begin{equation}\label{eq:derivative-FG}
    \rw' (r)= \frac{\partial r_1 }{\partial r}(r,\vf_{W}(r)) + \frac{\partial r_{1}}{\partial \vf} (r,\vf_{W}(r)) \vf_{W}'(r) = -\frac{\tau \kappa + \cos \vf_W + \tau \vf_W '}{\cos \vf_1},
\end{equation}
where all the functions in the last expression are evaluated at \((r, \vf_W(r))\). We have that \(\cos \vf_W \) is positive and, by \eqref{eq:bound-table}, \(\tau\) and \(\kappa\) are strictly positive. Moreover, because \(W\) is an unstable curve, \(\vf_W'\) is positive as well. This proves the first statement. As for the second statement, denoting by \(W_k\) the connected components of \(\cF^k(W)\) containing \((r_k, \vf_k) = x_k = \cF^k(r, \vf_W(r))\) and by \(\tau_k = \tau(x_k)\), \(\kappa_k = \kappa (x_k)\), we have
\[
\begin{split}
    {\rw^n}' &=  \prod_{k=0}^{n-1}-\frac{\tau_k \kappa_k + \cos \vf_k + \tau_k \vf_{W_k}'}{\cos \vf_{k+1}}\\
    &=-\frac{\tau_0 \kappa_0 + \cos \vf_W + \tau \vf_W '}{\cos \vf_{n}} \prod_{k=1}^{n-1}-\frac{\tau_k \kappa_k + \cos \vf_k + \tau \vf_{W_k}'}{\cos \vf_{k}}.
\end{split}
\]
Here \(\vf_{W_k}'\) is the derivative of the graph of \(W_k\) w.r.t.\! \(r_k\). By cone invariance \(\vf_{W_k}' \ge 0\) for all \(k\). Hence, we obtain the lower bound
\[
|{\rw^n}'| \ge \tau_0 \kappa_0 \prod_{k=1}^{n-1}(1 + \tau_k \kappa_k)  \ge \frac{\tau_{\mathrm{min}}\kappa_{\mathrm{min}}}{1 + \tau_{\mathrm{max}}\kappa_{\mathrm{max}}}(1 + \tau_{\mathrm{min}}\kappa_{\mathrm{min}})^{n}.
\]
This concludes the proof of the Lemma with \(\Lambda = 1 + \tau_{\mathrm{min}}\kappa_{\mathrm{min}} >1\).
\end{proof}

The next step is studying the distortion.
\begin{lemma}\label{lem:distortion}
There exists \(C>0\) such that, for every \(r\in I_W\), 
\[
\biggl|\frac {\rw''}{{\rw'}^2} (r) \biggr| \le \frac{C}{\cos \vf_1 (r,\vf_W(r))}.
\]
\end{lemma}
\begin{proof}
Set \(\bold g(r) = (r, \vf_W(r))\). Differentiating \eqref{eq:derivative-FG} and recalling the expression \eqref{eq:differential} for the differential of \(\cF\), we find
\begin{equation}\label{eq:second-deriv-rw}
    \begin{split}
    \rw'' (r) &= \frac{\bigl(A \circ \bold g(r) + B \circ \bold g(r)\vf_{W}'(r)\bigr)}{(\cos \vf_1\circ \bold g(r))^{2}} (\cos \vf_1\circ \bold g)'(r)  \\
    &-\frac{\bigl((A\circ \bold g)' (r) + (B\circ \bold g)' (r) \vf_{W}'(r) + B\circ \bold g(r) \vf_{W}''(r)\bigr)}{\cos \vf_1\circ \bold g(r)}.
\end{split}
\end{equation}
We also have
\begin{equation}\label{eq:deriv-cos-image}
\begin{split}
(\cos \vf_1\circ \bold g)'(r) &= - (\vf_1 \circ \bold g)' (r) \sin \vf_1\circ \bold g(r)  \\
&= \frac{\sin \vf_1}{\cos \vf_1}\circ \bold g(r) \bigl(C\circ \bold g(r) + D\circ \bold g(r)\vf_W' (r)\bigr).
\end{split}
\end{equation}
By \eqref{eq:second-deriv-rw} and \eqref{eq:deriv-cos-image} and using that both \(|\vf_{W}'|\) and \(|\vf_{W}''|\) are bounded and \((A\circ \bold g)'\) and \((B\circ \bold g)'\) are bounded  by \(\mathrm{const} \times (\tau \circ \bold g)' \sim 1/(\cos\vf_1 \circ \bold g)\), we obtain that \( |\rw'' (r)| \le \mathrm{const} \times (\cos \vf_1 \circ \bold g (r))^{-3}\). This, together with the first part of Lemma \ref{lem:basic-FGu}, proves the statement. 
\end{proof}

Lemma \ref{lem:distortion} is discouraging since it hints at the fact that distortion is unbounded near grazing collisions. Homogeneity strips are a tool to control distortion. 

\begin{definition}
    We call \(W \in \cW\) weakly homogeneous if \(W \subset \bH_k\) for some \(k\). We call \(W\) \(n\)-homogeneous if \(\cF^{-k}(W)\) is weakly homogeneous for \(k \in \{0,...,n-1\}\).
\end{definition}

In practice, homogeneous curves do not ‘feel' the unbounded distortion because they are small enough compared to their distance to the singularities. The following Lemma is essentially an adaptation of \cite[Lemma 5.27]{MR2229799} to our notation. 

\begin{lemma}\label{lem:bounded-distortion}
There exists \(C >0\) such that, for every \(n \in \bN\) and every \(n\)-homogeneous \(W \in \cW\), for all \(x, y \in I_{W}\),
\[
\biggl|\ln \frac{|{\rw^{-n}}' (x)|}{|{\rw^{-n}}' (y)|}\biggr| \le C |x - y|^{\frac 1 3}.
\]
\end{lemma}
\begin{proof}
Setting \(W_{j} = \cF^{-j} (W)\) and by the chain rule,
\[
\begin{split}
    \ln \frac{|{\rw^{-n}}' (x)|}{|{\rw^{-n}}' (y)|} = \sum_{j=1}^{n} \ln |r_{1,W_{j}}' \circ \rw^{-j}(y)| - \ln |r_{1, W_{j}}' \circ \rw^{-j}(x)|. 
\end{split}
\]
Moreover, 
\[
\begin{split}
\bigl |\ln | r_{1,W_j}' \circ &\rw^{-j}(y)| - \ln |r'_{1,W_j} \circ \rw^{-j}(x)| \bigr | \\
&=\bigl |\ln |r_{1,W_j}'\circ r_{1,W_{j-1}}^{-1}(\rw^{-j+1}(y))| - \ln |r_{1,W_j}'\circ r_{1,W_{j-1}}^{-1} (\rw^{-j+1}(x))| \bigr|.
\end{split}
\]
By the intermediate value theorem, the quantity above is less than
\[
| \rw^{-j+1}(y) -  \rw^{-j+1}(x)|\max_{I_{W_j}} \biggl|\frac{ r_{1,W_{j}}''}{(r_{1,W_{j}}')^2}\biggr|.
\]
(Note that by \eqref{eq:derivative-FG} the derivative \(r_{1,W}'\) is always strictly negative so that the function the function with the absolute value is still differentiable). Using that \(| \rw^{-j+1}(y) -  \rw^{-j+1}(x)| \le |I_{W_{j-1}}| \le |\cF^{-j+1}(W)|\) and Lemma \ref{lem:distortion}, we can further bound the expression above with
\[
| \rw^{-j+1}(y) -  \rw^{-j+1}(x)|^{\frac 1 3} \frac{|\cF^{-j+1}(W)|^{\frac 2 3}}{\min_{\vf \in \pi_{\vf} (\cF^{-j+1}(W))} \cos \vf} \le C| \rw^{-j+1}(y) -  \rw^{-j+1}(x)|^{\frac 1 3},
\]
for some \(C>0\). In the inequality we used Lemma \ref{lem:unstable-curves-cosine-k} together with the fact that \(\cF^{-j+1}(W)\) is weakly-homogeneous. Finally, by the second part of Lemma \ref{lem:basic-FGu}, for some \(\Lambda>1\),
\[
\begin{split}
     \biggl|  \ln \frac{|{\rw^{-n}}' (x)|}{|{\rw^{-n}}' (y)|}\biggr| &\le  C \sum_{j=1}^{n} | \rw^{-j+1}(y) -  \rw^{-j+1}(x)|^{\frac 1 3} \le C \sum_{j=1}^{n} \Lambda^{\frac{-j+1}{3}}|x-y|^{\frac 1 3} \le C |x-y|^{\frac 1 3}.
\end{split}
\]
This concludes the proof of the Lemma.
\end{proof}

We now consider the evolution of densities supported on unstable curves and prove that densities of some regularity are preserved by the dynamics. We also observe that densities with arbitrarily bad regularity recover exponentially fast. For any \(n\)-homogeneous unstable curve \(W\) and \(\rho: I_{\cF^{-n}(W)} \to \bR^+\), we let \(\cL_W^n \rho : I_W \to \bR^+\) be
\begin{equation}\label{eq:transfer-op}
    \cL_W^n \rho =   \rho \circ \rw^{-n} |{\rw^{-n}}'|.
\end{equation}
This defines a rule of evolution for densities supported on unstable curves. We may drop \(W\) from the notation and write simply \(\cL^n \rho\) when this does not create confusion. Recall \(\cC(\varpi)\) from \eqref{eq:cones-def}. In the following Lemma we prove that \(\cL\) leaves invariant the cone \(\cC(\varpi)\) with \(\varpi\) big enough.

\begin{lemma}\label{lem:distortion-bound-density}
There exist \(Q_1, Q_2 >0 \), \(\Lambda_1 >3\) and \(\bar n \in \bN\) such that, for any \(n \ge \bar n\), \(\rho \in \cC(\varpi)\):
\begin{itemize}
    \item[a)] If \(W\) is \(1\)-homogeneous and \(\rho: I_{\cF^{-1}(W)} \to \bR^+\), \(\cL\rho \in \cC( Q_{2} \varpi + Q_1)\);
    \item[b)] If \(W\) is \(n\)-homogeneous and \(\rho: I_{\cF^{-n}(W)} \to \bR^+\), \(\cL^{n} \rho \in \cC(\Lambda_1^{-1}\varpi + Q_1)\).
\end{itemize}
\end{lemma}
\begin{proof}
Recalling \eqref{eq:transfer-op} we obtain, for any \(n \in \bN\), \(x,y \in I_{W}\),
    \[
    \frac{\cL^n \rho (x)}{\cL^n \rho (y)} = \frac{\rho \circ \rw^{-n}(x)}{\rho \circ \rw^{-n}(y)}  \frac{|{\rw^{-n}}' (x)|}{|{\rw^{-n}}' (y)|}.
    \]
Let \(\bar n \in \bN\) be such that, according to Lemma \ref{lem:basic-FGu}, it holds \( |{\rw^{-\bar n}}'| \le C \Lambda^{-\bar n} \le \Lambda_1^{-1} < 1/3 \). Let \(W\) be \(n\)-homogeneous and pick \(x,y \in I_{W}\). By Lemma \ref{lem:bounded-distortion} and the equation above, and using that \(\rho \in \cC(\varpi)\), for \(Q_1 \in \bR^+\) big enough and for any \(n \ge \bar n\), 
\begin{equation}\label{eq:invariance-densities-final}
\begin{split}
     \biggl|\ln \frac{\cL^{n} \rho (x)}{\cL^{n} \rho (y)}\biggr| &\le \varpi |\rw^{-n}(x) - \rw^{- n}(y)|^{\frac{1}{3}} + Q_1|x-y|^{\frac{1}{3}}\\
     &\le (\Lambda_{1}^{-1} \varpi + Q_1)|x-y|^{\frac{1}{3}}.
\end{split}
\end{equation}
This concludes the proof of the second part of the Lemma. As for the first part, by the first line of \eqref{eq:invariance-densities-final} in the case \(n = 1\), we have,
\[
\begin{split}
     \biggl|\ln \frac{\cL \rho (x)}{\cL \rho (y)}\biggr| &\le \varpi |\rw^{-1}(x) - \rw^{-1}(y)|^{\frac{1}{3}} + Q_1|x-y|^{\frac{1}{3}}\\
     &\le (Q_2 \varpi + Q_1)|x-y|^{\frac{1}{3}},
\end{split}
\]
where we used the lower bound \(|\rw'| > c\) in Lemma \ref{lem:basic-FGu}. This concludes the proof.
\end{proof}

In order to have good distortion properties, we will always consider the functions \(\cL^n \rho\) restricted to \(n\)-homogeneous unstable curves. This is a consequence of how we define the evolution of standard families in Section \ref{subsec:standard-families}, treating \(\bS\) as real discontinuities.  Let \(\Lambda_1\), \(Q_1, Q_{2}\) and \(\bar n\) be given by Lemma \ref{lem:distortion-bound-density}. Set
\begin{equation}\label{eq:contstant-omega2-regularity}
\varpi_2 = 100(Q_2 + Q_1 +1)^{\bar n}\max\biggl\{\frac{Q_1}{1 - \Lambda_1^{-1}}, \varpi_1\biggr\}.
\end{equation}
As a Corollary of Lemma \ref{lem:distortion-bound-density}, we obtain invariance of regular densities. In the following statement, whose proof follows by a direct computation and is omitted, we let \(W\) be an \(n\)-homogeneous unstable curve and \(\rho\) be a density on \(I_{\cF^{-n}(W)}\).

\begin{corollary}\label{cor:useful-density}
There exists \(\Omega>1\) such that:
    \begin{itemize}
        \item[a)] For any \(\rho \in \cC(\varpi_1)\), \(\cL^n \rho \in \cC(\varpi_2)\) for all \(n \in \bN\);
        \item[b)] For any \(\rho \in \cC(\varpi)\) with \(\varpi>\varpi_2/2\), then \(\cL^n \rho \in \cC(\varpi /2)\) for all \(n \ge \bar n\);
        \item[c)] For any \(\rho \in \cC(\varpi)\), \(\cL^{n} \rho \in \cC(\max\{\Omega\varpi, \varpi_2\})\) for all \(n\in \bN\).
    \end{itemize}
\end{corollary}

As a byproduct of our coupling argument we will be forced to consider densities that lie in \(\cC(\varpi)\) with \(\varpi\) arbitrarily big. Thanks to part b) of Corollary \ref{cor:useful-density} densities recover exponentially fast. However, in this transient, the estimates for the evolution of \(\cZ\) would be out of control if we were to rely only on the cones \(\cC(\varpi)\). To overcome this problem, we will consider yet another type of cone given by the ratio \(\max \rho /\min \rho\). This cone has the advantage on the one hand to bound the oscillations of the density \(\rho\) and on the other hand to be ‘less demanding' than \(\cC(\varpi)\) so that the bad densities that we will consider have a max/min ratio of order one. We cannot rely on those cones only because there is no recovery. However, they have the important property of being `almost preserved' by the dynamics.

\begin{lemma}\label{lem:funny-cone}
    There exists \( Q_3 >0 \) such that, for any \(n \in \bN\), \(W \in \cW\) \(n\)-homogeneous and \(\rho: I_{\cF^{-n}(W)} \to \bR^+\),
    \[
    \frac{\max_{I_W} \cL^n \rho}{\min_{I_W} \cL^n \rho} \le  Q_3 \frac{\max_{I_{\cF^{-n}(W)}} \rho}{\min_{\cF^{-n}(W)} \rho}.
    \]
\end{lemma}
\begin{proof}
We have 
\[
\frac{\max_{I_W} \cL^n \rho}{\min_{I_W} \cL^n \rho} = \frac{\max_{I_W} \rho \circ \rw^{-n} |{\rw^{-n}}'|}{\min_{I_W} \rho \circ \rw^{-n} |{\rw^{-n}}'|} \le \frac{\max_{I_{\cF^{-n}(W)}} \rho}{\min_{I_{\cF^{-n}(W)}} \rho} \frac{\max_{I_W} |{\rw^{-n}}'|}{\min_{I_W} |{\rw^{-n}}'|}.
\]
By Lemma \ref{lem:bounded-distortion}, there exists \(C>0\) such that the ratio between the derivatives in the r.h.s.\! is less than \(e^{C |I_W|^{1/3}} \le e^{C |W|^{1/3}} \le e^{C\delta_*^{1/3}} \le e^{C}:= Q_3\). This concludes the proof of the Lemma.
\end{proof}

We conclude this section with a classic for cones. \begin{lemma}\label{lem:classical-cones}
Let \(\rho_1, \rho_2 \in \cC(\varpi)\), \(\varpi >0\), defined on \(I\). Then \(\min\left\{\rho_1, \rho_2\right\} \in \cC(\varpi)\) and, if \(\rho_1 > \rho_2\), then 
\[
\rho_1-\rho_2 \in \cC\biggl(\varpi\biggl(1 + 2 \sup_{I}( \rho_1/\rho_2 -1)^{-1}\biggr)\biggr).
\]
\end{lemma}
\begin{proof}
For \(x,y \in I\),
\[
\begin{split}
 \frac{\min\left\{\rho_1(x), \rho_2(x)\right\}}{\min\left\{\rho_1(y), \rho_2(y)\right\}} \le  \frac{\min\{\rho_1(y)e^{\varpi |x-y|^{\frac 1 3}}, \rho_2(y)e^{\varpi |x-y|^{\frac 1 3}}\}}{\min\left\{\rho_1(y), \rho_2(y)\right\}} \le e^{\varpi |x-y|^{\frac 1 3}},
\end{split}
\]
proving the first part. And,
\[
\begin{split}
\frac{\rho_1(x) -\rho_2(x)}{\rho_1(y) - \rho_2(y)} &\le \frac{e^{\varpi |x-y|^{\frac 1 3}}\left(\rho_1(y) - e^{-2\varpi |x-y|^{\frac 1 3}} \rho_2(y)\right)}{\rho_1(y) - \rho_2(y)} \\
&\le e^{\varpi |x-y|^{\frac 1 3}}\left(1 + 2\varpi |x-y|^{\frac 1 3}\frac{\rho_2(y)}{\rho_1(y) - \rho_2(y)} \right) \\
& \le e^{\varpi |x-y|^{\frac 1 3}\left(1 + 2 \sup_{I} \frac{\rho_2}{\rho_1 - \rho_2}\right) }.
\end{split}
\]
This concludes the proof.
\end{proof}

\subsection{Standard families and their evolution}\label{subsec:standard-families}
\textit{In this section we define the evolution of standard families.} 

We find it convenient to define the action of \(\cF\), \(\hat \cF_t\) and \(\cL_t\) on standard families in a canonical way. We consider three possible evolutions for a standard family. The first is induced by the standard billiard map \(\cF\). The second consists of evolving via the map \(\hat{\cF}_t\) which acts as \(\cF\) but has the additional discontinuities \(\Xi_{n,t}\) due to the boundary of the hole. The third evolution \(\cL_t\) is given by the leaky system and we renormalize by the surviving mass. We slightly abuse the notation by using the same symbol \(\cL_t\) for both the conditional evolution of measures \eqref{eq:iterate-transf-op-hole} and of standard families. To make this precise, we first consider a single standard pair \(\ell = (W, \rho)\). Let \(\{W_{j}\}\) be the connected components of \(W \setminus \cS_n^{\bH}\), \(\{\hat{W}_j\}\) the connected components of \(W \setminus \cS_{n,t}^{\bH}\) and \(\{V_{j}\}\) the connected components of \((W  \setminus \cS_{n,t}^{\bH}) \cap \cM_t^n\). Note that the \(n^{th}\) images of all these curves are \(n\)-homogeneous. Set 
\begin{equation}\label{eq:density-weights-evolution}
\begin{split}
&\rho_{A,j} = \frac{\cL^n (\rho \Id_{I_{W_j}})}{\int_{I_{W_j}} \rho}, \quad p_{A,j} = \int_{I_{W_j}}  \rho, \quad \rho_{B,j} = \frac{\cL^n (\rho \Id_{I_{\hat W_j}})}{\int_{I_{\hat W_j}} \rho}, \quad p_{B,j} = \int_{I_{\hat W_j}}  \rho,\\
&\hspace{3cm}\rho_{C,j} = \frac{\cL^n (\rho \Id_{I_{V_j}})}{\int_{I_{V_j}} \rho}, \quad p_{C,j} =\frac{ \int_{I_{V_j}}  \rho}{ \mu_{\ell}(\cM_t^n)}.
\end{split}
\end{equation}
These are the densities and weights of the evolved standard pairs. More precisely, we set \(\cF^n \ell = \{(p_{A,j}, \cF^n (W_j), \rho_{A,j})\}_j\), \(\hat \cF_t^n \ell = \{(p_{B,j}, \cF^n (\hat W_j), \rho_{B,j})\}_j\) and \(\cL_t^n \ell = \{(p_{C,j}, \cF^n (V_j), \rho_{C,j})\}_j\).
Note the total remaining mass in the denominator for the weights of \(\cL_t^n \ell\), which makes \(\mu_{\cL_t^n \ell}\) a probability measure. It is possible that \(|\cF^n(W_j)| \ge \delta_{*}\) for some \(W_{j}\) (or \(\hat W_{j}\), \(V_{j}\)). In this case, we partition \(\cF^n(W_{j})\) into at most \(\lfloor |\cF^n(W_{j})|/\delta_{*} \rfloor + 1\) pieces which become the defining unstable curves in \(\cF^n \ell\). In this case, the weights \(p_{A,j}\) are the integral of \(\rho\) over the \(n^{\text{th}}\)-preimage of these new pieces. Finally, given a standard family \(\cG = \{(p_k, W_k, \rho_k)\}\), we set \(\cF^n \cG = \{(p_k p_{A,j,k}, \cF^n(W_{k,j}), \rho_{A,k,j})\}\) where \(p_{A,j,k}\), \(W_{k,j}\) and \(\rho_{A,k,j}\) are defined as above, only relatively to the standard pair \(\ell_k = (W_k, \rho_k)\) in place of \(\ell = (W, \rho)\). The standard families \(\hat \cF_t^n \cG\) and \(\cL_t^n \cG\) are defined analogously with the only caveat that \(p_{C,j,k}\) need to be divided by the surviving mass \(\mu_{\cG}(\cM_t^n)\) and not just the mass lost by the standard pair \(\ell_k = (V_k, \rho_k)\). (This happens because \(\cL_t\) is not a linear operator). The next Lemma clarifies the raison d’être of these definitions.
\begin{lemma}\label{lem:push-forward-evolution}
    For any \(n \in \bN\) and standard family \(\cG\), \(\phi \in \cC^0 (\cM)\) and \(t \in [0,t_0]\),
    \[
   (\cL_{t}^n\mu_{\cG})(\phi) = \mu_{\cL_t^n \cG}(\phi) \quad \text{and} \quad \mu_{\cG}(\phi \circ \cF^n) = \mu_{\cF^n\cG}(\phi) = \mu_{\hat\cF_t^n\cG}(\phi).
    \]
\end{lemma}
\begin{proof}
Let \(\cG = \{(p_k, \ell_k)\}\) and \(\ell_k = (W_k, \rho_k)\). Recalling that \(W_k \cap \cM_t^n = \cup_j V_{k,j}\),
\[
\begin{split}
    \mu_{\ell_{k}} (\phi \circ \cF^n  \Id_{\cM_t^n}) = \sum_{j} \int_{I_{V_{k,j}}} \phi \circ \cF^n (r, \vf_{V_{k,j}}(r)) \rho_k(r) dr.
\end{split}
\]
By the change of variable \(r \to r'\) specified by \(\cF^n (r, \vf_{V_{k,j}}(r)) = (r', \vf_{\cF^n(V_{k,j})}(r'))\), we obtain
\begin{equation}\label{eq:single-sp-push-forward}
\begin{split}
       \mu_{\ell_{k}} (\phi \circ &\cF^n  \Id_{\cM_t^n}) = \sum_{j} \int_{I_{\cF^n(V_{k,j})}} \phi(r', \vf_{\cF^n(V_{k,j})}(r') ) \cL^n \rho_k (r') dr'\\
       &=\sum_{j} \int_{I_{V_{k,j}}} \rho_k\int_{I_{\cF^n(V_{k,j})}} \phi(r', \vf_{\cF^n(V_{k,j})}(r')) \frac{\cL^n \rho_k (r') dr'}{\int_{I_{\cF^n(V_{k,j})}} \cL^n \rho_k},
\end{split}
\end{equation}
where in the equality we divided and multiplied by \(\int_{I_{V_{k,j}}} \rho_k = \int_{I_{\cF^n(V_{k,j})}} \cL^n \rho_k\). By \eqref{eq:single-sp-push-forward}, and recalling \eqref{eq:density-weights-evolution} and the definitions given afterward,
\[
(\cL_{t}^n \mu_{\cG})(\phi) = \frac{\mu_{\cG}(\phi \circ \cF^n  \Id_{\cM_t^n})}{\mu_{\cG}(\cM_t^n)} = \frac{\sum_{k} p_k \mu_{\ell_{k}} (\phi \circ \cF^n  \Id_{\cM_t^n})}{\mu_{\cG}(\cM_t^n)} = \mu_{\cL_t^n \cG}(\phi).
\]
This proves the first part of the statement. The fact that \(\mu_{\cG}(\phi \circ \cF^n) = \mu_{\cF^n\cG}(\phi)\) follows from the first equation in the statement in the case \(t =0\) and the fact that \(\mu_{\cF^n\cG}(\phi) = \mu_{\hat\cF_t^n\cG}(\phi)\) follows from the fact that the maps \(\hat \cF_t\) and \(\cF\) act in the same way at the level of measures.
\end{proof}

\subsection{Invariance}\label{subsec:invariance}\textit{In this section we establish a Growth Lemma: this is an important tool that guarantees that expansion of unstable curves beats their fragmentation. We prove a Growth Lemma for \(\hat \cF_t\) and a weaker version for \(\cL_t\). An important difference between the two statements is that, for the dynamics with the hole, we can show invariance but we cannot guarantee that arbitrarily small unstable curves become regular families after enough iterations. The physical interpretation is clear: a small piece could just end up in the hole before recovering. We use these results to prove Proposition \ref{prop:invariance-standard-families} about invariance of standard families.}

Recall the definition of the boundary \(\cZ\) from equation \eqref{eq:boundary-def}. The following simple Lemma clarifies in which sense \(\cZ\) determines the proportion of long unstable curves in a family.
\begin{lemma}\label{lem:meaning-cZ}
For any \(\zeta >0\) and standard family \(\cG = \{(p_j, \ell_j)\}\), we have
    \[
    \sum_{j: |W_j| \le \zeta } p_j \le   \cZ(\cG)\zeta .
    \]
\end{lemma}
\begin{proof}
We have
    \[
  \sum_{j: |W_j| \le \zeta } p_j \le \sum_{j: |W_j| \le \zeta } p_j \frac{\zeta}{|W_j|}   =  \zeta \cZ(\cG).
    \]
\end{proof}
We now introduce a result from \cite{MR3213499}. For an unstable vector \((dr, d\vf) \in \hat \cC^u\), set
\begin{equation}\label{eq:adapted-norm}
\|(dr,d\vf)\|_{*} = \biggl(\kappa + \biggl|\frac{d\vf}{dr}\biggr|\biggr) \frac{\sqrt{dr^2 + d\vf^2}}{\sqrt{1 + \biggl(\frac{d\vf}{dr}\biggr)^2}}.
\end{equation}
By \cite[Exercise~5.55]{MR2229799}, \(\|\cdot\|_{*}\) is an adapted norm for the billiard map. Moreover, since \(C_{\mathrm{cone}}^{-1} \le |d\vf/dr| \le C_{\mathrm{cone}}\) for vectors in the unstable cone, the metric \eqref{eq:adapted-norm} is equivalent to the Euclidean metric in measuring the length of unstable curves. Recall the notation \(\hat W_j\) for the connected components of \(W \setminus \cS_{n,t}^{\bH}\) of an unstable curve \(W\) and denote by \(\hat{W}^n_j = \cF^n(\hat W_j) = \cF_t^n (\hat W_{j})\). Let also  \(|J_{\hat W_j}  \cF^{-n}|_{*}\) denote the smallest contraction on \(\hat W^n_j\) under \(\cF^{-n}\) in the metric induced by the adapted norm \eqref{eq:adapted-norm}. 
\begin{lemma}[{\cite[Lemma~8.4]{MR3213499}}]\label{lem:n-step-contraction-hole-D}
There exists \(\theta_{*} < 1\), \(B_0 >0 \) and a sequence \(\delta_n \searrow 0\) such that
\[
\sup_{W \in \cW, \text{ } |W| \le \delta_n} \sum_{j} |J_{\hat W^n_j}  \cF^{-n}|_{*} \le (1 + n(B_0 -1)) \theta_{*}^n.
\]
\end{lemma}

Here, the result is stated for powers \(\cF^{-n}\) of the inverse map and unstable curves, while the original statement in \cite{MR3213499} has \(\cF^n\) in place of \(\cF^{-n}\) and stable curves in place of unstable curves. The two statements are equivalent due to the time reversibility of the billiard. Moreover \(\theta_{*}\) is characterized as the one-step contraction of \(\cZ\) and \(B_0\) as the number of pieces in which an unstable (stable in the original reference) curve can be cut by \(\partial \cH_t\). In our case \(B_0 = 3\) but we won't need to be that precise. Fix some \(\gamma \in (0,1)\).
\begin{lemma}[Growth Lemma]\label{lem:invariance1}
There exist \(n_* \in \bN\), \(Z_0 >0\) such that the following is true. Let \(\cG = \{(p_k, W_k, \rho_k)\}\) be a standard family such that, for every \(k\), \(\max \rho_{k} /\min \rho_{k} \le \funny\). Then, for all \(t \in [0,t_0]\) and \(p \in \bN\),
\[
\cZ(\hat\cF^{(p+1)n_*}_{t} \cG) \le \gamma \cZ(\hat\cF^{pn_*}_{t}\cG) + Z_0.
\]
Moreover, there exists \(Z_1 >1\) such that, for all \(p \in \bN\) and \(t \in [0,t_0]\),
\[
\cZ(\hat \cF_t^{p+1} \cG) \le Z_1 \cZ(\hat\cF_t^p\cG).
\]
\end{lemma}
\begin{proof}
Because of the assumption on the max-min ratio of the initial densities and by Lemma \ref{lem:funny-cone}, the standard families \(\hat \cF_t^p\cG\) support densities \(\rho'_j\) with \(\max\rho'_j/\min \rho'_j \le \funny Q_3\) for all \(p\in \bN\). (Recall that by definition the unstable curves in \(\hat \cF_t^p\cG\) are \(p\)-homogeneous). Hence, without loss of generality, we prove the Lemma in the case \(p=0\) assuming that
\[
\max\rho_k/\min \rho_k \le \funny Q_3
\]
for all \(k\). For any \(W_k\) denote by \(\{\hat{W}_{k,j}\}\) the connected components of \(W_k \setminus \cS_{n,t}^{\bH}\) and by \(\hat W_{k,j}^n = \cF^n (\hat W_{k,j})\). Call \(\tilde \cG_n\) the corresponding ‘extended' standard families composed of the curves \(\{\hat W_{k,j}^n\}\). \(\tilde \cG_n\) are not actual standard families as their unstable curves may be longer than \(\delta_{*}\). We will partition each curve in \(\tilde \cG_n\) shortly and estimate the effect on the boundary. For the moment, we observe that
\begin{equation}\label{eq:first-tildeG}
\cZ(\tilde \cG_n) = \sum_k p_k \sum_{j}\int_{I_{\hat W_{k,j}}} \rho_k \frac{1}{|\hat W_{k,j}^n|} \le C_{\mathrm{metric}}\sum_k p_k \sum_{j}\int_{I_{\hat W_{k,j}}} \rho_k \frac{ |J_{\hat W^n_{k,j}}  \cF^{-n}|_{*}}{|\hat W_{k,j}|},
\end{equation}
where the constant \(C_{\mathrm{metric}}>0\) is given by the equivalence between the Euclidean metric \(|\cdot|\) and the adapted metric \(|\cdot|_{*}\) in \eqref{eq:adapted-norm} on vectors in the small unstable cone. Moreover, using that \(\rho_k\) are probability densities on \(I_{W_k}\) and \(\max \rho_{k} /\min \rho_{k} \le \funny Q_3 \), we have,
\begin{equation}\label{eq:unifor-density-effect}
 \int_{I_{\hat W_{k,j}}} \rho_k \le \funny Q_3  \frac{|I_{\hat W_{k,j}}|}{|I_{W_{k}}|}.
\end{equation}
Therefore, since \(|I_{\hat{W}_{k,j}}| \le |\hat W_{k,j}|\) and \(|W_k| \le C_{\mathrm{cone}} |I_{W_{k}}|\),
\begin{equation}\label{eq:intermidiate-ZtildeG}
\begin{split}
\cZ(\tilde \cG_n) &\le C_{\mathrm{metric}}\funny Q_3 \sum_k p_k \sum_{j}\frac{|I_{\hat W_{k,j}}|}{|I_{W_{k}}|}\frac{ |J_{\hat W^n_{k,j}}  \cF^{-n}|_{*}}{|\hat W_{k,j}|} \\
&\le  C_{\mathrm{metric}}\funny Q_3 \sum_k p_k \frac{1}{|I_{W_{k}}|} \sum_{j}|J_{\hat W^n_{k,j}}  \cF^{-n}|_{*}\\
&\le \funny Q_3 C_{\mathrm{cone}}C_{\mathrm{metric}} \sum_k p_k \frac{1}{|W_{k}|} \sum_{j}|J_{\hat W^n_{k,j}}  \cF^{-n}|_{*}.
\end{split}
\end{equation}
Recall that \(\gamma \in (0,1)\) is the sought contraction. Let \(n_*\in \bN\) be the smallest natural number such that, according to Lemma \ref{lem:n-step-contraction-hole-D}, for all \(|W_k|\le \delta_{n_*}\), we have
\begin{equation}\label{eq:tildegamma-contraction}
\begin{split}
 \funny Q_3 C_{\mathrm{cone}}&C_{\mathrm{metric}} \sum_{j} |J_{\hat W^{n_*}_{k,j}}  \cF^{-n_*}|_{*} \\
 &\le  \funny Q_3 C_{\mathrm{cone}}C_{\mathrm{metric}} (1 + n_*(B_0 -1)) \theta_{*}^{n_*} \le \gamma/( Q_3 \funny C_{\mathrm{cone}}) .
\end{split}
\end{equation}
Here, \(\delta_{n_*}\) is given by Lemma \ref{lem:n-step-contraction-hole-D} and we finally fix the parameter \(\delta_{*} = \min \{\delta_{n_*},1\}\) in the Definition \eqref{eq:graph-sp-def} of unstable curves. By \eqref{eq:intermidiate-ZtildeG} and \eqref{eq:tildegamma-contraction},
\begin{equation}\label{eq:contraction-Z}
\begin{split}
\cZ(\tilde \cG_{n_*}) \le \funny Q_3 C_{\mathrm{cone}}C_{\mathrm{metric}}  (1 + n_*(B_0 &-1)) \theta_{*}^{n_*}  \cZ(\cG) \le  \frac{\gamma}{Q_3 \funny C_{\mathrm{cone}}} \cZ(\cG).
\end{split}
\end{equation}
It remains to relate \(\cZ(\tilde \cG_{n_*})\) to \(\cZ(\hat \cF_t^{n_*} \cG)\). Let \(\{\hat W^{n_*}_{k,j,p}\}\) a partition of \(\hat W_{k,j}^{n_*}\) in components not bigger than \(\delta_{*}\). These are the unstable curves defining \(\hat \cF_t^{n_*} \cG\). We can ensure that for each \(k,j\) the number of derived components is less than \(|\hat W_{k,j}^{n_*}|/\delta_{*} + 1\). As we remarked at the beginning of the proof \(\max_{I_{\hat W^{n_*}_{k,j}}} \cL^{n_*} \rho_k /\min_{I_{\hat W^{n_*}_{k,j}}} \cL^{n_*} \rho_k \le \funny Q_3\) and using that \(\hat W^{n_*}_{k,j,p} \subseteq \hat W^{n_*}_{k,j}\), we have
\[
\begin{split}
\int_{I_{\hat W^{n_*}_{k,j,p}}} \cL^{n_*} \rho_k &\le \max_{I_{\hat W^{n_*}_{k,j,p}}} \cL^{n_*} \rho_k |I_{\hat W^{n_*}_{k,j,p}}| \le  \max_{I_{\hat W^{n_*}_{k,j}}} \cL^{n_*} \rho_k |I_{\hat W^{n_*}_{k,j,p}}|  \le \funny Q_3 \min_{I_{\hat W^{n_*}_{k,j}}} \cL^{n_*} \rho_k  |I_{\hat W^{n_*}_{k,j,p}}| \\
&\le \funny Q_3 \frac{|I_{\hat{W}^{n_*}_{k,j,p}}|}{|I_{\hat{W}^{n_*}_{k,j}}|} \int_{I_{\hat{W}^{n_*}_{k,j}}}\cL^{n_*} \rho_k.
\end{split}
\]
The expression above gives an estimate of the mass carried by each unstable curve in \(\hat \cF_t^{n_*} \cG\) after partitioning. Since \(|I_{\hat{W}^{n_*}_{k,j,p}}| \le |\hat W^{n_*}_{k,j,p}|\) and \(|\hat{W}^{n_*}_{k,j}| \le C_{\mathrm{cone}} |I_{\hat{W}^{n_*}_{k,j}}|\),
\[
\begin{split}
\cZ(\hat \cF_t^n \cG) &= \sum_k p_k \sum_{j}\sum_p\int_{I_{\hat W^{n_*}_{k,j,p}}} \cL^{n_*} \rho_k \frac{1}{|\hat W_{k,j,p}^{n_*}|} \\
&\le \funny Q_3 \sum_k p_k \sum_{j}\sum_p  \frac{|I_{\hat{W}^{n_*}_{k,j,p}}|}{|I_{\hat{W}^{n_*}_{k,j}}|} \int_{I_{\hat{W}^{n_*}_{k,j}}}\cL^{n_*} \rho_k \frac{1}{|\hat W^{n_*}_{k,j,p}|}\\
&\le \funny Q_3  C_{\mathrm{cone}} \sum_k p_k \sum_{j} \sum_p  \frac{1}{|\hat{W}^{n_*}_{k,j}|} \int_{I_{\hat{W}^{n_*}_{k,j}}}\cL^{n_*}\rho_k.
\end{split}
\]
Since \(\sum_p 1 \le |\hat W_{k,j}^{n_*}|/\delta_{*} + 1\) and \(\int_{I_{\hat{W}^{n_*}_{k,j}}}\cL^{n_*} \rho_k = \int_{I_{\hat{W}_{k,j}}}\rho_k\), we obtain the following estimate,
\[
\begin{split}
   &\cZ(\hat \cF_t^{n_*} \cG) \le  \funny Q_3 C_{\mathrm{cone}} \sum_k p_k \sum_{j} \biggl(\frac{|\hat W^{n_*}_{k,j}|}{\delta_{*}} +1 \biggr) \frac{1}{|\hat{W}^{n_*}_{k,j}|} \int_{I_{\hat{W}^{n_*}_{k,j}}}\cL^{n_*} \rho_k \\
   &\le \funny Q_3  C_{\mathrm{cone}} \biggl[ \sum_k p_k \sum_{j} \frac{1}{|\hat{W}^{n_*}_{k,j}|} \int_{I_{\hat{W}_{k,j}}}\rho_k +  \delta_{*}^{-1} \sum_k p_k \sum_{j} \int_{I_{\hat{W}_{k,j}}}\rho_k \biggr].
    \end{split}
\]
Using the first equality in \eqref{eq:first-tildeG} and the fact that \(\sum_k p_k \sum_{j} \int_{I_{\hat{W}_{k,j}}}\rho_k = 1\), we obtain
\begin{equation}\label{eq:chopping-estimate}
    \cZ(\hat \cF_t^{n_*} \cG)  \le  \funny Q_3 C_{\mathrm{cone}}\cZ(\tilde \cG_{n_*}) + \frac{\funny Q_3 C_{\mathrm{cone}}}{\delta_{*}}.
\end{equation}
Therefore, by \eqref{eq:contraction-Z} and \eqref{eq:chopping-estimate}, setting \(Z_0 = \frac{\funny Q_3 C_{\mathrm{cone}}}{\delta_{*}}\), we obtain that \(\cZ(\hat \cF_t^{n_*} \cG)  \le  \gamma \cZ(\cG) + Z_{0}\). As for the second part of the statement, by \eqref{eq:chopping-estimate} and the first inequality in \eqref{eq:contraction-Z} with \(1\) in place of \(n_*\) we obtain
\[
\cZ(\hat \cF_t \cG)  \le (\funny Q_3 C_{\mathrm{cone}})^2 C_{\mathrm{metric}} B_0 \theta_{*}  \cZ(\cG) + Z_0 \le \bigl((\funny Q_3 C_{\mathrm{cone}})^2   C_{\mathrm{metric}} B_0 + Z_0 \bigr)  \cZ(\cG).
\]
In the last inequality we have used that \(\cZ(\cG) \ge 1\) since each unstable curve is shorter than \(\delta_* \le 1\).
This implies that \(\cZ(\hat \cF_t \cG) \le Z_1 \cZ(\cG) \) for some \(Z_1>0\) big enough and concludes the proof of the statement.
\end{proof}

Let \(n_*\), \(Z_0\) and \(Z_1\) be given by Lemma \ref{lem:invariance1} and \(\Omega >1\) by Corollary \ref{cor:useful-density}. Set also, for some \(C_{\mathrm{cone}}>0\) big enough,
\begin{equation}\label{eq:intrinsic-constant-boundary}
\begin{split}
&\bold B_1 = Z_0/(1-\gamma), \quad \bold B_2 = Z_1^{n_*} \bold B_1 \quad \text{and} \\
&B_2 = \max\{1000 Z_1^{n_*} \max\{ B_1, \bold B_1\}, 10000  e^{\Omega \varpi_2}C_{\mathrm{cone}} \bold B_2 \}.
\end{split}
\end{equation}
 Note that by Lemma \ref{lem:invariance-unstable-curves}, \eqref{eq:intrinsic-constant-boundary} and \eqref{eq:contstant-omega2-regularity} the constants \(\varpi_2, B_2, D_2\) depend only on the table and the initial constants \(\varpi_1, B_1, D_1\), so that they can be incorporated in the generic constant \(C\). Recall also the terminology introduced after Proposition \ref{prop:invariance-standard-families}, which we prove at the end of this section: We call \((\varpi_1, B_1, D_1)\)-standard families initial regular standard families and \((\varpi_2, B_2, D_2)\)-standard families regular standard families. Lemma \ref{lem:invariance1} has the following Corollary, whose proof is a direct computation and it is omitted.
\begin{corollary}\label{cor:Growth-Lemma}
There exist \(\alpha_1,\alpha_2 >0\) such that, for all \(t \in [0,t_0]\) and standard family \(\cG\) with \(\max \rho_{k} /\min \rho_{k} \le \funny\) for every \(k\), we have:
  \begin{itemize}
        \item [a)] If \(\cZ(\cG) \le \bold B_1\) then \(\cZ(\hat \cF_t^n \cG) \le \bold B_2\) for all \(n \in \bN\);
         \item [b)] \(\cZ(\hat \cF_t^n \cG) \le \max \{Z_1^{n_*} \cZ(\cG), \bold B_2\}\) for all \(n \in \bN\);
        \item [c)] \(\cZ(\hat \cF_t^n \cG) \le \bold B_2\) for all \(n \ge \alpha_1\ln \cZ(\cG) + \alpha_2\).
    \end{itemize}
\end{corollary}

We now study the evolution \(\cL_t\) of the leaky system. To establish invariance, we use the contraction of \(\cZ\) under \(\hat \cF_t^{n_*}\) and the fact that after a finite number of iterations the lost mass is almost negligible for holes small enough. The following result is a weaker version of Lemma \ref{lem:lower-bound-survival} (to obtain Lemma \ref{lem:lower-bound-survival} we used invariance which we still haven't proved). 

\begin{lemma}\label{lem:lost-mass-with-t}
For any \(n \in \bN\) there exists \(C_{n} >0\) such that, for any regular standard family \(\cG\) and any \(t \in [0, t_0]\),
\[
\mu_{ \cG}(\cM_t^n) \ge 1 - C_{n} t .
\]
\end{lemma}
\begin{proof}
Using Lemma \ref{lem:push-forward-evolution} in the last equality,
\[
\mu_{\cG}(\cM_t^n) = 1 - \mu_{\cG}(\cup_{k=0}^n \cF^{-k}(\cH_t)) \ge  1 - \sum_{k=0}^n \mu_{\cG}(\cF^{-k}(\cH_t)) = 1 - \sum_{k=0}^n \mu_{\cF^{k}\cG}(\cH_t).
\]
By part b) of Corollary \ref{cor:Growth-Lemma} for \(t = 0\) and because \(\cG\) is a regular standard family, we have that \(\cZ(\cF^{k} \cG) \le  Z_1^{n_*} B_2\) for any \(k\) and by point c) of Corollary \ref{cor:useful-density}, we have that \(\cF^{k}\cG\) supports densities in \(\cC(\Omega \varpi_2)\) for all \(k\). Therefore, by Lemma \ref{lem:hole-intersection} and the equation above, for every \(n >0\) there exists \(C_n >0\) big enough such that, for all \(t \in [0,t_0]\),
\[
\mu_{\cG}(\cM_t^n) \ge 1 - \sum_{k=0}^{n} C e^{\Omega \varpi_2}B_2 t  \ge 1 - C_nt.
\]
This concludes the proof of the statement.
\end{proof}

The next Lemma relates the evolution of the boundary for the open system to the evolution \(\hat \cF_t\) with added discontinuities.

\begin{lemma}\label{lem:comparison-hole-regularity-images}
    For any \(n \in \bN\) and standard family \(\cG\),
    \[
    \cZ(\cL_t^n \cG) \le \frac{\cZ(\hat{\cF}_t^n \cG)}{\mu_{\cG}(\cM_t^n)}.
    \]
\end{lemma}
\begin{proof}
To fix the notation, let \(\cG = \{(p_k, W_k ,\rho_k)\}\) and \(\{\hat W_{k,j}\}\), \(\{V_{k,j}\}\) be the connected components of \(W_k \setminus \cS_{n,t}^{\bH}\) and \((W_k\setminus \cS_{n,t}^{\bH}) \cap \cM_{t}^n\) respectively. Let also \(\{\hat W^n_{k,j}\}\) and \(\{V^n_{k,j}\}\) be their \(n^{th}\) image (possibly further subdivided in curves of length less than \(\delta_*\)). Note that each element of the partition \(\{V^n_{k,j}\}\) of \(\hat \cF_t^n(W_k \cap \cM_t^n)\) is also an element of the partition \(\{\hat W^n_{k,j}\}\) of \(\hat \cF_t^n (W_k)\). Therefore, recalling the definition of \(\cL_t^{n}\cG\) in \eqref{eq:density-weights-evolution},
\[
\cZ(\cL_t^n \cG) = \sum_{k} p_k \sum_j \frac{\int_{V_{k,j}} \rho_k}{\mu_{\cG}(\cM_t^n)} \frac{1}{|V^n_{k,j}|} \le \sum_{k} p_k \sum_j \frac{\int_{\hat W_{k,j}} \rho_k}{\mu_{\cG}(\cM_t^n)} \frac{1}{|\hat W^n_{k,j}|} = \frac{\cZ(\hat{\cF}_t^n \cG)}{\mu_{\cG}(\cM_t^n)},
\]
where the first sums go over the elements of \(\{V^n_{k,j}\}\) and the second sums over those of \(\{\hat W^n_{k,j}\}\). This concludes the proof of the Lemma.
\end{proof}

We finally establish invariance for the leaky evolution. Recall that \((\varpi_1, B_1,D_1)\)-standard families are called initial regular standard families.

\begin{proof}[Proof of Proposition \ref{prop:invariance-standard-families}]
Let \(\cG\) be an initial regular standard family. By Lemma \ref{lem:push-forward-evolution}, \(\cL_t^n \mu_{\cG} = \mu_{\cL^n_t \cG}\). Hence the statement reduces to prove that \(\cL_t^n \cG\) is a regular standard family for all \(n\) and \(t \in [0,t_0]\) with \(t_0>0\) small enough. Set \(\cL_t^n \cG = \{(p^{(n)}_k, \rho^{(n)}_k, W^{(n)}_k)\}\). By construction the standard families \(\cL_t^n \cG\) are composed by \(n\)-homogeneous unstable curves and the fact that \(\rho^{(n)}_k \in \cC(\varpi_2)\) follows by part a) of Corollary \ref{cor:useful-density} and the fact that \(\rho_k^{(0)} \in \cC(\varpi_1)\). This proves that the densities belong to the right cones. Moreover, it also shows that, for all \(n \in \bN\),
\[
\frac{\max \rho^{(n)}_k}{ \min \rho^{(n)}_k} \le e^{\varpi_2 \delta_{*}^{1/3}} \le e^{\varpi_2}. 
\]
Hence, can apply Lemma \ref{lem:invariance1} to the standard families \(\cL_t^n \cG\). Let \(n_* \in \bN\) be given by Lemma \ref{lem:invariance1}. We first claim that for all \(k \in \bN\) we have \(\cZ(\cL_t^{kn_{*}} \cG ) \le 10\max\{ B_1, Z_0/(1-\gamma)\} < B_2\). We prove the statement by induction. The case \(k = 0\) follows from the fact that \(\cG\) is an initial regular standard family. Assume now that the statement holds for some \(k\).  By the inductive assumption and Lemma \ref{lem:lost-mass-with-t}, there exists \(C_{n_*}>0\) such that, for all \(t \in [0,t_0]\),
\begin{equation}\label{eq:bound-lost-mass-invariance-today}
     \mu_{\cL_t^{k n_{*}}\cG}(\cM_t^{n_*}) \ge 1-C_{n_*}t.
\end{equation}
Therefore, by Lemmata \ref{lem:comparison-hole-regularity-images} and \ref{lem:invariance1} and \eqref{eq:bound-lost-mass-invariance-today},
\[
\begin{split}
    \cZ(\cL_t^{(k+1) n_{*}} \cG ) \le \frac{\cZ(\hat{\cF}_t^{n_{*}} \cL_t^{k n_{*}}\cG)}{ \mu_{\cL_t^{k n_{*}}\cG}(\cM_t^n)} &\le \frac{\gamma  10\max\{ B_1, Z_0/(1-\gamma)\} + Z_0}{1-C_{n_*}t} \le  10\max\{ B_1, Z_0/(1-\gamma)\}.
\end{split}
\]
In the last inequality in the equation above, we consider \(t \in [0,t_0]\) for \(t_0 >0\) small enough depending on \(C_{n_*}\). This is fine since \(n_*\) is fixed (and depends only on the billiard and the initial regularity data). This concludes the proof of the claim. For general \(n = kn_{*} + q\), \(q < n_*\), by Lemma \ref{lem:comparison-hole-regularity-images} and \eqref{eq:bound-lost-mass-invariance-today} again, and for \(t \in [0,t_0]\) and \(t_0 >0\) small enough,
\begin{equation}\label{eq:this-one}
\begin{split}
\cZ(\cL_t^{k n_{*} + q} \cG ) &\le \frac{\cZ(\hat \cF_t^q (\cL_t^{kn_{*}}\cG))}{\mu_{\cL_t^{kn_*}\cG}(\cM_t^q)} \le \frac{Z_1^q \cZ(\cL_t^{kn_{*}}\cG) }{1 - C_{n_*} t}\le \frac{Z_1^{n_*}   10\max\{ B_1, Z_0/(1-\gamma)\}}{1/2}  \le B_2.
\end{split}
\end{equation}
where in the second inequality we also used the second part of Lemma \ref{lem:invariance1}. The last inequality follows from \eqref{eq:intrinsic-constant-boundary}. This concludes the proof of the Proposition.
\end{proof}

\subsection{Coupling} \textit{In this section we define coupled standard families and we show that after a fixed iteration of the map we can couple two initial regular standard families.} 

\begin{definition}\label{def:stacking}
Let \(\ell_1 = (W_1, \rho_1)\) and \(\ell_2 = (W_2, \rho_2)\) be two standard pairs and \(\eta >0\). We say that \(\ell_1\) and \(\ell_2\) are \(\eta\)-coupled if 
\begin{equation}\label{eq:eta-coupled-def}
I_{W_1} = I_{W_2}, \quad \|\vf_{W_1} - \vf_{W_2}\|_{\cC^1} \le \eta, \quad \rho_1 = \rho_2.
\end{equation}
Two standard families are \(\eta\)-coupled if there exists a bijection between their standard pairs so that each standard pair of the first family is \(\eta\)-coupled with the corresponding standard pair of the second family and the corresponding weights are equal. Moreover, if the first two conditions in \eqref{eq:eta-coupled-def} are satisfied for a pair of unstable curves \(W_1\) and \(W_2\) we say that \(W_1\) and \(W_2\) (or \(\vf_{W_1}\) and \(\vf_{W_2}\)) are \(\eta\)-stacked. 
\end{definition}

We start with a simple lemma which tracks the weights in standard families generated by a single standard pair. Intuitively, we want that a long curve in the push-forward carries at least some portion of the original mass. The dependence on \(n\) is not relevant as we need to couple only a small amount of mass but every fixed number of iterations.

\begin{lemma}\label{lem:control-weights}
Let \(\ell = (W, \rho)\) be a regular standard pair and set \(\hat\cF_t^n \ell  = \{(p_k, W_k', \rho_k)\}\). For any \(n \in \bN\), there exists \(c_n >0 \) such that \( p_k \ge  c_n |W_k'|^{2^n}\).
\end{lemma}
\begin{proof}
 Call \(\{W_{k}\}_k\) the connected elements of \(W \setminus \cS_{n,t}^{\bH}\) and \(W_{k}' = \cF^n(W_{k})\) (if \(|W_{k}'|\ge \delta_*\) we partition it in multiple components and this further subdivision does not affect the proof). According to \eqref{eq:density-weights-evolution} and the discussion afterwards, \(\hat \cF_t^{n} \ell  = \{(p_{k}, W_{k}', \rho_{k})\}\), where
    \[
    p_{k} =   \int_{I_{W_{k}}}\rho \ge  e^{-\varpi_2 } \frac{|I_{W_{k}}|}{|W|} \ge C_{\mathrm{cone}}^{-1} e^{-\varpi_2 } |W_k|.
    \]
In the first inequality we used that \(\ell\) supports a regular density and in the second that \(|W| \le 1\). By Lemma \ref{lem:upp-bound-stretch-curves}, for some \(C_n >0\),
    \[
    |W_k'| = |\cF^n(W_k)| \le C_n |W_k|^{\frac{1}{2^n}} \quad \text{or} \quad |W_k| \ge C_n^{-1} |W_k'|^{2^n}.
    \]
   This, together with the equation before, concludes the proof of the Lemma. 
\end{proof}

We introduce partitions \(\cQ_{\eta}\) which are instrumental for our coupling arguments. Let \(\cQ_{\eta}\) be a partition of \(\cM\) into squares, each one of size \(\eta>0\), whose sides are parallel to the \(r\) and \(\vf\) coordinates axes. Here and in the following the size of a square is the length of its sides. For any unstable curve \(W\) we denote by  
\begin{equation}\label{eq:center-unstable-curve-def}
\begin{split}
\Center (W) = \biggl\{ x \in W: &\text{ the segment in \(I_W\) connecting} \\
&\hspace{2.8cm}\text{\(\pi_r(x)\) to \(\pi_r(\partial W)\) is longer than }\frac{|I_W|}{100 e^{\varpi_2}}\biggr\}.
\end{split}
\end{equation}
Set also 
\begin{equation}\label{eq:setDeltastar}
\Delta_{*} = \frac{1}{100\bold B_2}.
\end{equation}
In view of part (c) of Corollary \ref{cor:Growth-Lemma}, \(\Delta_*\) is a lower bound on the length that most (\(\ge 99\%\)) unstable curves have at equilibrium. Note that \(\Delta_* \gg 1/B_2\). More precisely, by \eqref{eq:intrinsic-constant-boundary},
\begin{equation}\label{eq:relationDelta-B2}
    \Delta_* \ge \frac{100 C_{\mathrm{cone}} e^{\Omega\varpi_2}}{B_2} \ge \frac{100 C_{\mathrm{cone}} e^{\varpi_2}}{B_2}.
\end{equation}
This means that we have been quite loose in defining a standard family regular as at equilibrium standard families are actually much better than \(B_2\)-regular. In Lemma \ref{lem:fundamental-coupling} we couple a fraction of two standard families. The next three Lemmata are preliminary results that we comment on the way. First, we show that it is possible to \(\eta\)-stack long curves for \(\eta\) arbitrarily small if they intersect a square in \(\cM \setminus \cS_{-N}\) with \(N\) big enough. The general idea is that the sets \(\cM \setminus \cS_{-N}\) are good approximations of the unstable manifolds when \(N \to \infty\) and, if two long curves happen to be in the same connected component of \(\cM \setminus \cS_{-N}\), at least part of them is close in the \(\cC^1\) norm.

\begin{lemma}\label{lem:geometric-last}
    For any \(\eta>0\) there exists \(N >0\) such that if the centers of two \(N\)-homogeneous unstable curves \(W_{\iota}\) with \(|W_{\iota}| \ge \Delta_{*}\) intersect the same square \(Q\subset \cM\setminus \cS_{-N}\) then there exist two \(\eta\)-stacked curves \(W_{\iota,c}\subseteq W_{\iota}\) with \(|W_{\iota,c}|\ge 1/B_2\).
\end{lemma}
\begin{proof}
First notice that by sending \(N \to \infty\) we have that \(\mathrm{diam}(Q) \to 0\) for any square \(Q \subset \cM \setminus \cS_{-N}\). Indeed, if not, \(Q \subset \cM \setminus \cS_{-N}\) would contain a stable manifold of a certain fixed size, which is absurd when \(N\) approaches infinity. Therefore, we let \(N\) be large enough so that \(\mathrm{diam}(Q) \ll \Delta_{*}\) for any \(Q \subset \cM \setminus \cS_{-N}\). Because both curves \(W_{\iota}\) are aligned with \(\hat \cC^u\) their slopes \(\vf_{W_{\iota}}'\) are bounded from above and below, i.e., for some \(C_{\mathrm{cone}} >0\),
\begin{equation}\label{eq:too-many-cones}
C_{\mathrm{cone}}^{-1} \le |\vf_{W_{\iota}}'| \le C_{\mathrm{cone}}.
\end{equation}
Since \(|W_{\iota}| \ge \Delta_{*}\), \(\Center(W_{\iota}) \cap Q \neq \emptyset\) and \(\mathrm{diam}(Q) \ll \Delta_*\), there exists an interval \(I_c\) such that \(I_{W_1} \cap I_{W_2} \supseteq  I_c\).  Denote by \(\vf_{\iota, c} = \vf_{W_{\iota}} \Id_{I_c}\) and by \(W_{\iota, c}\) the graphs of \(\vf_{\iota, c}\). These are the stacked curves. By \eqref{eq:too-many-cones} the \(r\)-projections of the curves are long at least \(|I_{W_{\iota}}| \ge \Delta_{*}/C_{\mathrm{cone}}\) and by \eqref{eq:center-unstable-curve-def} each \(I_{W_{\iota}}\) extends beyond \(\pi_r(Q)\) on both sides by \(\Delta_*/(C_{\mathrm{cone}}100 e^{\varpi_2})\). In particular, we can require that the projections of the stacked curves are longer than
\begin{equation}\label{eq:geom-square-coupled}
    |I_{W_{\iota,c}}| \ge \frac{\Delta_*}{100 C_{\mathrm{cone}} e^{\varpi_2}} \quad \text{so that} \quad |W_{\iota,c}| \ge\frac{\Delta_{*}}{100 C_{\mathrm{cone}}e^{\varpi_2}} \ge \frac{1}{B_2},
\end{equation}
where in the second inequality on the right we used \eqref{eq:relationDelta-B2}. It remains to prove that \(\vf_{\iota,c}\) are \(\eta\)-stacked if \(N\) is big enough. Since both curves intersect the same square \(Q \subset \cM \setminus \cS_{-N}\) and are \(N\)-homogeneous, they belong to the same connected component of \(\cM \setminus \cS_{-N}\). Let \(r_0 \in I_c\) and \(\eta'>0\) be such that \(|\vf_{1,c}(r_0) -\vf_{2,c}(r_0)| = \eta'\) and denote by \(H_0\) the vertical line connecting \((r_0, \vf_{1,c}(r_0))\) to   \((r_0, \vf_{2,c}(r_0))\). By Lemma \ref{lem:continuation-of-singularities} about continuation of singularity lines, we have that \(H_0 \subseteq \cM \setminus \cS_{-N}\). Therefore, \(\cF^{-N}(H_0)\) is a (continuous) stable curve and its length is at most of order one, as it cannot exceed the phase space. On the other hand, since \(H_0\) is aligned with the large stable cone, by the expansion \eqref{eq:curve-expanding} of stable curves in the past,
\[
|\cF^{-N}(H_0)| \ge C \Lambda_0^{N}|H_0| \ge C\Lambda_0^N \eta'.
\]
Hence, the equation above implies that \(\eta' \lesssim \Lambda_0^{-N}\). Therefore, by choosing \(N\) large enough, we obtain that \(\|\vf_{1,c} -\vf_{2,c}\|_{\cC^0} \le \eta\). As for the bound on the difference of the slopes, for any \(r_0 \in I_c\), we have (see \cite[Equations (3.34) and (3.39)]{MR2229799} and notice that \(\cC^u\) corresponds to wave fronts with \(0 \le \cB \le \infty \))
\begin{equation}\label{eq:slope-vector-push-cone}
 \frac{ \cos \vf_{\iota,c}(r_0)}{\tau_{-1} + \frac{1}{\cR_{-1} + \frac{1}{\ddots +  \frac{1}{\cR_{-N+1}}}}} + \kappa(r_0)  \le \vf_{\iota,c}'(r_0) \le  \frac{\cos \vf_{\iota,c}(r_0)}{\tau_{-1} + \frac{1}{\cR_{-1} + \frac{1}{\ddots + \frac{1}{\tau_{-N}}}}}+ \kappa(r_0),
\end{equation}
where \(\cR = 2\kappa/\cos \vf\) and the negative subscripts indicate that the quantities are evaluated along the backward orbit of \(x_{\iota} = (r_0, \vf_{\iota, c}(r_0))\). Calling \(\cB_{\iota,N}^{e}\) and  \(\cB_{\iota,N}^{o}\) the continued fraction on the right and on the left, we have (see \cite[Lemma 4.36]{MR2229799})
\[
|\cB_{\iota,N}^{e} - \cB_{\iota,N}^o| \le \frac{1}{N \tau_{\mathrm{min}}}.
\]
Therefore, we pick \(M >0\) big enough \(M<N\) so that \(\vf_{\iota,r}'\) is approximated arbitrarily well on \(\cM \setminus \cS_{-M}\) by the truncations \(\cos \vf_{\iota,c} \cB_{\iota, M}^{e} + \kappa\). We now prove that these truncations can be made arbitrarily close at the two points the points \((r_0,\vf_{\iota,c}(r_0))\) whenever \(N\) is big enough. Note that \(\cos \vf_{\iota,c} \cB_{\iota, M}^{e} + \kappa\) are uniformly continuous functions on \(\cM \setminus \cS_{-N} \subset \cM \setminus \cS_{-M}\), \(N \gg M\), since they involve backward iterates until time \(-M\). Moreover, by the first part of the Lemma, the points \((r_0,\vf_{\iota,c}(r_0))\) can be made arbitrarily close by choosing \(N\) big enough. Therefore, by continuity, the function \(\cos \vf_{\iota,c} \cB_{\iota, M}^{e} + \kappa\) takes values arbitrarily close at the points \((r_0,\vf_{\iota,c}(r_0))\) when \(N \to \infty\) uniformly in \(r_0\). This concludes the proof of the Lemma.
\end{proof}

We call a standard pair regular if the associated standard family is a regular standard family. Once two curves are stacked, it is possible to couple a portion of the mass they carry.

\begin{lemma}\label{lem:geometric-intersection}
There exists \(p_c > 0 \) such that, for any two \(W_{1}, W_2 \in \cW\) as in Lemma \ref{lem:geometric-last} and standard pairs \(\ell_{\iota} = (W_{\iota}, \rho_{\iota})\) with \(\rho_\iota \in \cC(\varpi_2/2)\), there exist two \(\eta\)-coupled regular standard pairs \(\ell_{\iota,c}\) and two regular standard families \(\cG_{\iota,u}\) such that
    \[
    \mu_{\ell_{\iota}} = p_c \mu_{\ell_{\iota,c}} + (1-p_c) \mu_{\cG_{\iota,u}} .
    \]
\end{lemma}
\begin{proof}
The fact that there exist \(\eta\)-stacked unstable curves \(W_{\iota, c}\subseteq W_{\iota}\) follows by Lemma \ref{lem:geometric-last}. Since \(\rho_{\iota} \in \cC(\varpi_2 /2 )\), by Lemma \ref{lem:ubiquous-cone},
\[
\rho_{\iota} \ge \frac{e^{-\varpi_2 /2}}{|W_\iota|} \ge \frac{e^{-\varpi_2}}{\delta_{*}} \ge e^{-\varpi_2} := c_0.
\]
Let \(\rho_{*} : I_c \to \bR^+\) be the constant function equal to \(c_0/4\). By Lemma \ref{lem:classical-cones} and because \(\rho_{\iota} \in \cC(\varpi_2/2)\), we have that \((\rho_{\iota} - c_0/4) \Id_{I_c} \in \cC(\varpi_2)\). We define the coupled standard pairs and the coupled mass as
\[
\ell_{\iota, c} = \biggl(W_{\iota,c}, \frac{\rho_*}
{\int_{I_c} \rho_*} \biggr) \quad \text{and} \quad p_c = \int_{I_c} \rho_* = \frac{c_0 |I_c|}{4} \ge \frac{c_0 C_{\mathrm{cone}}^{-1}\min_{\iota}|W_{\iota,c}|}{4}\ge C_{\mathrm{cone}}^{-1} \frac{e^{-\varpi_2}}{ B_2},
\]
where in the first inequality we have used that, by Lemma \ref{lem:geometric-last}, it holds \(|W_{\iota,c}| \ge 1/B_2\). This also shows that the standard pairs \(\ell_{\iota,c}\) are regular standard pairs. Moreover, by \eqref{eq:geom-square-coupled} and by coupling less mass if necessary, we can require that \(|W_{\iota,c}| \le \Delta_{*}/(50 C_{\mathrm{cone}}e^{\varpi_2})\) so that \(|W_{\iota} \setminus W_{\iota,c}|> \Delta_{*}/2 \gg 1/B_2\). This, together with the fact that \((\rho_{\iota} - c_0/4) \Id_{I_c} \in \cC(\varpi_2)\), shows that the uncoupled standard family is regular as well.
\end{proof}

We will use the next result to show that the images of unstable curves actually intersect the small squares \(Q \subset \cM \setminus \cS_{-N}\) of before. The next Lemma is a direct consequence of the known equidistribution of standard pairs for the billiard map. Notably similar statements of `finite mixing for curves' can be deduced just by the mixing of the billiard map (after much more effort of course. See e.g. \cite[Proposition 7.83]{MR2229799}. Note that the statement below isn't an \textit{immediate} consequence of mixing because standard pairs correspond to singular measures.)

\begin{lemma}\label{lem:mixing-standard-pairs-unperturbed}
For any \(\eta> 0\) there exists \(N>0\) such that, for all \(n \ge N\), any regular standard pair \(\ell\) and \(t \in [0,t_0]\), we have
    \[
    \sum_{Q \in \cQ_{\eta}}\bigl|\mu_{\hat \cF^n_t(\ell)}(\Id_Q) - \mu_0 (Q)  \bigr| \le \frac{1}{100}.
    \]
\end{lemma}
\begin{proof}
For any \(Q \in \cQ_{\eta}\) and \(\ve >0\), we approximate \(\Id_{Q}\) with \(g_{Q,\ve} \in \cC^1 (\cM)\) such that \(\|g_{Q,\ve}\|_{\cC^0} \le 1\), \(\mu_0 (g_{Q,\ve}) = \mu_0 (Q)\) and \(g_{Q} = \Id_{Q}\) on \(\cM \setminus [\partial Q]_{\ve}\). Therefore, because \(\partial Q\) consists of a union of four segments transversal to the unstable cone, by the same argument of \eqref{eq:discontinuities-measure-standard-family}, there exists \(C_0 >0\) such that, for any \(n\in \bN\) and \(Q \in \cQ_{\eta}\),
\[
|\mu_{\hat \cF_t^n(\ell)}(\Id_Q - g_{Q\ve})| \le 2\mu_{\hat \cF_t^n(\ell)}([\partial Q]_{\ve}) \le C \cZ(\hat \cF_t^n (\ell)) \ve \le C_0 \ve.
\]
In the last inequality we have used part b) of Corollary \ref{cor:Growth-Lemma} and the fact that \(\ell\) is a regular standard family. Take any \(Q \in \cQ_{\eta}\) and set \(\ve = \mu_0(Q)/(C_0 200)\). By \cite[Theorem 7.31]{MR2229799}, we have that there exists \(n_Q\) such that for all \(n\ge  n_Q\),
\[
|\mu_{\hat \cF_t^n(\ell)}(\Id_Q) - \mu_0(Q) | = |\mu_{\hat \cF_t^n(\ell)}(\Id_Q - g_{Q\ve})| + |\mu_{\hat \cF_t^n(\ell)}(g_{Q,\ve}) - \mu_0(Q) | \le \frac{\mu_0(Q)}{200} + \frac{\mu_0(Q)}{200} =\frac{\mu_0(Q)}{100}.
\]
Since there are a finite number of elements \(Q \in \cQ_{\eta}\), the statement follows by taking \(N = \max_{Q \in \cQ_{\eta}} n_Q <\infty\) and the fact that \(\mu_0\) is a probability measure.
\end{proof}

Later on, we will fix \(\eta_0 \ll \Delta_*\) to be a very small number. This is the threshold distance that two standard pairs need to have to be meaningfully considered coupled. The precise value of \(\eta_0\) is given by Lemma \ref{lem:one-step-loss-mass} and depends only on the billiard and the initial regularity of standard families. In practice, we want \(\eta_0\) to be small enough so that not all the mass gets decoupled in a few iterations. In the following Lemma, which is the most important of this section, we argue that it is actually possible to couple some of the surviving mass of two standard families.

\begin{lemma}\label{lem:fundamental-coupling}
For any \(\eta_0>0\) there exists \(N>0\) such that the following is true. For any \(n\ge N\), there exist \(t_0 >0\), \(p_c, p_{\iota,u} \in (0,1)\), \(p_c + p_{\iota,u} \le 1\) with the following property. For any two regular standard families \(\cG_{\iota}\) there exist two \(\eta_0\)-coupled regular standard pairs \(\ell_{\iota,c}\) and two regular standard families \(\cG_{\iota,u}\) such that, for any \(\phi \in \cC^0(\cM)\) and \(t \in [0,t_0]\),
    \[
    \mu_{\cG_{\iota}} (\phi \circ \hat \cF^{n}_t \Id_{\cM_t^n}) = p_c \mu_{\ell_{\iota,c}} (\phi) + p_{\iota,u}\mu_{\tilde \cG_{\iota,u}} (\phi).
    \]
\end{lemma}
To help the reader, we give an informal description before proceeding with the actual proof. We first let \(N_{**}>0\) be big enough such that we can apply Lemma \ref{lem:geometric-last} and two curves that intersect a square in \(\cM \setminus \cS_{-N_{**}}\) can be \(\eta_0\)-stacked. We consider a partition \(\cQ_{\eta_1}\) of \(\cM\) in squares so small that most of them belong to \(\cM \setminus \cS_{-N_{**}}\). We then pick \(N \ge N_{**}\) big enough so that \(n \ge N\) satisfies two properties. The first is that \(\cZ(\hat\cF_t^{n} \ell_{\iota})\) is at equilibrium. Indeed, by \eqref{eq:relationDelta-B2} regular standard families are allowed to have a boundary much bigger (\(\sim B_2\)) than what they tend to have (i.e., roughly \(Z_0/(1-\gamma) \)) after enough iterations. The second property is that plenty of long curves in \(\cF_t^{n} \ell_{\iota}\) intersect most of the little squares in \(Q \subset \cM \setminus \cS_{-N_{**}}\). This follows by mixing of the billiard map via Lemma \ref{lem:mixing-standard-pairs-unperturbed}. Doing so, most curves will be long and equidistributed in the phase space: it is impossible that there are no couple of curves, one from the first family and one from the second, that do not intersect the same little square in \(Q \subset \cM \setminus \cS_{-N_{**}}\). Once \(n\) is fixed, we pick \(t_0 >0\) very small, so that the mass that ends up in the hole in the first \(n\) iterates is almost negligible.
\begin{proof}[Proof of Lemma \ref{lem:fundamental-coupling}.]
We first consider the case in which the two standard families are given single standard pairs \(\ell_{\iota} = (W_{\iota}, \rho_{\iota})\), \(\iota \in \{1,2\}\). Fix some \(\eta_0 >0\). By Corollary \eqref{cor:Growth-Lemma} point (c) and the fact that \(\cZ(\ell_{\iota}) \le B_2\), there exits \(N_{*} \sim \ln B_2\) such that, for any \(n \ge N_{*}\), \(\cZ(\hat \cF_t^n \ell_{\iota}) \le \bold B_2 \). Hence, setting \(\hat\cF^n_t \ell_{\iota} = \{(p_{\iota,k},W_{\iota,k}, \rho_{\iota,k})\}\), by Lemma \ref{lem:meaning-cZ} and \eqref{eq:setDeltastar}, for any \(n \ge N_*\),
\begin{equation}\label{eq:most-curves-are-long-image}
 \sum_{k : |W_{\iota,k}| < \Delta_{*}}p_{\iota,k} \le \cZ(\hat \cF_t^n \ell)  \Delta_{*}\le \bold B_2 \Delta_{*} = \frac{1}{100}.
\end{equation}
Let \(N_{**} >0 \) be given by Lemma \ref{lem:geometric-last} with \(\eta = \eta_0\) and let \(\eta_1>0\) be such that the overwhelming majority of squares in \(\cQ_{\eta_1}\) belong to \(\cM \setminus \cS_{-N_{**}}\). This is possible since \(\cS_{-N_{**}}\) is a finite union of \(\cC^1\) curves. In particular, we take \(\eta_1 > 0\) small enough such that
\begin{equation}\label{eq:eta1-squares}
    \sum_{Q \in \cQ_{\eta_1}: Q \cap \cS_{-N_{**}} \neq \emptyset}\mu_0 (Q) \le \frac{1}{100}
\end{equation}
Let also \(N_{***} \in \bN\) be given by Lemma \ref{lem:mixing-standard-pairs-unperturbed} with \(\eta = \eta_1\). Finally, let \(\bar n\) be given by part b) of Corollary \ref{cor:useful-density}. Set \(N = \max \{N_*, N_{**}, N_{***}, \bar n\}\) and fix some \(n \ge N \). We now define the set of squares \(\cQ_{good} (\iota)\) where we can couple mass. Let
\[
\begin{split}
N_t(\iota) &= \biggl \{x \in W_{\iota} \cap \cM_t^n: \text{ } \hat\cF_t^n(x) \text{ belongs to the Center of an unstable curve in \(\hat \cF_t^n \ell\) } \\
&\hspace{3cm} \text{ not smaller than } \Delta_{*} \biggr\}.
\end{split}
\]
These are the points that survive after \(n\) iterations and belong to the center of long unstable curves. We also set
\[
\cQ_{good} (\iota) = \{Q \in \cQ_{\eta_1}: Q \cap \cF^n(N_t(\iota)) \neq \emptyset \text{ and } Q\subset \cM\setminus \cS_{-N_{**}}\}.
\]
We first prove first the following fact. \textit{Claim \((\star)\)}: there exists \(t_0 >0\) small enough such that
\[
\text{ for any \(t \in [0, t_0]\),}\quad \cQ_{good}(1) \cap \cQ_{good}(2) \neq \emptyset.
\]
If \(Q \notin Q_{good}(\iota)\), either \(Q \cap \cS_{-N_{**}} \neq \emptyset\) or \(Q \cap \cF^n(N_t(\iota)) = \emptyset\). Therefore,
\[
\begin{split}
    &\sum_{Q \notin \cQ_{good} (\iota)} \mu_0 (Q) = \sum_{Q \in \cQ_{\eta_1}: Q \cap \cS_{-N_{**}} \neq \emptyset}\mu_0 (Q) + \sum_{Q \in \cQ_{\eta_1}: Q \cap \cF^n(N_t(\iota)) = \emptyset}\mu_0 (Q)\\
    &\le \sum_{Q \in \cQ_{\eta_1}: Q \cap \cS_{-N_{**}} \neq \emptyset}\mu_0 (Q) + \sum_{Q \in \cQ_{\eta_1}}\bigl| \mu_0 (Q) - \mu_{\hat\cF^n_t(\ell_\iota)}(\Id_Q)\bigr| +  \sum_{Q \in \cQ_{\eta_1}: Q \cap \cF^n(N_t(\iota)) = \emptyset} \mu_{\hat \cF^n_t(\ell_\iota)}(\Id_Q).
\end{split}
\]
By Lemma \ref{lem:mixing-standard-pairs-unperturbed} and Equation \eqref{eq:eta1-squares}, the first two sums in the second line contribute to less than \(1/100\) each. For the third sum, since \(\cQ_{\eta_1}\) is a partition of \(\cM\),
\[
\sum_{Q \in \cQ_{\eta_1}: Q \cap \cF^n(N_t(\iota)) = \emptyset} \mu_{\hat\cF_t^n(\ell_i)}(\Id_{Q}) \le \mu_{\hat\cF_t^n(\ell_{\iota})}(\cM \setminus \hat \cF_t^n(N_t(\iota))).
\]
We note that the points in \(\hat \cF_t^n (W_{\iota}) \cap (\cM \setminus \hat \cF_t^n(N_t(\iota)))\) are those points that either belong to small components, or belong to long components but not in their centers, or end up in the hole. Therefore, according to the previous equation, recalling the notation \(\hat\cF^n_t \ell_{\iota} = \{(p_{\iota,k},W_{\iota,k}, \rho_{\iota,k})\}\),
\[
\begin{split}
\sum_{Q \in \cQ_{\eta_1}: Q \cap \cF^n(N_t(\iota)) = \emptyset} \mu_{\hat\cF_t^n(\ell_i)}(\Id_Q) &\le \sum_{k : |W_{\iota,k}| < \Delta_{*}}  p_{\iota,k}
+\sum_{k : |W_{\iota,k}| \ge \Delta_{*}}p_{\iota,k} \mu_{\ell_{\iota,k}}(\Id_{\cM \setminus \Center W_{\iota, k}}) \\
&+ \mu_{\ell_{\iota}} (\cM \setminus \cM_{t}^n).
\end{split}
\]
We estimate each contribution one by one. By equation \eqref{eq:most-curves-are-long-image}, the first term is less than \(1/100\). By \eqref{eq:center-unstable-curve-def} and the fact that \(\rho_{\iota, k} \in \cC(\varpi_2)\),
\[
\mu_{\ell_{\iota,k}}(\Id_{\cM \setminus \Center W_{\iota,k}}) = \int_{I_{W_{\iota,k}\setminus \Center W_{\iota,k}}} \rho_{\iota,k} \le \frac{e^{\varpi_2 }}{|I_{W_{\iota,k}}|} |I_{W_{\iota, k} \setminus \Center W_{\iota,k}}| \le  \frac{2}{100}.
\]
This estimate holds for each standard pair, giving a \(2/100\) estimate for the second term. Finally, by Lemma \ref{lem:lost-mass-with-t} and for \(t_0 >0\) small enough depending on \(n\), and \(t \in [0,t_0]\),
\[
\mu_{ \ell_{\iota}} (\cM \setminus \cM_{t}^n) \le C_{ n} t \le \frac{1}{100}.
\]
Therefore, collecting the previous estimates, we have, for both \(\iota \in \{1,2\}\),
\[
 \sum_{Q \notin \cQ_{good} (\iota)} \mu_0 (Q)\le \frac{6}{100} < \frac{1}{2}.
\]
In particular, the equation above proves Claim \((\star)\). Let \(Q_c \in \cQ_{good}(1) \cap \cQ_{good}(2)\). By definition of good squares we have that a long unstable curve \(\tilde W_{1} \subset \hat \cF_t^n (W_1 \cap \cM_t^n)\) and a long unstable curve \(\tilde W_{2} \subset \hat \cF_t^n (W_{2} \cap \cM_t^n)\) intersect the same square \(Q_c\in \cQ_{\eta_1}\). Consider the standard pairs
\[
\tilde \ell_{\iota} = \biggl(\tilde W_{\iota}, \frac{\cL^n \rho_{\iota| I_{\tilde W_{\iota}}}}{ \int_{I_{\tilde W_{\iota}}}   \cL^n \rho}\biggr).
\]
By part b) of Corollary \ref{cor:useful-density}, \(\cL^n \rho_{\iota| I_{\tilde W_{\iota}}}\in \cC(\varpi_2/2)\). Moreover, by construction \(\tilde W_{\iota}\) are \(N_{**}\)-homogeneous (because they belong to \(\hat\cF_t^n \ell_{\iota}\) for \(n \ge N_{**}\)), they are long in the sense that \(|\tilde W_{\iota}| \ge \Delta_{*}\) and they intersect the same square \(Q_c \subset \cM \setminus \cS_{-N_{**}}\). Therefore, by Lemma \ref{lem:geometric-last} and the choice of \(N_{**}\), there exist \(\eta_0\)-stacked curves \(W_{\iota,c}\subseteq \tilde W_{\iota}\). Moreover, applying Lemma \ref{lem:geometric-intersection}, there exists \(\tilde p_c >0\), two regular standard families \(\tilde \cG_{\iota,u}\) and regular \(\eta_0\)-coupled standard pairs \(\ell_{\iota,c}\) such that
\[
\mu_{\tilde \ell_{\iota}} = \tilde p_c \mu_{\ell_{\iota,c}} + (1-\tilde p_c)\mu_{\tilde \cG_{\iota,u}}.
\]
Since the curves \(\tilde W_{\iota}\) belongs to \(\cF^n(\cM_t^n)\), the decomposition above implies that, for some \(\tilde p_{\iota}, p_{\iota,u}, p_{\iota,u}'>0\) and standard families \(\cG_{\iota,u}, \cG_{\iota,u}'\), for any \(\phi \in \cC^0(\cM)\),
\begin{equation}\label{eq:coupled-uncoupled-dec-first-meeting}
\mu_{\ell_{\iota}} (\phi \circ \hat \cF_t^n \Id_{\cM_t^n}) =  \tilde p_{\iota} \mu_{\tilde \ell_{\iota}}(\phi) +  p_{\iota,u}\mu_{\cG_{\iota,u}} (\phi)  = \tilde p_c \min_{\iota \in \{1,2\}} \tilde p_{\iota} \mu_{\ell_{\iota,c}} (\phi) + p_{\iota,u}'\mu_{\cG_{\iota,u}'}(\phi).
\end{equation}
Here, \(\cG_{\iota,u}'\) is a convex combination of \(\cG_{\iota,u}\), \(\tilde \cG_{\iota,u}\) and \(\ell_{\iota,c}\). By Lemma \ref{lem:control-weights}, for some \(c_n >0\),
\[
p_c := \tilde p_c \min_{\iota \in \{1,2\}} \tilde p_{\iota} \ge c_n \tilde p_c (\Delta_*)^{2^n} > 0, 
\]
yielding a lower bound on the amount of coupled mass that is uniform in the choice of the initial standard pairs. To conclude the Lemma we need to bound the regularity of \(\cG_{\iota,u}'\). The fact that the densities are in \(\cC(\varpi_2)\) follows from the fact that, by Lemma \ref{lem:geometric-intersection}, \(\tilde \cG_{\iota,u}\) and \(\ell_{\iota, c}\) are regular standard families and \(\cG_{\iota,u}\) supports regular densities which are push-forwards of the original regular density of \(\ell_{\iota}\). Moreover, by \eqref{eq:coupled-uncoupled-dec-first-meeting} and the equation after it,
\begin{equation}\label{eq:estimate-reg-uncoupled}
\mu_{\cL_t^n \ell_{\iota}} = \frac{p_c}{\mu_{\ell_{\iota}}(\cM_t^n)} \mu_{\ell_{\iota,c}} + \frac{p_{\iota,u}'}{\mu_{\ell_{\iota}}(\cM_t^n)}\mu_{\cG_{\iota,u}'}
\end{equation}
By Lemmata \ref{lem:comparison-hole-regularity-images} and \ref{lem:lost-mass-with-t} and recalling the choice of \(N_{*}\) before \eqref{eq:most-curves-are-long-image}, there exists \(t_0 >0\) small enough depending on \(n\) such that, for \(t \in [0,t_0]\), we have \(\mu_{\ell_{\iota}}(\cM_t^n) \ge 1/2\) and
\[
\cZ(\cL^n_t \ell_{\iota}) \le \frac{\cZ(\hat{\cF}_t^n \ell_{\iota})}{\mu_{\ell_{\iota}}(\cM_t^n)} \le \frac{\bold B_2}{1 - C_n t} \le 2 \bold B_2\ll B_2 .
\]
Therefore, by \eqref{eq:estimate-reg-uncoupled} and choosing possibly a smaller \(p_c\) one obtains that the boundary of \(\cG_{\iota,u}'\) is can be made only slightly bigger than \(2\cZ(\cL_t^n \ell_{\iota})\) and so \(\cG_{\iota,u}'\) is a regular standard family. This concludes the proof for the case of two standard pairs. In the general case of two regular standard families \(\cG_{\iota}\), by Corollary \eqref{cor:Growth-Lemma} point (c), after \(N_{*}\sim \ln B_2\) iterations \(\cZ(\cF_t^{N_*} \cG_{\iota}) \le \bold B_2\). Then, calling \(\cF_t^{N_*} \cG_{\iota} = \{(p_{\iota, k}, W_{\iota, k}, \rho_{\iota,k})\}\), 
\[
\sum_{p_{\iota,k}: |W_{\iota,k}| \ge 1/B_2 } p_{\iota,k} \ge 1 - \frac{\bold B_2}{B_2} \ge \frac{1}{2}.
\]
Here, we have used  Lemma \ref{lem:meaning-cZ} in the first inequality and Equation \eqref{eq:intrinsic-constant-boundary} in the second. Therefore, at least half of the mass belongs to regular standard pairs. Calling \(\cH_{\iota}\) the standard families formed by conditioning on these regular standard pairs, we can then apply the argument as above to couple each element of \(\cH_1\) with any other of the family \(\cH_2\). The above argument yields the estimates \(N_{*} + N\) in place of \(N\) for the minimum time needed for coupling and \(p_c/2\) in place of \(p_c\) for the coupled mass. This concludes the proof of the Lemma.
\end{proof}
\begin{comment}
We conclude this section by upgrading Lemma \ref{lem:fundamental-coupling} to standard families.

\begin{lemma}\label{lem:fundamental-coupling-standar-families}
     For any \(\eta >0 \) there exists \(N_{\eta} \in \bN\) such that for any \(n \ge N_{\eta}\) there exists \(p \in (0,1)\) with the following property. For any two regular standard families \(\cG_{\iota}\) there exist two \(\eta\)-coupled regular standard pairs \(\ell_{\iota,c}\) and two regular standard families \(\cG_{\iota,u}\) such that, for any \(\phi \in \cC^0(\cM)\),
    \begin{equation}\label{eq:decomposition-coupling-curves-first-time}
    \mu_{\hat \cF^{n}_t\cG_{\iota}} (\phi) = p \mu_{\ell_c} (\phi) + (1-p)\mu_{\cG_{\iota,u}} (\phi).
    \end{equation}
\end{lemma}
\begin{proof}
    aaa
\end{proof}
\end{comment}

\subsection{Decoupling because of discontinuities}\label{subsec:dec-disc} \textit{In this section we define fake stable foliations for coupling and compute how much mass gets ‘decoupled' because of discontinuities.}

A problematic aspect of Definition \ref{def:stacking} is that the \(\eta\)-coupled mass is not guaranteed to be coupled after just some iterations of the map. This happens because we couple mass along vertical lines instead of stable manifolds. However, we find that the mass lost in the process is very small, since it is proportional to some power of the distance between the standard pairs \(\eta\). For any \(x\in \cM\) and \(n \in \bN_0\), recall that \(\cF^n (x) = (r_n, \vf_n)\) and let \(\cH_n(x)\) be the maximal connected component of 
\begin{equation}\label{eq:holonomy-def}
\cF^{-n} \biggl(\biggl\{(r_n,\vf) \in \cM: \vf \in \biggl[-\frac \pi 2, \frac \pi 2 \biggr]\biggr\}\setminus \cS_{-n,t}^{\bH} \biggr)
\end{equation}
containing \(x\). (The above could be empty). \(\cH_n (x)\) is an approximation of the stable manifold containing \(x\) and we call \(\{\cH_n(x)\}_{x\in \cM}\) \textit{fake stable foliation}. For two \(\eta\)-stacked unstable curves \(W_{\iota}\), \(\iota \in \{1,2\}\), \(n \in \bN\), set
\begin{equation}\label{eq:domain-holonomy}
\hat W_{\iota}^{(n)} = \{x \in W_{\iota}: \cH_n (x) \cap W_{3-\iota} \neq \emptyset\},
\end{equation}
and we define the \textit{fake holonomy map} \(\bold h_n : \hat W_{\iota}^{(n)} \to \hat W_{3-\iota}^{(n)}\) as
\[
\bold h_n(x) = \cH_n(x) \cap W_{3-\iota}.
\]
Note that \(3-\iota\) simply refers to the `other' curve. The map \(\bold h_n\) is well defined since, by transversality between stable and unstable curves, the intersection in \eqref{eq:domain-holonomy} is unique.
 For \(x \in \hat W_1^{(n)}\),  call \(H_x \subseteq \cH_n(x)\) the stable curve connecting \(x\) to \(\bold h_n (x)\). Since the vertical direction belongs to the large stable cone \(\cC^s\) we have that \(\cF^{k}(H_x)\) is aligned with \(\hat \cC^s\) for each \(k \in \{0,...,n-1\}\). Moreover, by \eqref{eq:curve-expanding}, for some \(c, C>0\), \(\Lambda_0 >1\) and any \(k \in \{0,...,n\}\),
\begin{equation}\label{eq:contraction-fake-stable}
\begin{split}
&|\cF^{-k}(\cF^k (H_x))| \ge c\Lambda_0^k |\cF^k(H_x)| \quad \text{or} \quad |\cF^{k} (H_x) | \le C \Lambda_0^{-k}|H_x|.
\end{split}
\end{equation}
Note also that \(\cF^k(H_x)\) is \(k\)-homogeneous by construction. Consider two \(\eta\)-stacked curves \(W_{\iota}\). The set \(\hat W_{\iota}^{(n)}\) where \(\bold h_n\) is defined consists of many, possibly infinite, connected subsets \(\hat W_{\iota,p}^{(n)}\) of \(W_{\iota}^{(n)}\). We choose the indexing such that, for each \(p\),
\[
\bold h_n (\hat W_{\iota}^{(n)}) = \hat W_{3-\iota}^{(n)}.
\]
The gaps between these domains correspond to points where the fake stable foliation meets a singularity or, near the boundary, where it ‘misses' the other unstable curve. Fix some \(\chi \in (0,1/2)\) that represents the contraction among stacked graphs.
\begin{lemma}\label{lem:stacking-images}
For any \(n \in \bN\) big enough and \(\hat W_{\iota,p}^{(n)} \subseteq \hat W_{\iota}^{(n)} \) with \(\bold h_n (\hat W_{1, p}^{(n)}) = \hat W_{2,p}^{(n)}\),  the two unstable curves \(\cF^n (\hat W_{\iota, p}^{(n)})\) are \(\chi \eta\)-stacked.
\end{lemma}
\begin{proof}
For \(x = (r,\vf) \in \hat W_{1,p}^{(n)}\), recall the notation \(H_x \subseteq \cH_n(x)\) for stable curve connecting \(x\) to \(\bold h_n (x)\). By the definition \eqref{eq:holonomy-def} of fake holonomy, \(\cF^n(H_x)\) are vertical lines, the curves  \(\cF^n (\hat W_{\iota, p}^{(n)})\) are graphs of two \(\cC^2\) functions \(\vf_{\cF^n(\hat W_{\iota,p}^{(n)})}\) defined on the same domain and \(|\vf_{\cF^n(\hat W_{1,p}^{(n)})}(r_n) - \vf_{\cF^n(\hat W_{2,p}^{(n)})} (r_n)| = |\cF^n (H_x)|\). Since \(\cH_n(x)\) is aligned with \(\hat \cC^s\) and \(W_{\iota}\) are \(\eta\)-stacked one has that \(|H_x| \le C_{\mathrm{cone}} \eta\). Therefore, by \eqref{eq:contraction-fake-stable}, the \(\cC^0\) distance between  \(\vf_{\cF^n(\hat W_{1,p}^{(n)})}\) and  \(\vf_{\cF^n(\hat W_{2,p}^{(n)})}\) is not greater than 
\[
\sup_{x \in \hat W_{1, p}^{(n)}}|\cF^n(H_x)| \le  \sup_{x \in \hat W_{1,p}^{(n)}} C \Lambda_0^{-n}|H_x| \le C \Lambda_0^{-n} C_{\mathrm{cone}}\eta \le \chi \eta,
\]
where in the last inequality we consider \(n\) big enough. For the bound on the \(\cC^0\) distance of the derivatives, consider any \(x = (r,\vf)\in \hat W_{1,p}^{(n)}\). Note that  \(\delta := |x - \bold h_n(x)| \le |H_x| \le C_{\mathrm{cone}}\eta\) and, setting \(\bold h_n (x) = (\bar r, \bar \vf)\), we have that
\[
\begin{split}
    \gamma := |\vf_{W_1}'(r) - \vf_{W_2}'(\bar r)| &\le |\vf_{W_1}'(r) - \vf_{W_2}'(r)| + |\vf_{W_2}'(r) - \vf_{W_2}'(\bar r)| \\
    &\le |\vf_{W_1}'(r) - \vf_{W_2}'(r)| + D_2 \delta.
\end{split}
\]
Here \(D_2\) is the bound on the second derivatives of standard curves. By the equation above and using the definition of \(\eta\)-stacked curves we have that \(\gamma \le (1+D_2 C_{\mathrm{cone}})\eta\). We need to bound
\[
\gamma_n : = |\vf_{\cF^n(\hat W_{1,p}^{(n)})}'(r_n) - \vf_{\cF^n(\hat W_{1,p}^{(n)})}'(r_n)|.
\]
This is the angle between \(\cF^n(W_1)\) and \(\cF^n(W_2)\) at \(\cF^n(x)\) and \(\cF^n(\bold h_n(x))\). By Lemma \ref{lem:contraction-angle-stacked-curves} below, \(\gamma_n \le C(\delta n/\Lambda^{n} + \gamma/\Lambda^{n})\) for some \(C>0\) and \(\Lambda>1\). Hence, the statement follows from the previous estimates for \(\delta\) and \(\gamma\) and by taking \(n\) big enough.
\end{proof}

\begin{lemma}\label{lem:contraction-angle-stacked-curves}\cite[Exercise 5.44]{MR2229799}
    Let \(n \in \bN\), \(x \in \hat W_1^{(n)}\) and \(\bold h_n(x) \in \hat W_2^{(n)}\). Denote by \(\gamma\) the angle between the curves \(W_1\) and \(W_2\) at the points \(x\) and \(\bold h_n(x)\) and by \(\delta = |x - \bold h_n(x)|\). Finally, for \(k \in \{1,..,n\}\), denote by \(\gamma_k\) the angle between \(\cF^k(W_1)\) and \(\cF^k(W_2)\) at \(\cF^k(x)\) and \(\cF^k(\bold h_n(x))\). Then, there exists \(C>0\) such that \(\gamma_k \le C(\delta k/\Lambda^{k} + \gamma/\Lambda^{k})\).
\end{lemma}
In the original reference, Lemma \ref{lem:contraction-angle-stacked-curves} is stated for two points joined by a real holonomy and the statement holds for arbitrary \(n\). The proof uses the fact that \(\cH_n(x)\) is aligned with the small stable cone and so is not affected by the fact that we consider fake stable holonomies. The next result is intuitive: if the fake holonomy map \(\bold h_n\) is not defined on some point \(x \in W_\iota\) there can be only be two reasons: either \(x\) is near the boundary of \(W_{\iota}\), or the fake stable foliation \(\cH_n(x)\) intersects a discontinuity. In the second case \(x\) is close to \(W_\iota \cap  \cS^{\bH}_{n,t}\).
\begin{lemma}\label{lem:intersection-singularity}
There exist \(C_{\mathrm{cone}}>0 \) such that, for any \(n \in \bN\) and \(x\in W_{\iota} \setminus \hat W_\iota\), either \(x \in [\partial W_\iota]_{ C_{\mathrm{cone}}\eta}\), or \(x \in B(x_{*}, C_{\mathrm{cone}}\eta)\) for some \(x_{*} \in  W_\iota \cap \cS^{\bH}_{n,t}\) (or both).
\end{lemma}
\begin{proof}
Call the two cases (a) and (b). We assume without loss of generality that \(x \notin \cS^{\bH}_{n,t}\) (otherwise we are in case (b) with \(x_* = x\)). Therefore, \(\cH_n(x) \neq \emptyset\) and both endpoints of \(\cH_n(x)\) belong to \(\cS^{\bH}_{n,t}\). Since \(W_{\iota}\) are \(\eta\)-stacked, there exists \(C_{\mathrm{cone}}>0\) such that any long enough stable curve intersecting \(W_{\iota} \setminus [\partial W_\iota]_{ C_{\mathrm{cone}}\eta}\) must intersect also \(W_{3-\iota}\). Let us assume that we are not in case (a). By the sentence above, it must be that \(\cH(x)\) terminates on a point \(\bar x \in \cS^{\bH}_{n,t}\) before intersecting the other curve. By Lemma \ref{lem:continuation-of-singularities}, there exists a curve \(S \subseteq \cS^{\bH}_{n,t}\) connecting \(\bar x\) to \(\partial \cM\). Moreover, \(S\) is aligned to the stable cone and so it is transversal to \(W_{\iota}\). By possibly choosing a even bigger \(C_{\mathrm{cone}}\) in point (a), we can ensure that \(\bar x\) does not belong to a \(C_{\mathrm{cone}}'\eta\) neighborhood of \(\partial W_{\iota}\) for \(C_{\mathrm{cone}}'>0\) sufficiently big. Hence, \(S\) must intersect \(W_{\iota}\) in a point \(x_{*}\). Finally, for some other \(C_{\mathrm{cone}}''>0\), \(|x-x_{*}| \le |x-\bar x| + |\bar x - x_{*}| \le 2 C_{\mathrm{cone}}'' \eta\). After renaming the constants, this concludes the proof.
\end{proof}

The next Lemma tells that, for any fixed iterate, the complexity is uniform in the parameter of the perturbation.

\begin{lemma}\label{lem:complexity-small-pieces}
For any \(n \in \bN\), there exists \(K_n \in \bN\) such that for any unstable curve \(W\) (even longer than \(\delta_{*}\)) and \(t \in [0,t_0]\), 
    \[
   \# \{W \cap \cS_{n,t}\} \le K_n.
    \]
\end{lemma}
\begin{proof}
For any \(n \in \bN\), \(\cS_{n,t} = \cS_n \cup \Xi_{n,t}\) is a finite union of \(\cC^1\) curves. We say that a family \(\gamma_t: [0,1] \to \cM\) of curves varies continuously with \(t\) if \(t \to \gamma_t\) is a continuous function w.r.t.\! the \(\cC^0\) topology on the space of curves. We have that
\[
\cS_n \cup \Xi_{n,t} = \cS_n \cup \bigcup_{k=0}^{n}\cF^{-k}(\partial \cH_t \setminus \cS_{-k}).
\]
By transversality, \(\partial \cH_t \setminus \cS_{-k}\) is a finite union of curves which varies continuously with \(t\) (admitting the presence of degenerate elements, curves are allowed to gradually `grow or shrink' to a point) and by continuity of \(\cF^{-k}\) on \(\cM \setminus \cS_{-k}\) also \(\cS_{n,t} =\cS_n \cup \Xi_{n,t}\) is a finite union of curves which vary continuously with \(t\). Therefore, by the established continuity, there is a uniform lower bound on the distance between those curves in \(\cS_{n,t}\) that never intersect for \(t \in [0,t_0]\). I.e., if the intersection of two \(\cC^1\) curves is empty for all \(t \in [0,t_0]\), then they are at least \(\delta_{**}\) apart, for some \(\delta_{**} >0\), uniformly in \(t\). Hence, any unstable curve with length \(\delta_{**}/2\) can intersect \(\cS_{n,t}\) in at most \(M_n\) points, where \(M_n\) is the maximal number of intersections in a point for curves in \(\cS_{n,t}\) over all \(t \in [0,t_0]\). By Lemma \ref{lem:bound-number-smooth-components}, there are only finitely many smooth components in \(\cS_{n,t}\) uniformly in \(t\), so that \(M_n < \infty\). To conclude, we divide each unstable curves in \(\delta_{**}/2\)-long components and, since any unstable curve (possibly longer than \(\delta_*\)) is smaller than some constant \(\tilde C>0\), related to the diameter of the table, we prove the statement with \(K_n = M_n (2 \tilde C /\delta_{**} + 1)\).
\end{proof}

 Let \(W\) be an unstable curve. If \(\cS^{\bH}_{n,t} \cap W \neq \emptyset\), there are typically accumulation points of \(\cS^{\bH}_{n,t}\) on \(W\). In the next Lemma we characterize how fast these discontinuities accumulate around points in \(\cS^{\bH}_{n,t} \cap W\). We define the enlargement \(W_{e}\) as an unstable curve that contains \(W\) so that \(2\delta_{*} \le |W_{e}| \le 3\delta_{*}\) and \(W\) belongs to the ‘center' of \(W_e\). If the extension of \(W\) intersects the boundary \(\partial \cM\) we do not extend \(W_e\) any further independently of the length of \(W_e\). We obtain \(W_e\) by extending the function \(\vf_{W}\) to a bigger interval in a way that it belongs to \(\cW\) except for the requirement of the length being less than \(\delta_*\) (this is possible thanks to the standard extension theorems, while the actual way we extend \(\vf_{W}\) does not play any role in what follows). In practice, these extensions are a tool to book-keep the intersections of unstable curves with \(\cS^{\bH}_{n,t}\) in terms of the intersections with \(\cS_{n,t}\). In the next Lemma we fix \(k_0\) from \eqref{eq:homogeneity} large enough.
\begin{lemma}\label{lem:accumulation-disc}
There exists \(C>0\) with the following property. For any \(n \in \bN\), there exists \(k_0 \in \bN\) such that, for any \(k \ge k_0\), \(W \in \cW\) and \(x \in W \cap \cup_{s=0}^n \cF^{-s}(\bS_k)\), there exists \(z \in W_e \cap \cS_{n}\) with \(|x - z| \le C k^{-2}\).
\end{lemma}
\begin{proof}
For each \(s \in \{0,..,n\}\), there exists \(k_0 \in \bN\) big enough such that, for any \(W \in \cW\), if \(W \cap \cF^{-s}(\bS_k)\) for some \(k \ge k_0\) is non empty then also \(W_e \cap \cS_{s} = W_{e} \cap \cF^{-s}(\cS_{-s})\) is non empty. This is possible since \(\bS_{k}\), \(k \ge k_0\), approaches \(\cS_0\) as \(k_0 \to \infty\). Therefore, since \(\cF^{-s}(\bS_k)\) intersects \(W\), by continuity of \(\cF^{-s}\) on \(\cM \setminus \cS_{-s}\), we have that \(\cF^{-s}(\cS_{-s}) = \cS_s\) intersect the enlargement \(W_e\), which is at least \(\delta_{*}\) longer than the original \(W\). We set \(k_0\) to be the largest value so obtained for all \(s \in \{0,..,n\}\). Fix some \(s \in \{0,...,n\}\) and consider the curve \(\gamma \subseteq W_{e}\) that connects \(x \in W \cap \cF^{-s}(\bS_k)\) to  the first point \(z \in W_e \cap \cS_s\) and such that it is oriented in the direction of \(\cF^{-s}(\bH_p)\) with \(p\ge k\). In particular, we construct \(\gamma\) so that \(\cF^{s}(\gamma)\subseteq \cF^s(W_e)\) is a curve starting at \(\cF^s(x) \in \bS_k\), oriented in a way that \(\cF^s(\gamma) \subseteq \cup_{p\ge k}\bH_p\) (i.e., it goes in the direction of the singularity) and such that it terminates on the first intersection with \(\cF^s(\cS_s) = \cS_{-s}\). The inclusion \(\cF^s(\gamma) \subseteq \cup_{p\ge k}\bH_p\), together with the fact that \(\cF^s(\gamma)\) is an unstable curve, implies that \(|\cF^{s}(\gamma)| \le C_{\mathrm{cone}} k^{-2}\). Therefore, by \eqref{eq:curve-expanding}, for some \(C>0\) and \(\Lambda_0 >1\),
\[
|x-z| \le |\gamma| \le C \Lambda_0^{-s} |\cF^s(\gamma)| \le C C_{\mathrm{cone}} \Lambda_0^{-s} k^{-2} \le C k^{-2}.
\]
Since the above estimate holds for any \(x \in W\cap \cF^{-s}(\bS_k)\) and for any \(s \in \{0,..,n\}\), the equation above concludes the proof.
\end{proof}

On a couple of occasions, we will need to compute the effect on the boundary of conditioning on the event of being or not being stacked. Let \(\cG_{\iota} =\{(p_{j,\iota}, W_{j,\iota},\rho_{j,\iota})\}\) be two regular standard families such that \(W_{j,1}\) and \(W_{j,2}\) are \(\eta\)-stacked and let \(\cG_{\iota}^{(n)}\) and \(\bar \cG_{\iota}^{(n)}\) the standard families obtained by conditioning on the sets \( \hat W_{j,\iota}^{(n)}\) and \(W_{j,\iota} \setminus \hat W_{j,\iota}^{(n)}\) introduced in \eqref{eq:domain-holonomy}. More precisely, denote by \(\hat W_{j,p,\iota}^{(n)}\) and \(\bar W_{j,q,\iota}^{(n)}\) the maximal connected components of \(\hat W_{j,\iota}^{(n)}\) and \(W_{j,\iota} \setminus \hat W_{j,\iota}^{(n)}\) (possibly further refined if their \(n^{\mathrm{th}}\) images are longer than \(\delta_{*}\)) and set
\begin{equation}\label{eq:conditioning-def}
\begin{split}
&\cG_{\iota}^{(n)} = \biggl\{\biggl(\frac{p_{j,\iota} \int_{I_{\hat W_{j,p,\iota}^{n}}} \rho_{j,\iota}}{\sum_{j} p_{j,\iota} \int_{I_{\hat W_{j,p,\iota}^{(n)}}} \rho_{j,\iota}}, \hat W_{j,p,\iota}^{(n)}, \frac{\rho_{j,\iota}\Id_{I_{\hat W_{j,p,\iota}^{n}}}}{\int_{ \hat W_{j,p,\iota}^{(n)}}\rho_{j,\iota}} \biggr)\biggr\} \quad \text{and} \\
&\bar \cG_{\iota}^{(n)} = \biggl\{\biggl(\frac{p_{j,\iota} \int_{I_{\bar W_{j,q,\iota}^{n}}} \rho_{j,\iota}}{\sum_{j} p_{j,\iota} \int_{I_{\bar W_{j,q,\iota}^{(n)}}} \rho_{j,\iota}}, \bar W_{j,q,\iota}^{(n)}, \frac{\rho_{j,\iota}\Id_{I_{\bar W_{j,q,\iota}^{n}}}}{\int_{ \bar W_{j,q,\iota}^{(n)}}\rho_{j}}\biggr)\biggr\}.
\end{split}
\end{equation}

\begin{lemma}\label{lem:conditioning}
 For any \(n, \eta >0\) and any two regular standard family \(\cG_{\iota}\) as above,
    \[
    \cZ(\hat \cF_t^{n}\cG_{\iota}^{(n)}) \le  \frac{C_{\mathrm{cone}}e^{\Omega \varpi_2}\cZ(\hat \cF_t^{n}\cG_{\iota})}{\sum_{j} p_{j,\iota} \int_{I_{\hat W_{j,p,\iota}^{(n)}}} \rho_{j,\iota}} \quad \text{and}\quad \cZ(\hat \cF_t^{n}\bar\cG_{\iota}^{(n)}) \le  \frac{C_{\mathrm{cone}}e^{\Omega \varpi_2}\cZ(\hat \cF_t^{n}\cG_{\iota})}{\sum_{j} p_{j,\iota} \int_{I_{\bar W_{j,q,\iota}^{(n)}}} \rho_{j,\iota}}.
    \]
\end{lemma}
\begin{proof}
By Lemma \ref{lem:continuation-of-singularities} on the continuation of singularity lines, if two points lie in different connected components of \(\hat W_{j,\iota}^{(n)}\) then they must be separated by a a discontinuity, and hence lie in different components of \(W_{j,\iota} \setminus \cS_{n,t}^{\bH}\). Consequently, each element \(\hat W_{j,p,\iota}^{(n)}\) corresponds to a unique unstable curve \(W_{j,p,\iota}^{n}\) in the standard family \(\hat \cF_t^{n}\cG_{\iota}\) such that \(\cF^n(\hat W_{j,p,\iota}^{(n)}) \subseteq W_{j,p,\iota}^{n}\). Using part c) of Corollary \ref{cor:useful-density},
\[
 \int_{I_{\hat W_{j,p,\iota}^{n}}} \rho_{j,\iota}\frac{1}{|\hat \cF_t^{n}(\hat W_{j,p,\iota}^{(n)})|} =  \int_{I_{\hat \cF^{n}_t(\hat W_{j,p,\iota}^{n})}} \cL^n \rho_{j,\iota}\frac{1}{|\hat \cF_t^{n}(\hat W_{j,p,\iota}^{(n)})|} \le C_{\mathrm{cone}}e^{\Omega \varpi_2}\int_{I_{W_{j,p,\iota}^n}} \cL^n \rho_{j,\iota} \frac{1}{|W_{j,p,\iota}^n|}.
\]
Therefore, denoting by \(R = \sum_{j} p_{j} \int_{I_{\hat W_{j,p,\iota}^{(n)}}} \rho_{j,\iota}\) the total mass of the stacked family, using the definition in Equation \eqref{eq:conditioning-def},
\[
\begin{split}
 \cZ(\hat \cF_t^{n}\cG_{\iota}^{(n)}) &= R^{-1}\sum_{j} p_{j,\iota} \sum_p \int_{I_{\hat W_{j,p,\iota}^{n}}} \rho_{j,\iota}\frac{1}{|\hat \cF_t^{n}(\hat W_{j,p,\iota}^{(n)})|} \\
 &\le R^{-1}  C_{\mathrm{cone}}  e^{\Omega \varpi_2}\sum_j p_{j,\iota} \sum_p\int_{I_{W_{j,p,\iota}^n}} \cL^n \rho_{j,\iota} \frac{1}{|W_{j,p,\iota}^n|} \le R^{-1}C_{\mathrm{cone}} e^{\Omega\varpi_2}\cZ(\hat \cF_t^{n}\cG_{\iota}).
 \end{split}
\]
The proof of the second inequality is completely analogous with the only difference that each component of \(W_{j,\iota} \setminus \cS_{n,t}^{\bH}\) corresponds to at most two unstable curves among the family composed by \(\hat \cF_t^{n}(\bar W_{j,p,\iota})\). 
\end{proof}

We are ready to compute how much mass gets decoupled for two \(\eta\)-coupled standard families because of discontinuities. We find that if they start sufficiently close, they lose only \(\cO(\eta^{1/2})\) mass, during all their future evolution. Note however that the decoupled mass has very bad regularity (how bad depends on when it decouples), so that we need to implement some scheme in order to book-keep everything. This is done in Section \ref{subsec:linear-scheme} and the net effect is exponential mixing. Recall \(\chi \in (0,1/2)\) from Lemma \ref{lem:stacking-images} and \eqref{eq:domain-holonomy} for the notation \(\hat W_{\iota}^{(n)}\).

\begin{lemma}\label{lem:decouple-single-sp}
    There exist \(C, \eta_0 >0\) and \(k_0\in \bN\) such that, for any \(\eta<\eta_0 \), \(n \in \bN\), \(t \in [0,t_0]\) and any two \(\eta\)-coupled regular standard family \(\cG_{\iota} = \{(p_{j}, W_{j,\iota}, \rho_{j})\}\), for each \(\iota \in \{1,2\}\),
    \[
    \sum_{j} p_{j} \int_{I_{W_{j,\iota} \setminus \hat W^{(n)}_{j,\iota}}} \rho_{j} \le C \eta^{\frac 1 2}.
    \]
\end{lemma}
\begin{proof}
 Without loss of generality, we consider the case \(\iota =1\) and we set \(\cG_{1} = \{(p_{j}, W_{j}, \rho_{j})\}\). By part c) of Corollary \ref{cor:Growth-Lemma}, there exists \(n_0 \sim \ln B_2\) big enough such that \(\cZ(\hat \cF_t^{n_0} \tilde \cG) \le \bold B_2\) for any regular standard family \(\tilde \cG\). Let also \(n_0\) be big enough so that Lemma \ref{lem:stacking-images} is satisfied. We compute the lost mass by iterating \(\hat \cF_t^{n_0}\). According to Lemma \ref{lem:intersection-singularity}, we have \(W_{j} \setminus \hat W_{j}^{(n_0)} \subseteq  W_{j}^A \cup  W_{j}^B\), where
\[
\begin{split}
 W_{j}^A &= \{x \in W_{j}: \text{ there exists } z \in W_{j}\cap \cS_{n_0, t}^{\bH} \text{ such that } x \in B(z, C_{\mathrm{cone}}\eta)\}, \\
W_{j}^B &= \{x \in W_{j}: x \in [\partial W_{j}]_{C_{\mathrm{cone}}\eta}\},
\end{split}
\]
for some \(C_{\mathrm{cone}}>0\) big enough. Since \(\cG_1\) is a regular standard family, \(\rho_j \in \cC(\varpi_2)\) and by Lemma \ref{lem:ubiquous-cone}, we have
\begin{equation}\label{eq:estimate-boundary}
    \int_{I_{W_{j}^B}} \rho_{j} \ \le \frac{ C_{\mathrm{cone}}e^{\varpi_2} |I_{W_{j}^B}|}{|W_{j}|} = \frac{C \eta}{|W_{j}|}.
\end{equation}
As for the mass supported on \(W_{j}^A\), we need to study the intersection of \(W_{j}\) and \(\cS_{n_0, t}^{\bH}\). Recall the enlargements \(W_{j,e}\) of \(W_{j}\) defined before Lemma \ref{lem:accumulation-disc} and that \(\sup_{t \in [0,t_0]}\# \{ W_{j,e} \cap \cS_{n_0,t}\} \le K_{n_0}\) by Lemma \ref{lem:complexity-small-pieces}. We set
\begin{equation}\label{eq:list-intersection}
\begin{split}
    W_{j, e} \cap \cS_{n_0,t} &= \{z_{p}\}_{p \in \cI_1}, \quad \# \cI_{1} \le K_{n_0},\\
   W_{j} \cap (\cS_{n_0}^{\bH} \setminus \cS_{n_0}) & =  W_{j} \cap \bigcup_{k=k_0}^{\infty} \bigcup_{s=0}^{n_0}  \cF^{-s}(\bS_k)  = \{x_{r}\}_{r \in \cI_2}, \quad \cI_{2} \subseteq \bN.
\end{split}
\end{equation}
By Lemma \ref{lem:accumulation-disc}, there exist \(C >0 \) and \(k_0 >0\) such that any \(x \in W_{j} \cap \cup_{s=0}^{n_0}\cF^{-s}(\bS_k)\), \(k \ge k_0\), belongs to a \(C k^{-2}\)-neighborhood of some \(z \in W_{j,e} \cap \cS_{n_0,t}\). Therefore for some countable set \(\{x_{q_{k,p}}\}_{q_{k,p}} \subseteq W_{j}\), we have
\begin{equation}\label{eq:re-indexing-decoupling}
\begin{split}
 W_{j} \cap (\cS_{n_0}^{\bH} \setminus \cS_{n_0}) \subseteq \bigcup_{p \in \cI_1}& \bigcup_{k =k_0}^{\infty}  \bigl\{x_{q_{k,p}} \in W_{j} \cap \cup_{s=0}^{n_0}\cF^{-s}(\bS_k):  \text{ } |x_{q_{k,p}} - z_p| \le C k^{-2}\bigr\}, \\
 q_{k,p}\in \cJ_{k,p}, \quad \#\cJ_{k,p} &\le \# \{W_{j} \cap \cup_{s=0}^{n_0}\cF^{-s}(\bS_k)\}.
 \end{split}
\end{equation}
Here, we are labeling the intersections of \(W_j\) with the extended discontinuity set according to the \(\bS_k\) they originate from. By Lemma \ref{lem:complexity-small-pieces} with \(\cup_{s=0}^{n_0}\cF^{-s}(\bS_k)\) in place of \(\cS_{n_0} \subseteq \cS_{n_0,t}\), we have that \(\#\cJ_{k,p}  \le K_{n_0}\) for any \(k\) and \(p\). By definition of \(W_{j}^A\), and equations \eqref{eq:list-intersection} and \eqref{eq:re-indexing-decoupling},
\[
W_{j}^A \subseteq [W_{j} \cap \cup_{p} B(z_p, C_{\mathrm{cone}}\eta)] \cup [W_{j} \cap \cup_{p} \cup_k \cup_{q_{k,p}} B(x_{q_{k,p}}, C_{\mathrm{cone}}\eta)]
\]
and
\begin{equation}\label{eq:bound-density-12}
\int_{I_{ W_{j}^A}} \rho_{j} \le \int_{ I_{W_{j} \cap \cup_{p} B(z_{p}, C_{\mathrm{cone}}\eta )} } \rho_{j}  + \int_{I_{W_{j} \cap \cup_{p} \cup_k \cup_{q_{k,p}} B(x_{q_{k,p}}, C_{\mathrm{cone}}\eta)}} \rho_{j} .
\end{equation}
For the first term, using that \(\# \cI_{1} \le K_{n_0}\), we have, for some \(C_{n_0} >0\),
\[
    \int_{ I_{W_{j} \cap \cup_{p} B(z_{p}, C_{cone}\eta )} } \rho_{j} \le \sum_{p \in \cI_1} \int_{ I_{W_{j} \cap  B(z_{p}, C_{cone}\eta )} } \rho_{j} \le \frac{C_{\mathrm{cone}}e^{\varpi_2} K_{n_0} \eta }{|W_{j}|} = \frac{C_{n_0} \eta}{|W_{j}|}.
\]
The second contribution has to do with the mass lost near grazing collisions. On the one hand there are infinite singularities accumulating near a tangent collision, but on the other hand those singularities are closer and closer to each other, so that their contributions to decoupling do not pile up. Setting \(k_{*} = \lfloor  \eta^{-1/2}\rfloor \), for \(k \ge k_{*}\) and \(x \in B(x_{q_{k,p}}, C_{\mathrm{cone}}\eta)\), we have
\[
| x - z_p |  \le | x - x_{q_{k,p}}| + | x_{q_{k,p}} - z_p| \le C_{\mathrm{cone}} \eta + C k_{*}^{-2} \le (C_{\mathrm{cone}} + C)\eta = C \eta,
\]
where in the second inequality we used the defining property \eqref{eq:re-indexing-decoupling} of \(x_{q_{k,p}}\). Therefore,
\[
\begin{split}
&\cup_{p \in \cI_1}\cup_{k \ge k_0}\cup_{q_{k,p} \in \cJ_{k,p}} B(x_{q_{k,p}}, C_{\mathrm{cone}}\eta) \\
&= \cup_{p \in \cI_1 } \bigl(\cup_{k = k_0}^{k_*} \cup_{q_{k,p} \in \cJ_{k,p}}B(x_{q_{k,p}}, C_{\mathrm{cone}}\eta) \cup \cup_{k = k_* + 1}^{\infty} \cup_{q_{k,p} \in \cJ_{k,p}} B(x_{q_{k,p}}, C_{\mathrm{cone}}\eta) \bigr)\\
& \subseteq \cup_{p\in \cI_1} \bigl(\cup_{k = k_0}^{\lfloor \eta^{- 1/2} \rfloor} \cup_{q_{k,p} \in \cJ_{k,p}} B(x_{q_{k,p}}, C_{\mathrm{cone}}\eta)  \cup B(z_p, C \eta) \bigr).
\end{split}
\]
Hence, we have the bound, 
\begin{equation}\label{eq:second-term-decouples}
\begin{split}
\int_{I_{W_{j} \cap \cup_{p} \cup_k \cup_{q_k} B(x_{q_{k,p}}, C_{\mathrm{cone}}\eta)}}\rho_{j} &\le \sum_{p \in \cI_1} \sum_{k = k_0}^{\lfloor  \eta^{-1/2} \rfloor } \sum_{q_{k,p} \in \cJ_{k,p}} \int_{I_{W_{j}\cap B(x_{q_{k,p}}, C_{\mathrm{cone}}\eta)}} \rho_{j} \\
&\hspace{2cm}+ \sum_{p \in \cI_1}\int_{I_{W_{j} \cap B(z_{p}, C\eta)}} \rho_{j}.
\end{split}
\end{equation}
Using that the standard pair is regular and recalling that \(\# \cI_1 \le K_{n_0}\), \(\#\cJ_{k,p} \le K_{n_0}\), the quantity above is less than, for some \(C_{n_0} >0\),
\[
\frac{e^{\varpi_2} K_{n_0} ( C_{\mathrm{cone}} K_{n_0} \eta^{-\frac 1 2}\eta + C \eta)}{|W_{j}|} =  \frac{C_{n_0} \eta^{\frac 1 2 }}{|W_{j}|}.
\]
Collecting the previous estimates, we have obtained, for any standard pair in the family,
\[
 \int_{I_{W_{j} \setminus \hat W^{(n_0)}_{j}}} \rho_{j} \le \frac{C_{n_0}\eta^{\frac 1 2}}{|W_{j}|}.
\]
Summing over \(j\) and using the fact that \(\cG_1\) is a regular standard family, we have
\begin{equation}\label{eq:lost-mass-one-step}
    \sum_{j} p_{j} \int_{I_{W_{j} \setminus \hat W^{(n_0)}_{j}}} \rho_{j} \le C_{n_0}\eta^{\frac 1 2} \sum_j p_{j} \frac{1}{|W_{j}|}\le C_{n_0} B_2 \eta^{\frac 1 2}.
\end{equation}
We now upgrade the validity of \eqref{eq:lost-mass-one-step} to any \(n\). The idea is that, as stacked curves get closer and closer exponentially fast, the contributions to the lost mass at each step yield a convergent series. 

Let \(\eta_0 >0\) be so small that \(C_{n_0} B_2 \eta_0^{1/2} < 1/2\). Recall \eqref{eq:conditioning-def} for the definition of conditioning on the stacked part. In particular, \(\cG^{(k n_0)}_1\) are the standard families obtained by conditioning \(\cG_1\) on \(\hat W_j^{(kn_0)}\), i.e., on being the part that will be stacked in \(\hat\cF^{k n_0}_t\cG_1\). By Lemma \ref{lem:stacking-images}, the curves in \(\hat\cF^{k n_0}_t\cG^{(kn_0)}_1\) are \(\chi^{k}\eta\)-stacked with the unstable curves belonging to the other standard family \(\hat\cF^{k n_0}_t\cG_2^{(kn_0)}\). We first claim that \(\cZ(\hat\cF^{k n_0}_t\cG^{(kn_0)}_1) \le B_2\) for any \(k \in \bN_0\). For \(k =0\) the statement follows from the fact that \(\cZ(\cG^{(0)}_1) = \cZ(\cG_1) \le B_2\). Assume that the statement holds for some \(k\). Then, by Lemma \ref{lem:conditioning} and applying \eqref{eq:lost-mass-one-step} to the regular standard family \(\hat\cF^{k n_0}_t\cG_1^{(kn_0)} = \{(\tilde p_{j}, \tilde W_{j}, \tilde \rho_{j})\}\), we obtain
\[
\begin{split}
\cZ(\hat\cF^{(k+1) n_0}_t\cG^{((k+1)n_0)}_1)&=\cZ (\hat \cF_t^{n_0}(\hat\cF^{k n_0}_t\cG_1^{(kn_0)})^{(n_0)}) \le \frac{C_{\mathrm{cone}}e^{\Omega \varpi_2}\cZ (\hat \cF_t^{n_0}(\hat\cF^{k n_0}_t\cG_1^{(kn_0)}))}{\sum_{j}  \tilde p_{j} \int_{I_{\hat {\tilde{W}}^{(n_0)}_{j}}} \tilde \rho_{j}}  \\[8pt]
&=\frac{C_{\mathrm{cone}}e^{\Omega \varpi_2}\cZ (\hat \cF_t^{n_0}(\hat\cF^{k n_0}_t\cG_1^{(kn_0)}))}{1- \sum_{j}  \tilde p_{j} \int_{I_{\tilde W_{j} \setminus \hat {\tilde{W}}^{(n_0)}_{j}}} \tilde \rho_{j}} \le  \frac{C_{\mathrm{cone}}e^{\Omega\varpi_2}\bold B_2}{1-C_{n_0} B_2 \eta^{\frac 1 2}} \le B_2.
\end{split}
\]
In the last inequality we have used \eqref{eq:intrinsic-constant-boundary} and the fact that \(\eta_0\) is small and in the second to last we used that \(n_0\sim \ln B_2\) and the inductive hypothesis. This concludes the proof of the claim. Denote by \(m_k \in [0,1]\) the proportion of mass in \(\cG_1\) that is supported on \(\cG_1^{(kn_0)}\). Because  \(\hat \cF_t^{kn_0}\cG_1^{(kn_0)}\) are regular standard families, we can apply the argument before \eqref{eq:lost-mass-one-step} to these standard families and, by the mentioned equation, the mass lost from step \(k\) to step \(k+1\) is at most \(m_k C_{n_0} B_2 \eta^{1/2}\). Therefore, for any \(q\), 
\[
\begin{split}
 \sum_{j} p_{j} \int_{I_{W_{j} \setminus \hat W^{(q n_0)}_{j}}} \rho_{j} \le \sum_{k=0}^{q} m_k  C_{n_0} B_2 (\chi^k\eta)^{\frac 1 2}  \le \sum_{k=0}^{\infty} m_k  C_{n_0} B_2 (\chi^k\eta)^{\frac 1 2} \le C_{n_0} \eta^{\frac 1 2}.
 \end{split}
\]
In the last inequality we used that \(m_k \le 1\). To conclude, recall that \(n_0 \sim \ln B_2\) and \(B_2\) depends on the table and on the initial conditions \(\varpi_1, B_1\) only, so that the above quantity is less than \(C\eta^{1/2}\). Since for a general \(n\) the proportion of mass supported on \(\cG_1^{(n)}\) is more than the proportion of mass supported on \(\cG_1^{(qn_0)}\) for \(q> n/n_0\), the last equation concludes the proof of the Lemma.
\end{proof}

\subsection{Decoupling because of the mismatch of the densities}\label{subsec:decoupling-mismatch}\textit{In this section we compute how much mass gets ‘decoupled' because of the mismatch between the evolved densities of two coupled standard pairs.}

Let \(\ell_{\iota} = (W_{\iota}, \rho)\) be any two \(\eta\)-coupled regular standard pairs and recall \(\hat W^{(n)}_{\iota}\) from \eqref{eq:domain-holonomy}. In this section, we denote by \(V_{\iota}\) two arbitrary connected subsets of \(\hat W_{\iota}^{(n)}\) such that \(\bold h_n (V_1) = V_2\), and, for \(k \in \{0,...,n\}\), we also set \(V_{\iota,k} = \cF^k(V_{\iota})\) and \(I=I_{V_{1,n}} = I_{V_{2,n}}\). The following result is essentially a finite time version of \cite[Theorem 5.42]{MR2229799}. Recall the functions \(r^n_{1,W_{\iota}}\) from \eqref{eq:powers-expanding-base}.
\begin{lemma}\label{lem:estimates-mismatch-derivatives-different-graphs}
There exists \(C>0\) such that, for any \(n \in \bN\) and \(\eta>0\), for any \(r \in I\), 
    \[
\biggl|\ln \frac{|{r_{1, V_{1,n}}^{-n}}'(r)|}{|{r_{1, V_{2,n}}^{-n}}'(r)|} \biggr|\le C \eta^{\frac{1}{3}}. 
    \]
\end{lemma}
\begin{proof}
To any \(r \in I\) it corresponds a point \(x_{\iota} = (r, \vf_{V_{\iota,n}} (r))\) in \(V_{\iota,n}\). For \(k \in \{1,...,n\}\), set \(x_{k} = \cF^{-k}(r, \vf_{V_{1,n}} (r))\), \(\bar x_{k} = \cF^{-k}(r, \vf_{V_{2,n}} (r))\) and call \(x_{k} = (r_{k}, \vf_{k})\), \(\bar x_{k} = (\bar r_{k}, \bar \vf_{k})\). Then
\begin{equation}\label{eq:log-series-hol}
\ln \frac{|{r_{1, V_{1,n}}^{-n}}'(r)|}{|{r_{1, V_{2,n}}^{-n}}'(r)|} = \sum_{k=0}^{n-1} \ln |r_{1, V_{2, k}}' (r_{n-k})| - \ln | r_{1, V_{1, k}}' (\bar r_{n-k})|.
\end{equation}
Using the equation \eqref{eq:derivative-FG} for derivative of \(\rw\) and setting also \(\tau_j = \tau (x_{j})\), \(\kappa_j = \kappa(x_{j})\) and \(\bar \tau_j =  \tau (\bar x_{j})\), \(\bar \kappa_j = \kappa(\bar x_{j})\), we have
\begin{equation}\label{eq:log-1}
\begin{split}
    \ln |r_{1, V_{1, k}}' (r_{n-k})| & = - \ln \cos \vf_{n-k-1} \\
    &+ \ln (\tau_{n-k} \kappa_{n-k} + \cos \vf_{n-k} + \tau_{n-k} \vf_{V_{1,k}}'(r_{n-k}) ) 
\end{split}
\end{equation}
and similarly
\begin{equation}\label{eq:log-2}
\begin{split}
    \ln |r_{1, V_{2, k}}' (\bar r_{n-k})| & = - \ln \cos \bar \vf_{n-k-1} \\
    &+ \ln (\bar \tau_{k-n} \bar \kappa_{n-k} + \cos \bar \vf_{n-k} + \bar \tau_{n-k} \vf_{V_{2,k}}'(\bar r_{n-k}) ) .
\end{split}
\end{equation}
Let \(H_0 \subseteq \cH_n(x_n)\) be the curve connecting \(x_n\) to \(\bold h_n (x_n)\). Set \(H_k = \cF^k(H_0)\) and \(\delta_{k} = |H_k|\). By the definition \eqref{eq:holonomy-def} of fake holonomy, we have that \(H_k\) is weakly homogeneous for \(k \in \{0,1,...,n\}\). Recalling that \(x_{n-k}, \bar x_{n-k} \in H_k\), there exists \(C>0\) such that
\[
\begin{split}
    |\ln \cos \vf_{n -k -1} - \ln \cos \bar \vf_{n-k-1} | &\le C \frac{|\vf_{n -k - 1} - \bar \vf_{n -k- 1}|}{\cos \vf_{n -k-1}} \le C\frac{\delta_{k +1}}{\delta_{k + 1}^{\frac 2 3}} \le C \delta_{k +1}^{\frac{1}{3}},
\end{split}
\]
where in the second inequality we used Lemma \ref{lem:unstable-curves-cosine-k}. Denoting by
\[
\gamma_{k} = | \vf_{V_{1,k}}'(r_{n-k}) - \vf_{V_{2,k}}'(\bar r_{n-k}) |
\]
the angle between the tangent vectors to the unstable curves \(V_{\iota,k}\) and using that the argument of the logarithms in \eqref{eq:log-1} and \eqref{eq:log-2} are bounded from below, we obtain
\begin{equation}\label{eq:exp-decreasing-hol}
| \ln | r_{1, V_{1, k}}' (r_{n-k})| -  \ln | r_{1, V_{2, k}}' (\bar r_{n-k})| | \le C (\delta_{k+1}^{\frac 1 3} + \delta_{k} + \gamma_{k} ).
\end{equation}
Here we also used that \(\tau\) and \(\bar \tau\) are \(\cC^1\) as a function of \(r\) and \(r_1\) on \(V_{\iota, k}\subseteq \cM \setminus \cS_{-1}^{\bH}\) (see \eqref{eq:tau-regularity}) and so
\[
|\tau_{n-k} - \bar \tau_{n-k}| = |\tau (r_{n-k}, r_{n-k-1}) - \tau (\bar r_{n-k}, \bar r_{n -k-1})| \le C(\delta_{k} + \delta_{k+1}).
\]
By \eqref{eq:contraction-fake-stable} and  Lemma \ref{lem:contraction-angle-stacked-curves} there exists \(\Lambda >1\), \(C>0\) such that \(\delta_{k} \le C\Lambda^{-k} \delta_0\) and \(\gamma_k \le C (\delta_0 k/\Lambda^{k} + \gamma_0 /\Lambda^{k})\). Therefore, by \eqref{eq:exp-decreasing-hol} and \eqref{eq:log-series-hol}, there exists \(C>0\) such that, for all \(n \in \bN\),
\[
\biggl|\ln \frac{|{r_{1, V_{1,n}}^{-n}}'(r)|}{|{r_{1, V_{2,n}}^{-n}}'(r)|} \biggr| \le C (\delta_0^{\frac 1 3} + \gamma_0).
\]
Since \(W_{\iota}\) are \(\eta\)-stacked and \(H_0\) is a stable  curve, \(\delta_0 \le C_{\mathrm{cone}}\eta\). Moreover,
\[
\gamma_0 = |\vf_{W_1}'(r_n) - \vf_{W_2}'(\bar r_n)| \le |\vf_{W_1}'(r_n) - \vf_{W_2}'(r_n)| + |\vf_{W_2}'(r_n) - \vf_{W_2}'(\bar r_n)|,
\]
and the first quantity is less than \(\eta\) because \(W_{\iota}\) are \(\eta\)-stacked while the second is less than \(C\eta\) because \(|\vf_{W_2}''|\) is bounded and 
\begin{equation}\label{eq:displacemet-holonomy}
    |r_n - \bar r_n| \le |H_0| \le  C_{\mathrm{cone}} \eta.
\end{equation}
This concludes the proof of the statement.
\end{proof}

If we start with two \(\eta\)-coupled standard pairs their densities don't undergo the same evolution under \(\cL\) because they are supported on different curves. However, since the curves are close, the evolutions almost coincide.

\begin{lemma}\label{lem:mismatch-density}
There exists \(C>0\) such that, for all \(n \in \bN\) and \(\eta>0\), for any \(r \in I\),
 \[
\biggl|\ln \frac{\cL^n_{V_{1,n}} \rho (r)}{ \cL_{V_{2,n}}^n \rho (r)} \biggr |\le C  \eta^{\frac 1 3}.
 \]
\end{lemma}
\begin{proof}
Recalling \eqref{eq:transfer-op},
\[
\biggl|\ln \frac{\cL^n_{V_{1,n}} \rho (r)}{ \cL_{V_{2,n}}^n \rho (r)} \biggr| \le \biggl| \ln \frac{\rho \circ r_{1, V_{1,n}}^{-n}(r)}{\rho \circ r_{1, V_{2,n}}^{-n}(r)}\biggr| + \biggl|\ln \frac{|{r_{1, V_{1,n}}^{-n}}'(r)|}{|{r_{1, V_{2,n}}^{-n}}'(r)|}\biggr|. 
\]
Denote as before \( (r_{n}, \vf_{n}) = \cF^{-n}(r, \vf_{V_{1,n}} (r))\), \( (\bar r_{n}, \bar \vf_{n}) = \cF^{-n}(r, \vf_{V_{2,n}} (r))\). Using the fact that \(\rho \in \cC(\varpi_2)\) and \eqref{eq:displacemet-holonomy}, there exists \(C>0\) such that for all \(n\) and \(\eta\),
\[
\biggl|\ln\frac{\rho \circ r_{1, V_{1,n}}^{-n}(r)}{\rho \circ r_{1, V_{2,n}}^{-n}(r)}\biggr| \le \varpi_2 |r_n - \bar r_n|^{\frac{1}{3}} \le C \eta^{\frac 1 3}.
\]
This, together with Lemma \ref{lem:estimates-mismatch-derivatives-different-graphs}, concludes the proof of the statement.
\end{proof}

We now split the push-forward of the densities supported by the two curves into a coupled part \(\rho_*\) and uncoupled parts \(\rho_1\) and \(\rho_2\). We show that the mass that we are not able to couple is controlled by some power of the distance \(\eta\) between the original standard pairs. We find that the coupled density is regular while the other two densities (which we call \textit{leftover densities}) belong to the very bad cone \(\cC(C\eta^{-1})\). This does not compromise our coupling scheme as these densities recover after \(\sim \ln \eta^{-1}\) iterations. During this transient, before they recover, we find that their oscillations (precisely their \(\max/\min\) ratio) remain of order one because of Lemma \ref{lem:funny-cone}. Leveraging on this fact, we can control the evolution of the boundary \(\cZ\) of standard pairs supporting leftover densities.

\begin{lemma}\label{lem:mismatch-density2}
There exists \(C>0\) such that for all \(n \in \bN\) big enough and \(\eta>0\) small enough, there exist \(\rho_*, \rho_1, \rho_2: I \to \bR^+\) such that
    \[
\begin{split}
&\frac{\cL^n_{V_{1,n}} \rho}{\int_{ I} \cL^n_{V_{1,n}} \rho} = \rho_* + \rho_1, \quad \frac{\cL^n_{V_{2,n}} \rho}{\int_{I} \cL^n_{V_{2,n}} \rho}  = \rho_* + \rho_2 \quad \text{and}\quad \int_{I} \rho_{*} \ge 1- C\eta^{\frac{1}{3}}.
\end{split}
    \]
Moreover, \(\rho_{*} \in \cC(\varpi_2)\) and, for \(\iota \in \{1,2\}\), \(\rho_{\iota} \in \cC(C\eta^{-1})\) and \(\max \rho_{\iota} / \min \rho_{\iota} \le 2e^{\varpi_2}\).
\end{lemma}
\begin{proof}
Let \(\bar n \in \bN\) be given by part b) of Corollary \ref{cor:useful-density} and consider \(n \ge \bar n\). For some \(c_0 >0 \) to be fixed shortly, we define
\[
\begin{split}
    \rho_{*} &= \min \biggl\{\frac{\cL^n_{V_{1,n}} \rho}{\int_{ I} \cL^n_{V_{1,n}} \rho} , \frac{\cL^n_{V_{2,n}} \rho}{\int_{ I} \cL^n_{V_{2,n}} \rho} \biggr\}(1- c_0\eta^{\frac 1 3}),\\
    \rho_1 & = \frac{\cL^n_{V_{1,n}} \rho}{\int_{ I} \cL^n_{V_{1,n}} \rho}  - \rho_{*} \quad \text{and}\quad\rho_2  = \frac{\cL^n_{V_{2,n}} \rho}{\int_{ I} \cL^n_{V_{2,n}} \rho}  - \rho_{*}.
\end{split}
\]
By Lemma \ref{lem:mismatch-density}, there exists \(C>0\) such that, for all \(n \in \bN\) and \(\eta >0\) small enough, 
\[
\ \frac{\int_{ I} \cL^n_{V_{1,n}} \rho}{\int_{ I} \cL^n_{V_{2,n}} \rho} = \int_{I} (\cL^n_{V_{1,n}} \rho/\cL^n_{V_{2,n}} \rho) \frac{\cL^n_{V_{2,n}} \rho}{\int_{I} \cL^n_{V_{2,n}} \rho} \ge 1 - C\eta^{\frac 1 3}.
\]
Hence, using Lemma \ref{lem:mismatch-density} again, for any \(r \in I\) and \(\eta\) small enough,
\[
  \frac{\cL^n_{V_{2,n}} \rho \int_{I} \cL^n_{V_{1,n}} \rho}{\cL^n_{V_{1,n}} \rho \int_{I} \cL^n_{V_{2,n}} \rho} \ge 1 - 2C\eta^{\frac{1}{3}}.
\]
Renaming the constant, we have obtained that, for some \(C >0\),
\begin{equation}\label{eq:estimate-min-density-fraction}
  \min \biggl\{1, \frac{\cL^n_{V_{2,n}} \rho \int_{I} \cL^n_{V_{1,n}} \rho}{\cL^n_{V_{1,n}} \rho \int_{I} \cL^n_{V_{2,n}} \rho}\biggr\} \ge 1 - C\eta^{\frac{1}{3}}.
\end{equation}
Set \(c_0\) to be equal to the constant \(C\) in \eqref{eq:estimate-min-density-fraction}. Then,
\[
\begin{split}
\int_{I} \rho_{*} &= (1-c_0 \eta^{\frac{1}{3}}) \int_{I} \frac{\cL^n_{V_{1,n}} \rho}{\int_{I} \cL^n_{V_{1,n}} \rho} \min \biggl\{1, \frac{\cL^n_{V_{2,n}} \rho \int_{I} \cL^n_{V_{1,n}} \rho}{\cL^n_{V_{1,n}} \rho \int_{I} \cL^n_{V_{2,n}} \rho}\biggr\}\\
&\ge (1-c_0 \eta^{\frac{1}{3}})(1- c_0\eta^{\frac{1}{3}}) \ge 1 - 2c_0\eta^{\frac{1}{3}}.
\end{split}
\]
This proves the lower bound on the amount of coupled mass. The fact that \(\rho_{*} \in \cC(\varpi_2)\) is a consequence of the definition of \(\rho_*\) and the first part of Lemma \ref{lem:classical-cones}. As for the leftover densities, we have
\[
\begin{split}
\min \rho_1 &\ge \min_I \frac{\cL^n_{V_{1,n}} \rho}{\int_{I} \cL^n_{V_{1,n}} \rho}\biggl(1 - \max_I\min \biggl\{1, \frac{\cL^n_{V_{2,n}} \rho \int_{I} \cL^n_{V_{1,n}} \rho}{\cL^n_{V_{1,n}} \rho \int_{I} \cL^n_{V_{2,n}} \rho}\biggr\}(1 - c_0 \eta^{\frac 1 3})\biggr) \\
&\ge \frac{e^{-\delta_{*}^{1/3}\omega_2/2}}{|I|} [1 - (1-c_0 \eta^{\frac 1 3})] \ge \frac{c_0 e^{-\omega_2/2} \eta^{\frac 1 3 }}{|I|}.
\end{split}
\]
In the second inequality we have used that \(\cL^n_{V_{1,n}} \rho \in \cC(\varpi_2/2)\), which is true because of the choice of \(n\) specified at the beginning of the proof. Using \eqref{eq:estimate-min-density-fraction} again and the choice of \(c_0\),
\[
\begin{split}
    \max \rho_1 &\le \max_I \frac{\cL^n_{V_{1,n}} \rho}{\int_{I} \cL^n_{V_{1,n}} \rho}\biggl(1 - \min_I\min \biggl\{1, \frac{\cL^n_{V_{2,n}} \rho \int_{I} \cL^n_{V_{1,n}} \rho}{\cL^n_{V_{1,n}} \rho \int_{I} \cL^n_{V_{2,n}} \rho}\biggr\}(1 - c_0 \eta^{\frac 1 3 })\biggr)\\
    &\le \frac{e^{\delta_{*}^{1/3}\omega_2/2}}{|I|}\biggl(1 - (1-c_0\eta^{\frac 1 3 })(1-c_0 \eta^{\frac 1 3 }) \biggr)\le \frac{2e^{\omega_2/2}c_0\eta^{\frac 1 3}}{|I|}.
\end{split}
\]
Thus, \(\max \rho_1/\min \rho_1 \le 2e^{\omega_2}\). The estimate for \(\rho_2\) is analogous. Finally, for \(\eta\) small enough,
\[
\frac{\cL^n_{V_{1,n}} \rho}{\rho_* \int_{ I} \cL^n_{V_{1,n}} \rho} = \frac{1}{(1-c_0\eta^{\frac 1 3}) \min \biggl\{1, \frac{\cL^n_{V_{2,n}} \rho \int_{I} \cL^n_{V_{1,n}} \rho}{\cL^n_{V_{1,n}} \rho \int_{I} \cL^n_{V_{2,n}} \rho}\biggr\}} \ge 1 + \frac{c_0}{2}\eta^{\frac{1}{3}}.
\]
Therefore, for \(\eta\) small enough, by Lemma \ref{lem:classical-cones}, we have that \(\rho_{1} \in \cC(\varpi_*)\) with
\[
\varpi_* = \varpi_2 \biggl (1 + 2 \sup \frac{1}{\frac{\cL^n_{V_{1,n}} \rho}{\rho_* \int_{ I} \cL^n_{V_{1,n}} \rho} -1} \biggr) = \varpi_2 \biggl( 1 + 4 \frac{1}{c_0 \eta^{\frac 1  3}} \biggr)\le C \eta^{-\frac 1 3}.
\]
This concludes the proof.
\end{proof}

\subsection{Evolution of two coupled standard families}\label{subsec:evolution-two-coupled-sf}\textit{In this section we collect the results of Sections \ref{subsec:dec-disc} and \ref{subsec:decoupling-mismatch} in a single statement which records how much mass gets decoupled and the time needed to recover. The main result of this section is Lemma \ref{lem:one-step-loss-mass}.}

We consider a longer time step for the dynamics. This is a big enough iterate \(\bar N_c\) of \(\hat \cF_t\) such that a list of good properties is met. First, we require that 
\begin{equation}\label{eq:conditionN_c}
\bar N_c \ge \alpha_1 \ln B_2 + \alpha_2,
\end{equation}
where \(\alpha_1\) and \(\alpha_2\) are given by Corollary \ref{cor:Growth-Lemma}. We also assume that \(\bar N_c\) is big enough so that Lemmata \ref{lem:distortion-bound-density}, \ref{lem:stacking-images} and \ref{lem:mismatch-density2} apply. Informally, at the time scale \(\bar N_c\) stacked curves become even more stacked and the regularity data \((\varpi, B)\) have enough time to see contraction. 

We now discuss in detail the order of the quantifiers. We first fix the length scale \(\eta_0\) for coupling. The number \(\eta_0 > 0\) is small enough such that Lemma \ref{lem:one-step-loss-mass} is satisfied and depends only on the billiard table and the regularity data  \(\varpi_1, B_1\) of the initial standard families. With \(\eta_0\) fixed, we set \(N_c \ge \bar N_c\) depending on \(\eta_0\) so that Lemma \ref{lem:fundamental-coupling} is satisfied. This guarantees that a positive amount of mass can be coupled at the length scale \(\eta_0\) every \(N_c\) iterations of the dynamics. Finally, we set \(t_0\) small enough depending on \(N_c\). For example, by Lemma \ref{lem:lost-mass-with-t}, we can require \(t_0 >0\) to be small enough such that, for every regular standard family \(\cG\),
\begin{equation}\label{eq:condition-on-t_0}
    \mu_{\cG}(\cM_t^{N_c}) \ge \frac{3}{4},
\end{equation}
for all \(t \in [0,t_0]\) and \(\iota \in \{1,2\}\). To keep the presentation readable, the explicit order of all the quantifiers is specified only in Lemma \ref{lem:one-step-loss-mass}, which summarizes all the results of this section. We now introduce \textit{regularity classes} \(\fR_k\). These are a tool to book-keep very irregular mass and the time needed to become regular again.

\begin{definition}\label{def:regularity-classes}
We say that a standard family \(\cG\) belongs to \(\fR_{k}\), \(k \in \bN_0\), if \(\hat \cF_t^{kN_c} \cG\) is a regular standard family.
\end{definition}

Note that Definition \ref{def:regularity-classes} does not include the hole but only its derived discontinuities, even if ultimately we are interested in exponential mixing for the leaky evolution. Because of the Growth Lemma, mass with vary bad regularity takes only moderate time to recover.

\begin{lemma}\label{lem:characterization-regularity} There exists \(C>0\) such that the following is true. Let  \(\cG = \{(p_k, W_k, \rho_k)\}\) be a standard family such that \(\max \rho_k/ \min \rho_k \le \funny\) for each \(k\). For all \(0< \eta< 1/2\), if either \(\cG\) is a \((\omega_2, \eta^{-2})\)-standard family or a \((\eta^{-2}, B_2)\)-standard family, then  \(\cG \in \fR_{\lfloor C\ln \eta^{-1}\rfloor }\).
\end{lemma}
\begin{proof}
Take a \((\omega_2,  \eta^{-2})\)-standard family \(\cG\). Since \(N_c \ge \bar n\), by Lemma \ref{lem:distortion-bound-density} the standard families \(\hat \cF_t^{nN_c}\cG\) support densities in \(\cC(\varpi_2)\) for all \(n \in \bN\). Moreover, by part c) of Corollary \ref{cor:Growth-Lemma}, for any \(k \ge \alpha_1 \ln \eta^{-2} + \alpha_2\), we have that \(\cZ(\hat \cF_t^{k N_c} \cG) \le \bold B_2 \ll B_2\). Since \( \alpha_1 \ln \eta^{-2} + \alpha_2 \le \lfloor C \ln \eta^{-1} \rfloor\) for \(C>0\) big enough, this proves the first part of the statement. The second part follows in a similar way using part b) of Corollary \ref{cor:useful-density}, the assumption on the densities \(\rho_k\) and the invariance of \(\cZ(\cG)\) guaranteed by Lemma \ref{lem:invariance1}.
\end{proof}

Let \(\cG_{\iota} = \{(p_j, W_{j,\iota}, \rho_j)\}\) be two \(\eta\)-coupled regular standard families and denote by \(W_{j,p, \iota}\) the connected component of \(W_{j,\iota} \setminus \cS^{\bH}_{N_c, t}\) such that \(W_{j,p,\iota} \subset \cM_t^{N_c}\). By further subdividing them, we require as usual that \(|\cF^{N_c} (W_{j,p, \iota})| \le \delta_*\). Denote by \(\hat W_{j,p,\iota} = \hat W_{j,p,\iota}^{(N_c)} \subseteq  W_{j,p, \iota}\) the maximal connected subsets of \(W_{j, \iota}\) where \(\bold h_{N_c}\) is well defined and such that \(\bold h_{N_c} (\hat W_{j,p,1}) = \hat W_{j,p,2}\). Finally, call \(\bar W_{j,q, \iota}\) the connected components of \(W_{j,p, \iota} \setminus \hat W_{j,p, \iota}\) such that \(\bar W_{j,q, \iota} \subseteq \cM_t^{N_c}\). We are considering all the mass in \(\hat \cF_t^{N_c} \cG_{\iota}\) that will not end up in the hole in the next \(N_c\) iterations and we are further subdividing it according to whether it is supported on curves that are getting stacked or not. We now introduce two standard families \(\hat \cG_{\iota}\) and \(\bar \cG_{\iota}\). These are \textit{essentially} obtained by conditioning \(\cG_{\iota}\) on the events \(\hat W_{j,p,\iota}\) and \(\bar W_{j,p,\iota}\). However, in order to have coupled standard families, we need to make sure that \(\hat \cG_{1}\) and \(\hat \cG_{2}\) have the same weights, and so we need to declare uncoupled some mass supported on \(\hat W_{j,p,\iota}\) as well.  \mbox{More precisely, call}
\[
\begin{split}
 &\Omega_c = \sum_{j} p_j \sum_{p} \min_{\iota \in \{1,2\}}\int_{I_{\hat W_{j,p,\iota}}} \rho_j \quad \text{and}\\
 &\Omega_{u,\iota} = \sum_{j} p_j \sum_{q} \int_{I_{\bar W_{j,q,\iota}}} \rho_j + \sum_{j} p_j \sum_{p} \biggl(\int_{\hat W_{j,p,\iota}} \rho_j - \min_{\iota \in \{1,2\}}\int_{I_{\hat W_{j,p,\iota}}} \rho_j \biggr),
\end{split}
\]
the coupled and non-coupled mass respectively. Since some mass might end up in the hole, \(\Omega_c\) and \(\Omega_{u,\iota}\) don't necessarily sum up to one. We set 
\begin{equation}\label{eq:almost-coupled-def}
\begin{split}
&\hat \rho_{j,p,\iota} = \frac{\rho_j \Id_{I_{\hat W_{j,p, \iota}}}}{\int_{I_{\hat W_{j,p, \iota}}} \rho_j}, \quad \hat s_{j,p} = \frac{p_j \min_{\iota \in \{1,2\}}\int_{I_{\hat W_{j,p, \iota}}} \rho_j}{\Omega_c},\\
&\bar \rho_{j,q,\iota} = \frac{\rho_j \Id_{I_{\bar W_{j,q, \iota}}}}{\int_{I_{\bar W_{j,q, \iota}}} \rho_j}, \quad \bar s_{j,q,\iota} = \frac{p_j \int_{I_{\bar W_{j,q, \iota}}} \rho_j}{\Omega_{u,\iota}},
\end{split}
\end{equation}
and we consider the standard families 
\[
\begin{split}
&\hat \cG_{\iota} = \{( \hat s_{j,p}, \hat W_{j,p, \iota}, \hat \rho_{j,p,\iota})\}\quad \text{and} \\[10pt] 
&\bar \cG_{\iota} = \{(\bar s_{j,q,\iota}, \bar W_{j,q, \iota},\bar \rho_{j,q,\iota})\}\cup \biggl\{\biggl(\frac{p_j \bigl(\int_{I_{\hat W_{j,p, \iota}}} \rho_j - \min_{\iota \in \{1,2\}}\int_{I_{\hat W_{j,p, \iota}}} \rho_j\bigr)}{\Omega_{u,\iota}}, \hat W_{j,p, \iota}, \hat \rho_{j,p,\iota}\biggr)\biggr\} .
\end{split}
\]
Since \(\cG_{\iota}\) are regular standard families, by Lemma \ref{lem:decouple-single-sp},
\[
\begin{split}
    \sum_{j} p_j \sum_p \int_{I_{\bar W_{j,p, \iota}}} \rho_j \le \sum_{j} p_j \int_{I_{W_{j,\iota} \setminus \hat W_{j, \iota}}} \rho_j &\le C \eta^{\frac 1 2}.
\end{split}
\]
and, using Lemma \ref{lem:mismatch-density} and denoting by \(I_{j,p} = I_{\cF^n(\hat W_{j,p, 1})} = I_{\cF^n(\hat W_{j,p, 2})}\), for any \(\eta>0\) small enough, we have that
\begin{equation}\label{eq:difference-stacked-mass}
\begin{split}
 \sum_{j} p_j\sum_{p}&\biggl|\int_{I_{\hat W_{j,p, 1}}} \rho_j - \int_{I_{\hat W_{j,p, 2}}} \rho_j \biggr| = \sum_{j} p_j \sum_{p}\int_{I_{\hat W_{j,p, 2}}} \rho_j\biggl|\frac{\int_{I_{j,p}} \cL^n_{\cF^n(\hat W_{j,p, 1})} \rho_j}{\int_{I_{j,p}} \cL^n_{\cF^n(\hat W_{j,p, 2})} \rho_j} - 1 \biggr|\\[8pt]
 &\le \sum_{j} p_j \sum_{p}\int_{I_{\hat W_{j,p, 2}}} \rho_j\int_{I_{j,p}} \biggl|\frac{\cL^n_{\cF^n(\hat W_{j,p, 1})}\rho_j}{\cL^n_{\cF^n(\hat W_{j,p, 2})}\rho_j} - 1\biggr| \frac{\cL^n_{\cF^n(\hat W_{j,p, 2})}\rho_j}{\int_{I_{j,p}} \cL^n_{\cF^n(\hat W_{j,p, 2})} \rho_j}\le C\eta^{\frac 1 3}.
 \end{split}
\end{equation}
Therefore, the non-stacked mass \(\Omega_{u,\iota}\), which is bounded by the contributions from the last two equations, is small. In particular, by possibly coupling less mass, we can ensure that, for \(\eta\) sufficiently small,
\begin{equation}\label{eq:estimate-a}
   \Omega_c = 1 - \Omega_{u,\iota} - \mu_{\cG_{\iota}}(\cM \setminus \cM_t^{N_c}) \ge \frac{1}{2} \quad \text{and} \quad \frac{C \eta^{\frac 1 3}}{2} \le \Omega_{u,\iota} \le C \eta^{\frac 1 3}.
\end{equation}
For the lower bound for \(\Omega_c\), we use the upper bound on \(\Omega_{u,\iota}\) and the fact that,  according to \eqref{eq:condition-on-t_0}, the mass that enters the hole is less than \(1/4\) for \(t_0\) small enough. Let also \(\hat \cG_{\iota,N_c} = \hat \cF_t^{N_c} \hat \cG_{\iota}\) and \(\bar \cG_{\iota,N_c} = \hat \cF_t^{N_c} \bar \cG_{\iota}\). By construction, for any \(\phi \in \cC^0(\cM)\),
\begin{equation}\label{eq:first-dec}
\mu_{\cG_{\iota}} (\phi \circ \hat \cF_t^{N_c} \Id_{\cM_t^{N_c}}) = \Omega_c \mu_{\hat\cG_{\iota,N_c}}(\phi) + \Omega_{u,\iota} \mu_{\bar \cG_{\iota,N_c}}(\phi).
\end{equation}

\begin{lemma}\label{lem:reg-coupled-left-over}
For all \(\eta >0\) small enough, \(\hat \cG_{\iota, N_c}\) are regular standard families. Moreover, there exists \(C>0\) such that, for all \(\eta >0\) small enough, \(\bar \cG_{\iota, N_c} \in \fR_{\lfloor C\ln \eta^{-1}\rfloor }\).
\end{lemma}
\begin{proof}
By \eqref{eq:almost-coupled-def} and the fact that \(\cG_{\iota}\) are regular standard families, we have that \(\hat \rho_{j,p, \iota}\) and \(\bar \rho_{j,q,\iota}\) belong to \(\cC(\varpi_2)\). Consequently, by Lemma \ref{lem:distortion-bound-density} and the fact that \(N_c \ge \bar n\), we have that both \(\hat \cG_{\iota,N_c}\) and \(\bar \cG_{\iota, N_c}\) support regular densities. Recalling \eqref{eq:conditioning-def}, denote by \(\cG_{\iota}^{(N_c)}\) the standard families obtained by conditioning on being supported on the `matched part' of the holonomy and \(\hat\cF_t^{N_c}\cG_{\iota}^{(N_c)}\) their images, which coincide with conditioning \(\hat\cF_t^{N_c}\cG_{\iota}\) on the event of being stacked (here we are considering also the mass that ends up in the hole). Then, by Lemma \ref{lem:conditioning} and the fact that the decoupled mass is less than \(1/2\) for \(\eta\) small enough, we have that \(\cZ(\hat\cF_t^{N_c}\cG_{\iota}^{(N_c)}) \le 2e^{\Omega \varpi_2}C_{\mathrm{cone}}\cZ(\hat\cF_t^{N_c}\cG_{\iota})\). Therefore, by the choice \eqref{eq:conditionN_c} of \(N_c\) and part c) of Corollary \ref{cor:Growth-Lemma},
\[
\cZ(\hat\cF_t^{N_c}\cG_{\iota}^{(N_c)}) \le 2 e^{\Omega\varpi_2}C_{\mathrm{cone}}\bold B_2.
\]
Consequently, using also the lower bound for \(\Omega_c\) in Equation \eqref{eq:estimate-a}, we obtain that
\[
\begin{split}
\cZ(\hat \cG_{\iota, N_c}) &= \cZ(\hat \cF_t^{N_c} \hat \cG_{\iota}) = \sum_{j}\sum_p \hat s_{j,p} \frac{1}{|\hat\cF_t^{N_c}(\hat W_{j,p,\iota})|} = \frac{\sum_j p_j \sum_p \min_{\iota}\int_{I_{\hat W_{j,p,\iota}}}\rho_j \frac{1}{|\hat\cF_t^{N_c}(\hat W_{j,p,\iota})|}}{\Omega_c}\\
& \le \frac{\sum_j p_j \sum_p \int_{I_{\hat W_{j,p,\iota}}}\rho_j \frac{1}{|\hat\cF_t^{N_c}(\hat W_{j,p,\iota})|}}{\Omega_c} \le \frac{\cZ(\hat\cF_t^{N_c}\cG_{\iota}^{(N_c)})}{\Omega_c} \le 4 e^{\Omega\varpi_2}C_{\mathrm{cone}}\bold B_2 \le B_2.
\end{split}
\]
In the second inequality, we used that the sums go over the stacked curves that do not end up in the hole, whereas \(\cZ(\hat\cF_t^{N_c}\cG_{\iota}^{(N_c)})\) contains the contributions from all the stacked curves. This proves that \(\hat \cG_{\iota, N_c}\) are regular standard families. As for the second part of the statement, a similar computation based on Lemma \ref{lem:conditioning} yields, for \(\eta>0\) small enough,
\[
\begin{split}
\cZ (\bar \cG_{\iota, N_c}) &\le  \frac{\sum_{j} p_j \sum_q \int_{I_{\bar W_{j,q, \iota}}} \rho_j \frac{1}{|\hat \cF_t^{N_c}(\bar W_{j,q, \iota})|} + \sum_j p_j \sum_p \int_{I_{\hat W_{j,p,\iota}}}\rho_j \frac{1}{|\hat\cF_t^{N_c}(\hat W_{j,p,\iota})|}}{ \Omega_{u,\iota}}\\
&\le  \frac{C_{\mathrm{cone}}e^{\Omega \varpi_2}2\cZ(\hat\cF_t^{N_c}\cG_{\iota})}{ \Omega_{u,\iota}} \le  \frac{C_{\mathrm{cone}}e^{\Omega \varpi_2}2\bold B_2}{ \Omega_{u,\iota}}.
\end{split}
\]
Here, we used Lemma \ref{lem:conditioning} in the second inequality after dividing and multiplying for the correct renormalization. Hence, \(\cZ (\bar \cG_{\iota, N_c}) \lesssim 1/\Omega_{u,\iota} \le C \eta^{-1/3} \le \eta^{-2}\). Therefore, \(\bar \cG_{\iota, N_c}\) is a \((\varpi_2, \eta^{-2})\)-standard family and Lemma \ref{lem:characterization-regularity} concludes the proof.
\end{proof}

By Lemma \ref{lem:stacking-images} and because \(N_c\) is big enough, the two standard families \(\hat\cG_{\iota,N_c}\) are supported on \(\chi \eta\)-stacked unstable curves. However their densities \(\hat \rho_{j,p,\iota}\) may slightly differ. Recall that, according to Definition \ref{def:stacking}, we need the densities to be equal for declaring two standard pairs coupled. In the next Lemma, we take care of this ‘mismatch' between the densities and we estimate the regularity class of standard families supporting leftover densities.

\begin{lemma}\label{lem:holonomy-mismatch-decomposition}
    There exists \(C>0\) such that, for all \(\eta >0\) small enough, there exists \(b >0\),  \(b \le C\eta^{1/3}\), such that, for all \(t \in [0,t_0]\),
    \[
    \mu_{\hat \cG_{\iota, N_c}} = (1-b)\mu_{\hat \cG_{\iota,N_c, c}} + b\mu_{\hat \cG_{\iota,N_c, u}},
    \]
    where \(\hat \cG_{\iota, N_c, c}\) are \(\chi\eta\)-coupled regular standard families and \(\hat \cG_{\iota,N_c, u} \in \fR_{\lfloor C \ln \eta^{-1}\rfloor}\).
\end{lemma}
\begin{proof}
Recall that the standard families \(\hat \cG_{\iota,N_c} = \hat \cF_t^{N_c}\hat \cG_{\iota}\) support densities which are push-forwards of the densities \(\rho_j\) in \(\hat \cG_{\iota}\). By Lemma \ref{lem:mismatch-density2}, there exist \(\rho^{(c)}_{j,p} \in \cC(\varpi_2)\), \(\rho^{(u)}_{j,p,\iota} \in \cC(C \eta^{-1})\) with \(\max \rho^{(u)}_{j,p,\iota} / \min \rho^{(u)}_{j,p,\iota} \le 2e^{\varpi_2}\) such that
\[
\begin{split}
\frac{\cL^{N_c} \rho_j }{\int_{I_{\cF^{N_c}(\hat W_{j,p, \iota})}} \cL^{N_c}\rho_j} = \rho^{(c)}_{j,p} +\rho^{(u)}_{j,p,\iota} = \biggl(\int \rho^{(c)}_{j,p}\biggr) \frac{\rho_{j,p}^{(c)}}{\int \rho^{(c)}_{j,p}} + \biggl(\int \rho^{(u)}_{j,p,\iota}\biggr) \frac{\rho^{(u)}_{j,p,\iota}}{\int \rho^{(u)}_{j,p,\iota} }.
\end{split}
\]
Introducing \( b = \sup_j \int \rho^{(u)}_{j,p,\iota}\) and \(\tilde\rho^{(u)}_{j,p,\iota}\), which is some weighted average of \(\rho^{(u)}_{j,p,\iota} /(\int \rho^{(u)}_{j,p,\iota})\) and \(\rho^{(c)}_{j,p,\iota} /(\int \rho^{(c)}_{j,p,\iota})\), we have, for all \(j,p\),
\begin{equation}\label{eq:dec-densities}
    \frac{\cL^{N_c} \rho_j }{\int_{I_{\cF^{N_c}(\hat W_{j,p, \iota})}} \cL^{N_c}\rho_j} = (1-b) \frac{\rho_{j,p}^{(c)}}{\int \rho^{(c)}_{j,p}} + b \tilde\rho^{(u)}_{j,p,\iota}.
\end{equation}
The fact that \( b = \sup_j \int \rho^{(u)}_{j,p,\iota} = 1- \inf_{j}\int \rho^{(c)}_{j,p} \le C\eta^{1/3}\) is also part of Lemma \ref{lem:mismatch-density2}. Moreover, \(\tilde\rho^{(u)}_{j,p,\iota}\) belongs to \(\cC(C\eta^{-1})\) as well, because \(\rho_{j,p,\iota}^{(u)}\) does and \(\rho_{j,p}^{(c)}\) belongs to the much better cone \(\cC(\varpi_2)\). Equation \eqref{eq:dec-densities} induces a decomposition for the standard families \(\hat \cG_{\iota,N_c}\), which we denote with the same notation as in the statement of the Lemma. By Lemma \ref{lem:stacking-images} and the choice of \(N_c\), the curves \(\cF^{N_c}(\hat W_{j,p,\iota})\) are \(\chi \eta\)-stacked and the two standard families \(\hat \cG_{\iota, N_c, c}\) support the same probability density \(\rho^{(c)}_{j,p}/\int \rho^{(c)}_{j,p}\). This proves that \(\hat \cG_{\iota, N_c, c}\) are \(\chi \eta\)-coupled. It remains to study the regularity of the standard families. As we already mentioned, \( \rho^{(c)}_{j,p} \in \cC(\varpi_2)\). Since the decomposition \eqref{eq:dec-densities} does not change the unstable curves in the family nor the weights, we have \(\cZ(\hat \cG_{\iota,N_c, c}) = \cZ(\hat \cG_{\iota, N_c})\). Moreover, by Lemma \ref{lem:reg-coupled-left-over}, we have that \(\hat \cG_{\iota,N_c}\) are regular standard families, proving that \(\hat \cG_{\iota, N_c, c}\) are also regular. As for \(\hat\cG_{\iota, N_c, u}\), for \(\eta\) small enough, we have that  \(\tilde \rho^{(u)}_{j,p,\iota} \in \cC(C\eta^{-1}) \subseteq \cC(\eta^{-2})\) and since the decomposition \eqref{eq:dec-densities} does not affect the boundary, \(\cZ(\hat \cG_{\iota,N_c,u}) = \cZ(\hat \cG_{\iota,N_c}) \le B_2\). Finally, using that \(\max \tilde \rho^{(u)}_{j,p,\iota}/\min \tilde \rho^{(u)}_{j,p,\iota}\le 2e^{\varpi_2}\), we can apply Lemma \ref{lem:characterization-regularity} and \(\cG_{\iota,N_c, u} \in \fR_{\lfloor C\ln \eta^{-1}\rfloor }\). This concludes the proof of the Lemma.
\end{proof}

Recall that \(\cG_{\iota}\) are two \(\eta\)-coupled standard families. In the next result, we argue that if two standard families are close, the amount of mass that ends up in the hole for the two of them is also close.

\begin{lemma}\label{lem:difference-measure-hole}
There exists \(C>0\) such that, for all \(\eta>0\) small enough and \(t \in [0,t_0]\),
    \[
    | \mu_{\cG_{1}} (\cM \setminus \cM_t^{N_c}) -  \mu_{\cG_{2}} (\cM \setminus \cM_t^{N_c})| \le C \eta^{\frac{1}{3}}.
    \]
\end{lemma}
\begin{proof}
Recall the notation \(\cG_{\iota} = \{(p_j, W_{j,\iota}, \rho_j)\}\) for \(\iota \in \{1,2\}\). Then,
\[
\begin{split}
| \mu_{\cG_{1}} (\cM \setminus \cM_t^{N_c}) -  &\mu_{\cG_{2}} (\cM \setminus \cM_t^{N_c})| \le  \sum_{j} p_{j} \bigg| \int_{I_{W_{j,1}}} \rho_{j}(r) \Id_{\cM \setminus \cM_t^{N_c}}(r, \vf_{W_{j,1}}(r))dr \\
&- \int_{I_{W_{j,2}}} \rho_{j}(r) \Id_{\cM \setminus \cM_t^{N_c}}(r, \vf_{W_{j,2}}(r))dr \biggr|.
\end{split}
\] 
Moreover, partitioning \(W_{j,\iota} = \hat  W_{j,\iota} \cup (W_{j,\iota} \setminus \hat W_{j,\iota})\) in points which are linked by the fake holonomy and those that aren't, we have
\begin{equation}\label{eq:split-hole-and-holonomy}
\begin{split}
&\int_{I_{W_{j,1}}} \rho_{j}(r) \Id_{\cM \setminus \cM_t^{N_c}}(r, \vf_{W_{j,1}}(r))dr - \int_{I_{W_{j,2}}} \rho_{j}(r) \Id_{\cM \setminus \cM_t^{N_c}}(r, \vf_{W_{j,2}}(r))dr \\
&= \int_{I_{\hat W_{j,1}}} \rho_{j}(r) \Id_{\cM \setminus \cM_t^{N_c}}(r, \vf_{W_{j,1}}(r))dr - \int_{I_{\hat W_{j,2}}} \rho_{j}(r) \Id_{\cM \setminus \cM_t^{N_c}}(r, \vf_{W_{j,2}}(r))dr\\
&+\int_{I_{W_{j,1} \setminus \hat W_{j,1}}} \rho_{j}(r) \Id_{\cM \setminus \cM_t^{N_c}}(r, \vf_{W_{j,1}}(r))dr - \int_{I_{W_{j,2} \setminus \hat W_{j,2}}} \rho_{j}(r) \Id_{\cM \setminus \cM_t^{N_c}}(r, \vf_{W_{j,2}}(r))dr.
\end{split}
\end{equation}
For any point \(x \in \hat W_{j,1}\), we have that \(x\) belongs to \(\cM_t^{N_c}\) if and only if \(\bold h_{N_c} (x) \in \cM_t^{N_c}\). Indeed, \(x\) and \(\bold h_{N_c}(x)\) are connected via the curve \(\cH_{N_c} (x) \subset \cM \setminus \cS_{N_c,t}^{\bH}\). Hence, denoting by \(\hat W_{j,p,\iota}\) the maximal connected components of \(\hat W_{j,\iota}\), we have that \(\Id_{\cM\setminus\cM_t^{N_c}}\) is either one or zero on both \(\hat W_{j,p,1}\) and \(\hat W_{j,p,2}\). Therefore, by \eqref{eq:difference-stacked-mass}, the contribution derived from the terms in the first line of the equation above is bounded by 
\[
\begin{split}
&\sum_j p_j\sum_p \biggl|\int_{I_{\hat W_{j,p,1}}} \rho_{j}  - \int_{I_{\hat W_{j,p,2}}} \rho_{j}\biggr|\le C\eta^{\frac{1}{3}}.
\end{split}
\]
As for the contribution deriving from the second line of \eqref{eq:split-hole-and-holonomy}, we bound the difference with the sum of the two terms and, by Lemma \ref{lem:decouple-single-sp}, we obtain
\[
\sum_{j} p_j \int_{I_{W_{j,1} \setminus \hat W_{j,1}}} \rho_{j} + \sum_j p_j \int_{I_{W_{j,2} \setminus \hat W_{j,2}}} \rho_{j} \le C\eta^{\frac{1}{2}}.
\]
This concludes the proof of the Lemma.
\end{proof}
We conclude this section by summarizing in the next Lemma all we need to know about the iterations of two coupled standard families. In simple words, if we start with two coupled regular standard families, some mass will be even more coupled, some mass will decouple because of discontinuities and distortion (causing the appearance of the leftovers densities), and some mass will enter the hole. The mass that enters the hole amounts to a common normalization of the two families. Note however that the lost mass between the two families is slightly different, so that also the hole contributes to the decoupling. As we will see in the next section, this difference will appear in our computation in the form of another pair of uncoupled regular standard families \(\cE_{\iota}\). Importantly, by Lemma \ref{lem:difference-measure-hole}, the difference of lost mass is very small when the two coupled standard families are close. We group this mass with the uncoupled one. In this way, all the uncoupled mass is proportional to some power of the proximity \(\eta\) of the initial coupled standard families and the uncoupled part takes \(\sim \ln \eta^{-1}\) iterations to recover. Here, we are also more explicit about the orders of the quantifiers.
\begin{samepage}
\begin{lemma}\label{lem:one-step-loss-mass} 
There exist \(C, \eta_0, \bar N_c > 0\), depending only on the billiard table and the initial regularity, such that, for all \(N_c > \bar N_c\), there exists \(t_0>0\), such that for all \(t \in [0,t_0]\) the following is true. For all \(\eta \le \eta_0\) and any two \(\eta\)-coupled regular standard families \(\cG_{\iota}\) and any two regular standard families \(\cE_{\iota}\), \(\iota \in \{1,2\}\), there exist \(\Omega_c', \Omega_{u}' \ge 0\), \(\Omega_c' + \Omega_{u}' \le 1\), \(\Omega_{u}' \le \eta^{1/4}\), two \(\chi\eta\)-stacked regular standard families \(\cG_{\iota,c}\) and two standard families \(\cG_{\iota,u} \in \fR_{\lfloor C\ln \eta^{-1}\rfloor }\) such that, for all \(\phi \in \cC^0(\cM)\),
    \[
    \begin{split}
        \mu_{\cG_{\iota}} (\phi \circ \hat \cF^{N_c}_t \Id_{\cM_t^{N_c}}) + \bigl(\mu_{\cG_{\iota}} (\cM \setminus \cM_t^{N_c}) - \min_{\iota \in \{1,2\}} \mu_{\cG_{\iota}} (\cM \setminus \cM_t^{N_c}) \bigr)& \mu_{\cE_{\iota}} (\phi)\\
        &= \Omega_c' \mu_{\cG_{\iota,c}}(\phi) + \Omega_{u}' \mu_{\cG_{\iota,u}} (\phi).
    \end{split}
    \]
\end{lemma}
\end{samepage}
\begin{proof}
By \eqref{eq:first-dec} and Lemma \ref{lem:holonomy-mismatch-decomposition}, there exist \(\Omega_c, \Omega_{u,\iota}, b >0\) and standard families such that
\[
\begin{split}
       \mu_{\cG_{\iota}} (\phi \circ \hat \cF_t^{N_c} \Id_{\cM_t^{N_c}}) &= \Omega_c \mu_{\hat\cG_{\iota, N_c}}(\phi) + \Omega_{u,\iota} \mu_{\bar \cG_{\iota,N_c}}(\phi) \\
       &= \Omega_c \bigl((1-b)\mu_{\hat \cG_{\iota,N_c, c}} (\phi)+ b\mu_{\hat \cG_{\iota, N_c, u}} (\phi)\bigr)+ \Omega_{u,\iota} \mu_{\bar \cG_{\iota, N_c}}(\phi).
\end{split}
\]
The equation above yields the desired decomposition with 
\[
\Omega_{c}' = \Omega_c (1-b)\quad \text{and}\quad \Omega_{u}' = \Omega_c b + \Omega_{u,\iota} + \mu_{\cG_{\iota}} (\cM \setminus \cM_t^{N_c}) - \min_{\iota \in \{1,2\}} \mu_{\cG_{\iota}} (\cM \setminus \cM_t^{N_c}).  
\]
Accordingly, \(\cG_{\iota,c} = \hat \cG_{\iota,N_c, c}\) and \(\cG_{\iota,u}\) is a suitable convex combinations of \(\hat \cG_{\iota,N_c, u}\), \(\bar \cG_{\iota, N_c}\) and \(\cE_{\iota}\). Note that \(\Omega_c' +\Omega_u' = 1 - \min_{\iota} \mu_{\cG_{\iota}} (\cM \setminus \cM_t^{N_c})\), so that \(\Omega_u'\) does not depend on the family \(\iota\). Moreover, by Lemmata \ref{lem:holonomy-mismatch-decomposition} and \ref{lem:difference-measure-hole} and Equation \eqref{eq:estimate-a}, for \(\eta_0\) small enough and \(\eta \le \eta_0\),
\[
\Omega_{u}' \le b + \Omega_{u,\iota} + \mu_{\cG_{\iota}} (\cM \setminus \cM_t^{N_c}) - \min_{\iota \in \{1,2\}} \mu_{\cG_{\iota}} (\cM \setminus \cM_t^{N_c})  \le C \eta^{\frac 1 3} = C \eta_0^{\frac 1 {12}} \eta^{\frac 1 4}  \le \eta^{\frac 1 4},
\]
as required by the statement. By Lemma \ref{lem:reg-coupled-left-over} \(\cG_{\iota,c} = \hat \cG_{\iota, N_c, c}\) is a regular standard family and by Lemmata \ref{lem:reg-coupled-left-over} and \ref{lem:holonomy-mismatch-decomposition} \(\hat \cG_{\iota, N_c, u}\) and \(\bar \cG_{\iota, N_c} \in \fR_{\lfloor C \ln \eta^{-1}\rfloor}\). Moreover, \(\cE_{\iota}\) are regular standard families and so they belong to \(\fR_0 \subset \fR_{\lfloor C \ln \eta^{-1}\rfloor}\). Since a convex combination of standard families with a certain regularity has the same regularity, we have that \(\cG_{\iota,u} \in \fR_{\lfloor C \ln \eta^{-1}\rfloor}\).
\end{proof}
\subsection{Decay of Correlations}\label{subsec:linear-scheme}\textit{In this section we introduce a linear scheme that records the evolution in time of how much mass is coupled, at which distance, and the regularity of the uncoupled part. We use this to prove Theorem \ref{thm:exp-mixing-sf} about decay of correlations.}

Let \(\eta_0>0\), \(\eta_0 \le 10^{-10}\) be small enough so that Lemma \ref{lem:one-step-loss-mass} holds. We also let \(N_c\) be big enough such that Lemma \ref{lem:one-step-loss-mass} and Lemma \ref{lem:fundamental-coupling} hold for the given \(\eta_0\). We fix two initial regular standard families \(\cG_{\iota}\), \(\iota \in \{1,2\}\), as in the statement of Theorem \ref{thm:exp-mixing-sf}  and write \(\cG_{\iota}(k) = \cL_t^{kN_c}\cG_{\iota}\), \(k \in \bN_0\), for their leaky evolution. Recall that \(\chi \in (0,1/2)\) \mbox{is introduced before Lemma \ref{lem:stacking-images}.}

\begin{definition}[Coupled-uncoupled decomposition]\label{def:ckuk}
We say that the two standard families \(\cG_{\iota}(k)\) admit a \((c_r, u_r)\)-decomposition for \(r \in \bN_0\), \(c_r, u_r, \in [0,1]\), if there exist \(\chi^{r}\eta_0\)-coupled regular standard families \(\cG_{1,r}^c\), \(\cG_{2,r}^c\) and standard families \(\cG_{1,r}^u\), \(\cG_{2,r}^u\) in \(\fR_r\) such that,
\[
   \mu_{\cG_{\iota}(k)} =   \sum_{r = 0}^{\infty}u_r\mu_{\cG_{\iota,r}^u}  + \sum_{r = 0}^{\infty} c_r \mu_{\cG_{\iota,r}^c} .
\]
\end{definition}
Recall Definition \ref{def:regularity-classes} for the regularity classes \(\fR_r\). We are now ready to implement the coupling scheme for the leaky system. Using the temporary notation
\[
X(\phi) =  \mu_{\cG_{\iota}(k)}(\phi \circ \cF^{N_c} \Id_{\cM_t^{N_c}}) = \sum_{r = 0}^{\infty} u_r \mu_{\cG_{\iota,r}^u} (\phi \circ \cF^{N_c} \Id_{\cM_t^{N_c}}) + \sum_{r = 0}^{\infty} c_r \mu_{\cG_{\iota,r}^c} (\phi \circ \cF^{N_c} \Id_{\cM_t^{N_c}}),
\]
if the standard families \(\cG_{\iota}(k)\) admit a decomposition as in Definition \ref{def:ckuk}, we have 
\begin{equation}\label{eq:support1}
\begin{split}
    (\cL^{N_c}_t \mu_{\cG_{\iota}(k)}) (\phi) &= \frac{X(\phi)}{\mu_{\cG_{\iota}}(\cM_t^{N_c})} = X(\phi) + X(\phi) \biggl(\frac{1}{\mu_{\cG_{\iota}(k)}(\cM_t^{N_c})} -1\biggr) \\
    &= X(\phi) + (\cL_t^{N_c} \mu_{\cG_{\iota}(k)})(\phi) (1 - \mu_{\cG_{\iota}(k)}(\cM_t^{N_c})).
\end{split}
\end{equation}
Moreover,
\[
1 - \mu_{\cG_{\iota}(k)}(\cM_t^{N_c}) = \mu_{\cG_{\iota}(k)} (\cM \setminus \cM_t^{N_c}) = \sum_{r = 0}^{\infty}  u_r \mu_{\cG_{\iota,r}^u} (\cM \setminus \cM_t^{N_c}) + \sum_{r = 0}^{\infty}  c_r \mu_{\cG_{\iota,r}^c} (\cM \setminus \cM_t^{N_c}).
\]
Inserting the definition of \(X(\phi)\) and the expression above in \eqref{eq:support1}, we obtain
\[
\begin{split}
(\cL^{N_c}_t \mu_{\cG_{\iota}(k)}) (\phi) &= \sum_{r = 0}^{\infty} u_r \bigl(\mu_{\cG_{\iota,r}^u} (\phi \circ \cF^{N_c} \Id_{\cM_t^{N_c}}) + \mu_{\cG_{\iota,r}^u} (\cM \setminus \cM_t^{N_c}) (\cL_t^{N_c} \mu_{\cG_{\iota}(k)})(\phi)\bigr) \\
&+ \sum_{r = 0}^{\infty} c_r \mu_{\cG_{\iota,r}^c} (\phi \circ \cF^{N_c} \Id_{\cM_t^{N_c}}) +  (\cL_t^{N_c} \mu_{\cG_{\iota}(k)})(\phi) \sum_{r = 0}^{\infty} c_r \mu_{\cG_{\iota,r}^c} (\cM \setminus \cM_t^{N_c}),
\end{split}
\]
and, denoting by 
\[
A_{\iota,r} = \mu_{\cG_{\iota,r}^c} (\cM \setminus \cM_t^{N_c}) - \min_{\iota \in \{1,2\}} \mu_{\cG_{\iota,r}^c} (\cM \setminus \cM_t^{N_c}),
\]
we obtain
\[
\begin{split}
(\cL^{N_c}_t \mu_{\cG_{\iota}(k)}) (\phi) &= \sum_{r = 0}^{\infty} u_r \bigl(\mu_{\cG_{\iota,r}^u} (\phi \circ \cF^{N_c} \Id_{\cM_t^{N_c}}) + \mu_{\cG_{\iota,r}^u} (\cM \setminus \cM_t^{N_c}) (\cL_t^{N_c} \mu_{\cG_{\iota}(k)})(\phi)\bigr) \\
&+ \sum_{r = 0}^{\infty} c_r \bigl(\mu_{\cG_{\iota,r}^c} (\phi \circ \cF^{N_c} \Id_{\cM_t^{N_c}}) + A_{\iota,r}(\cL_t^{N_c} \mu_{\cG_{\iota}(k)})(\phi)  \bigr)\\
&+  (\cL_t^{N_c} \mu_{\cG_{\iota}(k)})(\phi) \sum_{r = 0}^{\infty} c_r \min_{\iota \in \{1,2\}} \mu_{\cG_{\iota,r}^c} (\cM \setminus \cM_t^{N_c})  .
\end{split}
\]
After an algebraic manipulation, we obtain the following expression for \((\cL_t^{N_c} \mu_{\cG_{\iota}(k)})(\phi)\),
\begin{equation}\label{eq:very-funny-decomposition}
   \begin{split}
    (\cL^{N_c}_t \mu_{\cG_{\iota}(k)})(\phi) &= \sum_{r=0}^{\infty} \frac{u_r}{1-\sum_{r = 0}^{\infty} c_r \min_{\iota \in \{1,2\}} \mu_{\cG_{\iota,r}^c} (\cM \setminus \cM_t^{N_c})} \\[10pt]
   & \hspace{3cm}\bigl( \mu_{\cG_{\iota,r}^u} (\phi \circ \cF^{N_c} \Id_{\cM_t^{N_c}})+ \mu_{\cG_{\iota,r}^u} (\cM \setminus \cM_t^{N_c}) (\cL_t^{N_c} \mu_{\cG_{\iota}(k)})(\phi)\bigr) \\[8pt]
&\!\!\!\!\!\!\!\!\!\!\!\!\!\!\!\!\!\!\!\!\!\!\!\!\!\!\!\!\!\!\!\!+ \sum_{r=0}^{\infty} \frac{c_r}{1-\sum_{r = 0}^{\infty} c_r \min_{\iota \in \{1,2\}} \mu_{\cG_{\iota,r}^c} (\cM \setminus \cM_t^{N_c})} \bigl(\mu_{\cG_{\iota,r}^c} (\phi \circ \cF^{N_c} \Id_{\cM_t^{N_c}}) + A_{\iota,r} (\cL_t^{N_c} \mu_{\cG_{\iota}(k)})(\phi)\bigr).
\end{split}
\end{equation}
Note that the expression above contains \( (\cL^{N_c}_t \mu_{\cG_{\iota}(k)})(\phi)\) also on the right hand side. The reason for this decomposition is the following. After conditioning on an irregular standard family (i.e., on \(\cG_{\iota,r}^u\) with \(r \ge 1\)), we would like to observe a regularizing effect of the dynamics. However, small pieces can intersect the hole, so that the surviving part is irremediably small, posing some problem. To obtain regularization we may argue in the following manner: if some mass enters the hole, it is as if an equivalent mass is generated in the form of the whole push-forward \(\cL_t^{N_c}\mu_{\cG_{\iota}(k)}\). Equation \eqref{eq:very-funny-decomposition} makes this fact manifest. By Proposition \ref{prop:invariance-standard-families} on invariance of standard families, which we already established, the measures \(\cL_t^{N_c}\mu_{\cG_{\iota}(k)}\) correspond to regular standard families. Therefore, when some mass enters the hole, it behaves as if it is automatically regularized. The net effect is that short standard families recover. We express this in the following Lemma.
\begin{lemma}\label{lem:growth-conditional}
Assume that \(\cG_{\iota}(k)\) admit a \((c_r, u_r)\) decomposition as in Definition \ref{def:ckuk}. There exists \(t_0 >0\) small enough such, that for all \(t \in [0, t_0]\) and \(r \ge 1\) we have
    \[
  \mu_{\cG_{\iota,r}^u} (\phi \circ \cF^{N_c} \Id_{\cM_t^{N_c}}) + \mu_{\cG_{\iota,r}^u} (\cM \setminus \cM_t^{N_c}) (\cL_t^{N_c} \mu_{\cG_\iota (k)})(\phi) = \mu_{\cG_{\iota, r-1}^u} (\phi),
    \]
for some standard family \(\cG_{\iota, r-1}^u \in \fR_{r-1}\).
\end{lemma}
\begin{proof}
Denote by \(\cG_{\iota,r}^u = \{(p_{j,\iota}, W_{j,\iota}, \rho_{j,\iota})\}\) and let \(W_{j,k,\iota}\) be the maximal connected subsets of \(W_{j,\iota} \setminus \cS_{N_c,t}^{\bH}\). Set also \(\hat \cF_t^{N_c} \cG_{\iota,r}^u = \{(p_{j,k,\iota},\hat{\cF}_t^{N_c} (W_{j,k,\iota}), \rho_{j,k,\iota})\}\) and let \(\cA_{\iota}\) be the collection of indices \((j,k)\) such that \(W_{j,k,\iota} \subset \cM \setminus \cM_t^{N_c}\). We have 
\begin{equation}\label{eq:mass-enter-hole}
 \mu_{\cG_{\iota,r}^u} (\cM \setminus \cM_t^{N_c}) = \sum_{(j,k) \in \cA_{\iota}} p_{j,k,\iota}.
\end{equation}
By Proposition \ref{prop:invariance-standard-families} about invariance of standard families and the fact that \(\cG_{\iota}\) are initial regular standard families, we have that \(\cL_t^{N_c} \mu_{\cG_\iota (k)}  = \cL_t^{(k+1)N_c} \mu_{\cG_{\iota}}= \mu_{\tilde \cG_{\iota}}\) for some regular standard families \(\tilde \cG_{\iota} = \{(\tilde p_{s,\iota}, \tilde \ell_{s,\iota})\}_s\). We set
\[
\cG_{\iota,r-1}^u =  \{(p_{j,k,\iota},\hat{\cF}_t^{N_c} (W_{j,k,\iota}), \rho_{j,k,\iota})\}_{(j,k) \notin \cA_{\iota}} \cup \{(\mu_{\cG_{\iota,r}^u} (\cM \setminus \cM_t^{N_c}) \tilde p_{s,\iota}, \tilde \ell_{s,\iota})\}_s.
\]
The standard families \(\cG_{\iota,r-1}^u\) yield the decomposition in the statement of the Lemma but to conclude the proof we need to bound the regularity of \(\cG_{\iota,r-1}^u\). Note that the densities supported by \(\cG_{\iota,r-1}^u\) either belong to the standard family \(\tilde \cG_{\iota}\) and so they are regular, or they are push-forwards of densities supported on an unstable curve in \(\cG_{\iota,r}^u\). By definition, these densities become regular after at most \(r\) iterates of \(\hat \cF_t^{N_c}\), and so the densities supported by \(\cG_{\iota,r-1}^u\) become regular after at most \(r-1\) iterates of \(\hat \cF_t^{N_c}\). As for the boundary, by the definition in the last displayed equation,
\[
\cZ(\cG_{\iota,r-1}^u) = \sum_{(j,k) \notin \cA_{\iota}} p_{j,k,\iota}  \frac{1}{|\hat \cF_t^{N_c} (W_{j,k,\iota})|} + \mu_{\cG_{\iota,r}^u} (\cM \setminus \cM_t^{N_c}) \cZ (\tilde \cG_{\iota}) .
\]
By \eqref{eq:mass-enter-hole} and  the fact that \(\cZ(\tilde \cG_{\iota}) \le B_2\) because \(\tilde \cG_{\iota}\) are regular standard families, we have
\begin{equation}\label{eq:support-recovery-hole}
\cZ(\cG_{\iota,r-1}^u) \le  \sum_{(j,k) \notin \cA_{\iota}} p_{j,k,\iota}  \frac{1}{|\hat \cF_t^{N_c} (W_{j,k,\iota})|} + B_2 \sum_{(j,k) \in \cA_{\iota}} p_{j,k,\iota}.
\end{equation}
Consider now an unstable curve \(W\) such that \(W \subset \cM \setminus \cM_{t}^{N_c} = \cup_{n = 0}^{N_c} \cF^{-n}(\cH_t)\). Then, by Lemma \ref{lem:neigh-sing} (which is independent on the results of this section), \(W \subset \bigl[ \cS_{N_c} \cup \cup_{n=0}^{N_c}\cF^{-n}(\partial \cH_t) \bigr ]_{h_{N_c}(t)}\) for some function \(h_{N_c}\) such that \(\lim_{t \to 0}h_{N_c} (t) = 0\). By transverslity, this implies that \(|W| \le C_{\mathrm{cone}}h_{N_c} (t)\). In particular, using any bound on the expansion of unstable curves (e.g., iterating Lemma \ref{lem:upp-bound-stretch-curves}), there exists \(t_0 >0\) small enough such that \(1/|\hat \cF_t^{N_c} (W)| \ge B_2\) for any  \(W \subset \cM \setminus \cM_{t}^{N_c}\) and \(t \in [0,t_0]\). Applying this estimate to the unstable curves \(W_{j,k,\iota}\subset \cM \setminus \cM_{t}^{N_c}\) with \((j,k) \in \cA_\iota\) and by \eqref{eq:support-recovery-hole}, we have, for any \(t \in [0,t_0]\),
\[
\cZ(\cG_{\iota,r-1}^u) \le \sum_{(j,k) \notin \cA_{\iota}} p_{j,k,\iota}  \frac{1}{|\hat \cF_t^{N_c} (W_{j,k,\iota})|} + \sum_{(j,k) \in \cA_{\iota}} p_{j,k,\iota} \frac{1}{|\hat \cF_t^{N_c} (W_{j,k, \iota})|} = \cZ(\hat \cF_t^{N_c} \cG_{\iota,r}^u).
\]
Hence, since \(\cG_{\iota,r}^u \in \fR_r\) and consequently \(\hat \cF_t^{N_c} \cG_{\iota,r}^u \in \fR_{r-1}\), we have that \(\cG_{\iota,r-1}^u \in \fR_{r-1}\) as well. This concludes the proof of the statement.
\end{proof}

The following linear scheme encodes all the results obtained so far on coupling. We use the notation \(M \bN = \{M, 2M, ...\}\).

\begin{lemma}\label{lem:linear-scheme1}
    Assume that \(\cG_{\iota}(k)\) admit a \((c_r, u_r)\)-decomposition as in Definition \ref{def:ckuk}. There exist \(C, p_c >0\), \(M \in \bN\) and \(t_0>0\) such that, for all \(t \in [0,t_0]\), the standard families \(\cL_t^{N_c}\cG_{\iota}(k)\) admit a \((c_r', u_r')\)-decomposition with
    \[
        \begin{split}
         & c_r ' \le (1+ C t)c_{r-1}\quad \text{if } r\ge 1,\\
        & c_0' \le (1+ C t)p_c u_0,\\
        & u_0 ' \le (1+ C t) ( u_1 + (1-p_c)u_0 ),\\
         &u_r ' \le \begin{cases}
        (1+ C t) u_{r+1}  & \text{if } r \notin M \bN\\
        (1+ C t)\bigl(u_{r+1} +  \eta_{0}^{\frac 1 4}c_{\frac r M -1}\chi^{\frac{1}{4}(\frac{r}{M} -1)} \bigr) &  \text{if } r \in M \bN.
    \end{cases}
    \end{split}
    \]
\end{lemma}
\begin{proof}
We decompose \(\mu_{\cL_t^{N_c} \cG_{\iota}(k)} = \cL_t^{N_c}\mu_{\cG_{\iota}(k)} \) as in \eqref{eq:very-funny-decomposition} and we study each term separately. For convenience, we set
\[
W_{t} = \frac{1}{1-\sum_{r = 0}^{\infty} c_r \min_{\iota \in \{1,2\}}\mu_{\cG_{\iota,r}^c} (\cM\setminus \cM_t^{N_c})}.
\]
By Lemma \ref{lem:growth-conditional}, for \(t_0>0\) small enough, there exist standard families \(\tilde \cG_{\iota, r-1}^{u}  \in \fR_{r-1}\) such that, for any \(t \in [0,t_0]\) and \(\phi \in \cC^0(\cM)\),
\[
\begin{split}
    \sum_{r=1}^{\infty} u_r \bigl ( \mu_{\cG_{\iota,r}^u} (\phi \circ \cF^{N_c} \Id_{\cM_t^{N_c}}) + \mu_{\cG_{\iota,r}^u} (\cM \setminus \cM_t^{N_c}) (\cL_t^{N_c} \mu_{\cG_{\iota}(k)})(\phi)\bigr)  &= \sum_{r=1}^{\infty} u_r   \mu_{\tilde \cG_{\iota,r-1}^u}(\phi) . 
\end{split}
\]
Moreover, by Lemma \ref{lem:fundamental-coupling} and our choice of \(N_c\), for \(t_0 >0\) small enough, there exist \(p_c, p_{\iota,u}>0\), two \(\eta_0\)-coupled regular standard pairs \(\ell_{\iota,c}\) and two regular standard families \(\tilde \cG_{\iota, 0}^u \) such that
\[
\begin{split}
    & u_0 \bigl(\mu_{\cG_{\iota,0}^u} (\phi \circ \cF^{N_c} \Id_{\cM_t^{N_c}}) + \mu_{\cG_{\iota,0}^u} (\cM \setminus \cM_t^{N_c}) (\cL_t^{N_c} \mu_{\cG_{\iota}(k)})(\phi)\bigr) \\
   &\hspace{2cm} = u_0 \bigl( p_c \mu_{\ell_{\iota,c}} (\phi) + p_{\iota,u}\mu_{\tilde \cG_{\iota,0}^u} (\phi) +  \mu_{\cG_{\iota,0}^u} (\cM \setminus \cM_t^{N_c}) (\cL_t^{N_c} \mu_{\cG_{\iota}(k)})(\phi)\bigr).
\end{split}
\]
By the invariance established in Proposition \ref{prop:invariance-standard-families} and the fact that \(\cG_{\iota}\) are initial regular, \(\cL_t^{N_c} \mu_{\cG_{\iota}(k)} = \cL_t^{(k+1)N_c}\mu_{\cG_{\iota}}\) are associated to regular standard families. Hence, there exist yet other two regular standard families, which we also denote by \(\tilde \cG_{\iota, 0}^u \), such that,
\[
u_0 \bigl( \mu_{\cG_{\iota,0}^u} (\phi \circ \cF^{N_c} \Id_{\cM_t^{N_c}}) + \mu_{\cG_{\iota,0}^u} (\cM \setminus \cM_t^{N_c}) (\cL_t^{N_c} \mu_{\cG_\iota(k)})(\phi)\bigr)= u_0 \bigl( p_c \mu_{\tilde \cG_{\iota, 0}^c}(\phi) + (1-p_c) \mu_{\tilde \cG_{\iota,0}^u}(\phi)\bigr).
\]
In the last equality we used the notation \( \tilde \cG_{\iota, 0}^c =\ell_{\iota,c}\) and the fact that the total mass on the l.h.s.\! has to coincide with the total mass on the r.h.s.\! Continuing with the last terms, recalling the definition of \(A_{\iota, r}\) and  the fact that \(\cL_t^{N_c} \mu_{\cG_\iota (k)}\) are associated to regular standard families (which we may denote to \(\cE_{\iota}\) to match notations), we can apply Lemma \ref{lem:one-step-loss-mass}. Hence, there exists \(C>0\) and \(t_0 >0\) such that, for each \(r \in \bN\) and \(t \in [0,t_0]\), there exist \(\Omega_{c,r}', \Omega_{u, r}' \in [0,1]\), \(\Omega_{u, r}' \le \eta_0^{1/4} \chi^{r/4}\), and \(\chi^{r+1}\eta_0\)-coupled regular standard families \(\tilde \cG_{\iota, r+1}^{c}\) and standard families \(\tilde \cG_{\iota, M (r+1)}^u \in \fR_{C\ln \eta_0^{-1} \chi^{-r}}\) such that
\begin{equation}\label{eq:tiring}
\begin{split}
\sum_{r=0}^{\infty} c_r \bigl (\mu_{\cG_{\iota,r}^c} (\phi \circ \cF^{N_c} \Id_{\cM_t^{N_c}}) &+  A_{\iota,r} (\cL_t^{N_c} \mu_{\cG_\iota (k)})(\phi) \bigr) \\
&= \sum_{r=0}^{\infty} c_r \bigl( \Omega_{c,r}' \mu_{\tilde \cG_{\iota, r+1}^{c}}(\phi) + \Omega_{u, r}'\mu_{\tilde \cG_{\iota, M (r + 1)}^u} (\phi) \bigr).
\end{split}
\end{equation}
Note that \(\fR_{C\ln (\eta_0^{-1} \chi^{-r}) } \subset \fR_{M(r+1)}\) for \(M \in \bN\) large enough, so that \(\tilde \cG_{\iota, M( r+1)}^u \in \fR_{M(r+1)}\), explaining the notation. Therefore, we can re-write the expression \eqref{eq:very-funny-decomposition} as
\[
\begin{split}
    \cL^{N_c}_t \mu_{\cG_\iota (k)} &= W_{t} \sum_{r=1}^{\infty} u_r \mu_{\tilde \cG_{\iota,r-1}^u} + W_{t} u_0 \bigl(p_c \mu_{\tilde \cG_{\iota, 0}^c} + (1-p_c) \mu_{\tilde \cG_{\iota,0}^u}\bigr) \\
    &+ W_{t}\sum_{r=0}^{\infty} c_r \bigl( \Omega_{c,r}'\mu_{\tilde \cG_{\iota, r+1}^{c}} + \Omega_{u, r}'\mu_{\tilde \cG_{\iota, M(r+1)}^u} \bigr).
\end{split}
\]
Relabeling the indices, we obtain,
\[
\begin{split}
    \cL^{N_c}_t \mu_{\cG_\iota (k)} &= W_{t}\mu_{\tilde \cG_{\iota,0}^u} \bigl((1-p_c) u_0  + u_1\bigr)  + W_{t} \sum_{r=1}^{\infty} u_{r+1} \mu_{\tilde \cG_{\iota,r}^u} + W_{t} u_0 p_c \mu_{\tilde \cG_{\iota, 0}^c} \\
    &+ W_{t}\sum_{r=1}^{\infty} c_{r-1}  \Omega_{c,r-1}' \mu_{\tilde \cG_{\iota, r}^{c}}(\phi) + W_{ t} \sum_{\substack{r=M \\ r \equiv 0 \ (\mathrm{mod}\ M)}}^{\infty} c_{\frac r M -1}\Omega_{u,\frac r M -1}'\mu_{\tilde \cG_{\iota, r}^u}\quad \text{and}\\ 
    &\hspace{3cm}\Omega_{u,\frac r M -1}' \le \eta_{0}^{\frac 1 4}\chi^{\frac{1}{4}(\frac{r}{M}-1)}.
\end{split}
\]
Once we take care of the common factor \(W_{t}\), we can compare the expression above with Definition \ref{def:ckuk} and obtain the desired decomposition.  By Definition \ref{def:ckuk}, \(\cG_{\iota,r}^c\) are regular standard families. Therefore, by Lemma \ref{lem:lost-mass-with-t}, there exists \(C_{N_c}>0\) such that \(\mu_{\cG_{\iota,r}^c} (\cM \setminus \cM_t^{N_c}) \le C_{N_c} t\) for both standard families. Therefore, for \(t_0>0\) small enough and because \(\sum_r c_r\) is less than one
\[
    |W_{t}| \le \frac{1}{1 - C_{N_c} \sum_{r = 0}^{\infty} c_r  t} \le \frac{1}{1-C_{N_c} t} \le 1 + 2C_{N_c} t.
\]
Here we can eliminate the subscript in \(C_{N_c}\) since \(N_c\) is already fixed depending on \(\eta_0\), which in turn depends on the table and the initial regularity (see the discussion at the beginning of Section \ref{subsec:evolution-two-coupled-sf}). This concludes the proof of the Lemma after renaming the constant.
\end{proof}

In order to prove decay of correlations with the scheme introduced above, we introduce the following norm. Recall \(M \in \bN\) from the statement of Lemma \ref{lem:linear-scheme1}. Let \(\psi_{-}<1\) and \(\psi_{+} >1\) such that
\[
\chi^{1/4} < (1/2)^{1/4} < \psi_{-} < 9/10, \qquad 1 < \psi_{+} < \biggl( \frac{\psi_-}{\chi^{1/4}}\biggr)^{\frac 1 M} \quad \text{and} \quad \psi_{+}^{M+1} \eta_{0}^{\frac 1 4} < \frac{1}{100}.
\]
The above relations hold choosing \(\psi_{+}\) sufficiently close to \(1\). For \(c_r,u_r \in [0,1]\), we set
\begin{equation}\label{eq:fancy-norm}
\|(c_r, u_r)\|_{*} = \sum_{r =0}^{\infty} u_r \psi_{+}^{r+1} + \sum_{r =0}^{\infty} c_r \psi_{-}^r .
\end{equation}

\begin{lemma}\label{lem:contraction-norm}
  Assume that \(\cG_{\iota} (k)\) admit a \((c_r, u_r)\)-decomposition as in Definition \ref{def:ckuk}. There exist \(\tau \in (0,1)\) and \(t_0>0\) such that, for all \(t \in [0,t_0]\), the standard families \(\cL_t^{N_c} \cG_{\iota}(k)\) admit a \((c_r', u_r')\)-decomposition with
    \[
    \|(c_r', u_r')\|_{*} \le \tau  \|(c_r, u_r)\|_{*}.
    \]
\end{lemma}
\begin{proof}
By Lemma \ref{lem:linear-scheme1}, \(\cL_t^{N_c} \cG_{\iota}(k)\) admit a \((c_r', u_r')\)-decomposition and,  for some \(C, p_c, M >0\),
    \[
    \begin{split}
        &\|(c_r', u_r')\|_{*} =  \sum_{r =0}^{\infty}  \psi_{+}^{r+1} u_r ' + \sum_{r =0}^{\infty}\psi_{-}^r  c_r ' \\
        &\le (1 + C t) \biggl (\sum_{r=1}^{\infty} \psi_{+}^{r+1}u_{r+1} + \eta_{0}^{\frac 1 4}\sum_{r = 0}^{\infty} \psi_{+}^{(r+1)M + 1}\chi^{\frac r 4} c_{r} + \psi_+ ( u_1 + (1-p_c)u_0 )+ p_c u_0 + \sum_{r=1}^{\infty} \psi_{-}^r c_{r-1} \biggr ).
    \end{split}
    \]
After renaming the indices, we obtain for the part in big brackets
\[
\begin{split}
    \psi_+^{-1} \sum_{r=1}^{\infty} &\psi_+^{r+1} u_r + \psi_+ u_0 ((1-p_c) + \psi_+^{-1}p_c ) + \sum_{r=0}^{\infty} \psi_{-}^r c_r (\psi_{-} + \eta_0^{\frac 1 4}\psi_{+}^{(r+1)M + 1} \psi_{-}^{-r}\chi^{\frac r 4}) \le  \tau  \|(c_r, u_r)\|_{*},
\end{split}
\]
where we have set
\[
\tau = \max \biggl\{(1-p_c) + \psi_{+}^{-1}p_c, \text{ }\sup_{r\ge 0}\bigl\{\psi_{-} + \eta_0^{\frac 1 4}\psi_{+}^{(r+1)M + 1} \psi_{-}^{-r}\chi^{\frac r 4}\bigr\}\biggr\}.
\]
By the conditions before \eqref{eq:fancy-norm}, we have that \(\psi_{+}^M \psi_{-}^{-1}\chi^{1/4} <1\) and consequently, since we also required that \(\eta_{0}^{1/4}\psi_{+}^{M+1} <1/100\), we have that \(\eta_{0}^{1/4}\psi_{+}^{M(r+1) + 1} \psi_{-}^{-r}\chi^{r/4} \le \eta_{0}^{1/4}\psi_{+}^{M+1} < 1/100\) for all \(r \in \bN\). Therefore, using also that \(\psi_{-}<9/10\) and that \(p_c\) is strictly positive,
\[
\tau \le \max \{(1-p_c) + \psi_{+}^{-1}p_c, \text{ } \psi_{-} + \eta_0^{\frac 1 4}\psi_{+}\} < 1.
\]
Thus, we have obtained that \(\|(c_r', u_r')\|_{*} \le (1+ Ct) \tau \|(c_r, u_r)\|_{*}\) and the Lemma follows by taking \(t_0\) small enough and renaming \(\tau\).
\end{proof}

We now show how the above implies exponential decay of correlations.

\begin{proof}[Proof of Theorem \ref{thm:exp-mixing-sf}]
Let \(\cG_{\iota}\) be two initial regular standard families and \(t_0\) be small enough so that Lemma \ref{lem:contraction-norm} applies. For any \(n \in \bN\), write \(n = kN_c + n_0\) with \(n_0 < N_c\). By Proposition \ref{prop:invariance-standard-families} about invariance, the measures \(\cL_t^{n_0}\mu_{\cG_\iota}\) are associated to the regular standard families \(\cL_t^{n_0}\cG_{\iota} = \cG_{\iota}(n_0)\). Therefore, they admit a \((c_r,u_r)\)-decomposition with \(u_0 = 1\) and \(u_r = c_r = c_0 =0\) for \(r\) positive. Applying repeatedly Lemma \ref{lem:linear-scheme1}, we obtain that for any \(p \in \bN\) the measures \(\cL_t^{k N_c} (\cL_t^{n_0}\mu_{\cG_\iota})\) admit a \((c_r,u_r)\)-decomposition. Moreover, denoting these decompositions by \((c_r (k), u_r(k))\), we have, for any \(\phi \in \cC^{1}(\cM)\),
    \[
    |\cL_t^{n} \mu_{\cG_1} (\phi) - \cL_t^n \mu_{\cG_2} (\phi) | \le  \sum_{r=0}^{\infty} u_r(k) |\mu_{\cG_{1,r}^u} (\phi) - \mu_{\cG_{2,r}^u} (\phi) | + \sum_{r=0}^{\infty} c_r(k) |\mu_{\cG_{1,r}^c} (\phi) - \mu_{\cG_{2,r}^c} (\phi) |,
    \]
for some standard families \(\cG_{\iota,r}^{c/u}\) as in Definition \ref{def:ckuk}. Since \(\cG_{\iota,r}^c\) are \(\chi^r\eta_0\)-coupled, the quantity above is bounded by
\[
2\|\phi\|_{\cC^0}\sum_{r=0}^{\infty} u_r(k) + \eta_0 \|\phi\|_{\cC^1} \sum_{r=0}^{\infty}c_r(k)\chi^r \le 2 \|\phi\|_{\cC^1} \|(c_r(k), u_r(k))\|_{*}, 
\]
where in the inequality we used that \(\chi \le \chi^{1/4}  \le \psi_{-}\) and \(\psi_{+} \ge 1\). Therefore, using Lemma \ref{lem:contraction-norm} repeatedly and recalling that \(k = \lfloor n/N_c\rfloor\), we have, for some \(\tau \in (0,1)\),
\[
 |\cL_t^{n} \mu_{\cG_1} (\phi) - \cL_t^n \mu_{\cG_2} (\phi) | \le 2\|\phi\|_{\cC^1} \tau^{k} \|(c_r(0), u_r(0))\|_{*} = 2\psi_{+}\|\phi\|_{\cC^1} \tau^{\lfloor \frac n N_c \rfloor },
\]
where in the last equality we have used that \(c_{r}(0) = 0\), \(u_{0}(0) = 1\) and \(u_r(0) = 0\) for \(r \ge 1\). By setting \(\gamma = \tau^{1/N_c}\) in the equation above, we conclude the proof of the theorem. 
\end{proof}

\newpage

\bibliographystyle{alpha}
\bibliography{modernity.bib}

\end{document}